\documentclass[twoside,linksfromyear]{article}

\usepackage{float,graphicx}
\usepackage{amsmath,amsfonts,amsthm}
\RequirePackage[colorlinks,citecolor=blue,urlcolor=blue]{hyperref}
\usepackage{enumerate}
\usepackage{subfigure}
\usepackage{changebar}
\usepackage{authblk}
\usepackage{xr}
\usepackage{natbib}

\nochangebars

\theoremstyle{definition}
\newtheorem{definition}{Definition}[section]
\newtheorem{theorem}{Theorem}
\newtheorem{lemma}{Lemma}
\newtheorem{corollary}{Corollary}
\theoremstyle{plain}

\newtheorem{proposition}[definition]{Proposition}

\newcommand{\N}{{\mathbb N}}
\newcommand{\R}{{\mathbb R}}
\newcommand{\calO}{{\mathcal O}}

\def\E{{\mathbb E}}
\def\R{{\mathbb R}}
\def\C{{\mathbb C}}
\def\N{{\mathbb N}}

\def\ZZ{{\mathbb Z}}

\def\ph{\varphi}

\def\av{{\arrowvert}}
\def\Av{{\Arrowvert}}

\def\bphi{\bm{\varphi}}
\def\bgamma{\bm{\gamma}}

\def\beq{\begin{equation}}
\def\eeq{\end{equation}}
\def\beqa{\begin{eqnarray}}
\def\eeqa{\end{eqnarray}}
\def\beqann{\begin{eqnarray*}}
\def\eeqann{\end{eqnarray*}}

\newcommand{\calI}{{\mathcal I}}

\newcommand{\Var}{\operatorname{var}}
\newcommand{\Cov}{\operatorname{cov}}

\newcommand{\gpncor}{\operatorname{corr}}
\newcommand{\cum}{\operatorname{cum}}

\newcommand{\mmspe}{\operatorname{MSPE}}
\newcommand{\ess}{\operatorname{ess}}

\newcommand{\TV}{\operatorname{TV}}
\newcommand{\bv}{\operatorname{bv}}

\newcommand{\reals}{{\mathbb R}}
\newcommand{\nats}{{\mathbb N}}
\newcommand{\ints}{{\mathbb Z}}

\newcommand{\cov}[1]{\, {\rm cov}\left( #1 \right) }
\newcommand{\cor}[1]{\, {\rm corr}\left( #1 \right) }

\newcommand{\bm}[1]{\mbox{\protect\boldmath $ #1 $}}

\newcommand{\p}[2]{\, q\left( #1, #2 \right) }
\newcommand{\pT}[2]{\, q_T\left( #1, #2 \right) }
\newcommand{\ptilde}[2]{\, \tilde{q} \left( #1, #2 \right) }
\newcommand{\ptildeW}[2]{\, \tilde{q}_W \left( #1, #2 \right) }

\newcommand{\mspeA}[2]{\, \mbox{MSPE}\left( \hat{#1}_{#2},#1_{#2}\right) }

\newcommand{\mspeeA}[2]{\, \E\left( \hat{#1}_{#2} - #1_{#2}\right)^2 }

\newcommand{\mspeebA}[2]{\, \mbox{MSPE}\left( \hat{#1}_{#2}^{(b)},#1_{#2}\right) }

\newcommand{\mspeefA}[2]{\, \mbox{MSPE}\left( \hat{#1}_{#2}^{(f)},#1_{#2}\right) }

\newcommand{\rhs}{\, \bm{r} }
\newcommand{\rhsT}{\, \bm{r}_{[zT];T} }

\newcommand{\bb}{\bm{b}_{[zT]}^{(b)}}
\newcommand{\bbtilde}{\, \tilde{\bm{b}}_{[zT]}^{(b)} }
\renewcommand{\bf}{\bm{b}_{[zT]+\tau}^{(f)}}
\newcommand{\bftilde}{\, \tilde{\bm{b}}_{[zT]+\tau}^{(f)} }

\newcommand{\Bb}{{\rm B}_{[zT]}^{(b)} }
\newcommand{\Bbtilde}{\, \tilde{{\rm B}}_{[zT]}^{(b)} }
\newcommand{\Bf}{{\rm B}_{[zT]+\tau}^{(f)}}
\newcommand{\Bftilde}{\, \tilde{{\rm B}}_{[zT]+\tau}^{(f)}}
\newcommand{\BT}{{\rm B}_{[zT]}}
\newcommand{\sigb}{\Sigma_{[zT];T}^{(b)}}
\newcommand{\sigf}{\Sigma_{[zT]+\tau;T}^{(f)}}
\newcommand{\sig}{\Sigma_{[zT];T}}

\newcommand{\phis}{\ph_{[zT],\tau,\tau}}
\newcommand{\phisT}{\ph_{[zT],\tau,\tau;T}}
\newcommand{\phistilde}{\tilde{\ph}_{[zT],\tau,\tau}}

\newcommand{\PhisT}{\bm{\ph}_{[zT],\tau;T}}
\newcommand{\PPhis}{\bm{\ph}_{[zT],\tau}}
\newcommand{\Phistilde}{\tilde{\bm{\ph}}}

\newcommand{\otone}{o_T(1)}
\newcommand{\OTinv}{\calO(T^{-1})}

\newcommand{\bcb}{\begin{changebar}}
\newcommand{\ecb}{\end{changebar}}

\title{The Local Partial Autocorrelation Function and Some Applications}

\author[1]{Rebecca Killick\thanks{r.killick@lancs.ac.uk}}
\author[2]{Marina I.\ Knight\thanks{marina.knight@york.ac.uk}}
\author[3]{Guy P.\ Nason\thanks{g.nason@imperial.ac.uk}}
\author[1]{Idris A.\ Eckley\thanks{i.eckley@lancs.ac.uk}}
\affil[1]{Dept.\ of Mathematics and Statistics, Lancaster University, Lancaster, LA1 4YF, UK}
\affil[2]{Dept.\ of Mathematics, University of York, York, YO10 5DD, UK}
\affil[3]{Dept.\ of Mathematics, Imperial College, London, SW7 2AZ, UK}

\begin{document}
\bibliographystyle{guy3}
\maketitle

\begin{abstract}
The classical regular and partial autocorrelation functions are powerful tools for stationary time series modelling and analysis. However, it is increasingly recognized that many time series are not stationary and the use of  classical
global autocorrelations can give misleading answers. This article introduces two estimators of the local partial autocorrelation function and establishes their asymptotic properties.
The article then illustrates the use of these new estimators on both simulated and
real time series. The examples clearly demonstrate the strong practical
benefits of local estimators for time series that exhibit nonstationarities.

\end{abstract}


\section{Introduction}\label{sec:intro}

Much work has been undertaken to develop both theory and methods for the use of the
autocorrelation and partial autocorrelation for mean zero second-order stationary
time series. See, for example, \cite{priestley82:spectral},
\cite{brockwell91:time} or \cite{chatfield03:the}. For stationary time series, both autocorrelations
are fundamental for eliciting second-order structure and are particularly useful for subsequent modelling and prediction. Unfortunately,
in many applied situations,
for example neurophysiology \citep{fiecas16:modeling} or biology \citep{hargreaves17:clustering}, the stationarity assumption is not tenable and,
hence, use of the classical stationary-based autocorrelations is highly questionable.
\begin{changebar} Indeed, it is not possible for a 
time-varying parameter to be adequately summarised
by a single coefficient.
\end{changebar}
Before practical analysis, one should therefore attempt to assess whether the series is stationary or not.
Many techniques and software packages exist that enable such assessment, see reviews
in~\cite{dahlhaus12} or \cite{CardinaliNason18} or newer techniques that
measure, rather than test, \begin{changebar} the \end{changebar} degree of nonstationarity, e.g.\ \cite{DN16}.

A large literature on nonstationary time series modelling has developed since the
1950s. See, for example,
\cite{page:instantaneous}, \cite{silverman:locally}, \cite{Whittle63},
\cite{priestley:evolutionary}, \cite{tong74:on} and \cite{dahlhaus97:fitting}.
Alternative model forms including the piecewise stationary time series of \cite{adak98:time}; the wavelet models of \cite{nason00:wavelet};
and the SLEX models of \cite{ombao02:the} have been proposed. A comprehensive
review
of locally stationary series can be found in \cite{dahlhaus12}.
As part of these developments,
the local autocovariance, for non- or locally stationary processes, has been studied in the literature and details on specific estimators can be found in \cite{HW97}, \cite{nason13}, \cite{cardinali14} and \cite{zhao15}, for example.
However, to date, little attention seems to have been paid to  \emph{local} partial autocorrelation and the benefits it could bring. An exception is
\cite{degerine96:evolutive} and \cite{degerine03:characterization},  who
extended the classical partial autocorrelation to encompass nonstationary processes.
Their seminal work mentions estimation, including the windowing idea that we use
in Section~\ref{sec:3},
but provides no theory for their estimator nor evaluation via simulation
or on real time series.
More recently,
\cite{yang16:bayesian}
use a hierarchical Bayesian modelling approach to estimate process time-frequency structure,
linking the time-dependent partial autocorrelations to the coefficients of a time-varying autoregressive process.

\begin{changebar} 
Autocorrelation and partial autocorrelation are intimately related,
presenting complementary views on the underlying
structure within a time series. For example, arguably,  partial autocorrelation provides direct information on the order
and underlying structure of autoregressive-type processes (see Appendix~\ref{sec:pacf} for additional background on its interpretation). 
As in the stationary case, for real-life statistical analysis one needs both local autocorrelation {\em and} partial autocorrelation.
This article fills the gap for the latter.
We introduce two new estimators of the local partial autocorrelation function,
supplying new results on their theoretical properties. We further exhibit our estimators on a simulated
series and three real time series that demonstrate the importance of using a local approach.
In addition, our work also provides a freeware R software package, {\tt lpacf}, for local partial autocorrelation that complements existing
software for
local autocorrelations, such as {\tt lacf} in the {\tt locits} package.

\end{changebar}


\section{The Local Partial Autocorrelation Function}\label{sec:2}

\subsection{The (process) local partial autocorrelation function, $q_T$, for a locally stationary process}\label{sec:lpacf:process}

Let $\{ X_{t,T} \}$ be a zero-mean locally stationary process such as the  locally stationary Fourier process, \citet[Definition~2.1]{dahlhaus97:fitting}, or the
locally stationary wavelet process, \citet[Definition~1]{nason00:wavelet} (for ease of reference, these definitions can also be found in Appendix \ref{sec:defls}).
Locally stationary process theory \bcb supports short-memory processes \ecb and often has quantities of interest such as
the time-varying spectrum, $f (z, \omega)$ at (rescaled) time $z \in (0,1)$ and frequency $\omega$, or
local autocovariance $c(z, \tau)$ at location $z$ and lag $\tau$, which are estimated via a process quantity ($f_T$ or $c_T$), which depends on the
sample size $T$ and asymptotically approaches to the quantity of interest as $T \rightarrow \infty$.
\begin{changebar}
Consider, for example, $f_T(z, \omega)$ from \cite{NeumannvonSachs97} 
or $c_T(z, \tau)$ from \cite{nason00:wavelet}.\end{changebar}
We follow this paradigm by first introducing the process local partial autocorrelation, $q_T$.

The (process) partial autocorrelation function, $q_T (z, \tau)$, of a zero-mean {\em locally} stationary process can be understood informally as
\begin{displaymath}
\pT{z}{\tau}= \cor{X_{[zT],T},X_{[zT]+\tau,T}|\text{``in-between'' data}},
\end{displaymath}
where $[x]$ denotes the integer part of the real number $x$. A formal definition follows.
\begin{definition}
The local process partial autocorrelation of a zero-mean locally stationary process
$\{ X_{t, T}\}_{t=0}^{T-1}$,
 at rescaled time $z\in(0,1)$ and lag $\tau$, is defined by
\begin{equation}
\label{eq:deflpacf}
  \pT{z}{\tau}= \gpncor \left\{ X_{{[zT]+\tau},T}-P_{[zT],\tau}(X_{{[zT]+\tau},T}),X_{[zT],T}-P_{[zT],\tau}(X_{{[zT]},T}) \right\},
\end{equation}
where $P_{[zT],\tau}(\cdotp)$ is the projection operator onto
$\overline{\mbox{sp}}(X_{[zT]+1,T}, \ldots, X_{[zT]+\tau-1,T})$.
Here $\overline{\mbox{sp}}$ is the closed span defined by
\cite{brockwell91:time}.
\end{definition}
\noindent The next proposition shows an alternative useful representation of $q_T$.
\begin{proposition}\label{prop:lpacf}
Let $\{X_{t, T}\}$ be a zero-mean locally stationary process. Then the process local
partial autocorrelation, $q_T$, can be expressed as
\beq\label{eq:deflpacf2}
\pT{z}{\tau}= \phisT \left[ \frac
{\operatorname{Var} \{ X_{[zT],T}-P_{[zT],\tau}(X_{{[zT]},T}) \} }
{\operatorname{Var} \{ X_{{[zT]+\tau},T}-P_{[zT],\tau}(X_{{[zT]+\tau},T}) \}} \right]^{1/2},
\eeq
where $\phisT$ is from projecting $X_{[zT]+\tau,T}$ onto
$\overline{\mbox{sp}} ( X_{[zT],T},\ldots,X_{[zT]+\tau-1,T})$, the projection being
\beq\label{eq:calclpacf}
\hat{X}_{[zT]+\tau,T}=\ph_{[zT],\tau,1;T}X_{[zT]+\tau-1,T}+\ldots+\ph_{[zT],\tau,\tau;T}X_{[zT],T}.
\eeq
\end{proposition}
\begin{proof}
  See Section~\ref{app:prop:lpacf}.
\end{proof}

Formulae~\eqref{eq:deflpacf} and~\eqref{eq:deflpacf2} are natural generalisations of their
stationary equivalents, compare for example with Definitions~3.4.1 and~3.4.2 from
\citet{brockwell91:time}.

\subsection{Equivalent expressions for the process local partial autocorrelation
function, $q_T$} \label{sec:lpacf:equivalent}

As a step to estimation, we will express $q_T$  by exploiting a well-known connection between partial autocorrelation and linear prediction.
We introduce the following notation $P_{[zT],\tau}(X_{{[zT]},T})=\hat{X}^{(b)}_{[zT],T}$ and
$P_{[zT],\tau}(X_{{[zT]+\tau},T})=\hat{X}^{(f)}_{[zT]+\tau,T}$. These are simply the respective linear predictors of $X_{{[zT]},T}$ ({\em b}ack-casted), and $X_{{[zT]+\tau},T}$ ({\em f}orecasted), using the predictor set  $X_{[zT]+1,T}, \ldots, X_{[zT]+\tau-1,T}$.
The numerator and denominator in~\eqref{eq:deflpacf2} can  be re-expressed as a Mean Squared Prediction Error (MSPE). Consequently, we can rewrite $\pT{z}{\tau}$ as
\beq\label{eq:deflpacf3}
\pT{z}{\tau}= \phisT \left\{ \frac
{\mbox{MSPE}(\hat{X}^{(b)}_{[zT],T},X_{{[zT]},T})}
{\mbox{MSPE}(\hat{X}^{(f)}_{[zT]+\tau,T},X_{{[zT]+\tau},T})} \right\}^{1/2}. 
\eeq
For details see Section~\ref{app:prop:lpacf2}.
For stationary processes the square root term in~\eqref{eq:deflpacf3} equals one and
$\pT{z}{\tau}$ coincides with the classical $q(\tau)$.

In general, given $t$ observations of a zero-mean locally stationary process, $X_{0,T},\ldots,X_{t-1,T}$, the mean squared prediction error of a linear predictor of $X_{t,T}$,
$\hat{X}_{t,T}=\sum_{s=0}^{t-1}b_{t-1-s,T}X_{s,T}$, can be written as
\beq\nonumber
\mspeA{X}{{t,T}}=\bm{b}_t^T \Sigma_{t;T} \bm{b}_{t},
\eeq
where $\bm{b}_t = ( b_{t-1,T},  \ldots, b_{0,T}, -1)^T$ and $\Sigma_{t,T}$ is the covariance of $X_{0,T}, \ldots, X_{t,T}$, see, e.g., \citet[Section 3.3]{fryz03:forecasting}.
In our case, the back-casted and forecasted values of $X_{{[zT]},T}$ and $X_{{[zT]+\tau},T}$
are also linear predictors using the window of observations $X_{[zT]+1,T}, \ldots, X_{[zT]+\tau-1,T}$, and can be expressed as
\beq\nonumber
\hat{X}_{[zT],T}^{(b)} =\sum_{p=1}^{\tau-1}
b_{{p},T}^{(b)}X_{{[zT]+p},T} \mbox{\ \ and\ \ } \hat{X}_{[zT]+\tau,T}^{(f)}=\sum_{p=1}^{\tau-1} b_{{\tau-1-p},T}^{(f)}X_{{[zT]+p},T},
\eeq
respectively.
Here, the $\bm{b}^{(b)}$, $\bm{b}^{(f)}$ coefficient vectors are obtained through minimisation of
the corresponding mean squared prediction error
using the same principle as in the stationary case.

We next give a proposition that paves the way towards a natural definition of the local partial autocorrelation function $q$ in Section~\ref{sec:lpacf:lpaf}.
\begin{proposition}\label{prop:lpacf2}
Let $\{X_{t,T}\}$ be a zero-mean locally stationary process. Then $q_T$ can also be expressed as
\begin{equation}
\label{eq:qTexpress}
  \pT{z}{\tau} = \phisT \left\{ \frac{(\bb)^T\sigb\bb}{(\bf)^T\sigf\bf} \right\}^{1/2},
 \end{equation}
where $\phisT$ is as in~\eqref{eq:calclpacf}, and $\bb= ( -1,  \tilde{b}^{(b)}_{1,T}, \ldots,  \tilde{b}^{(b)}_{{\tau-1},T})^T$ and
$\bf= ( {b}^{(f)}_{\tau-2,T}, \ldots, {b}^{(f)}_{0,T},  -1 )^T$ are $\tau \times 1$ coefficient vectors. To simplify notation we have suppressed the dependency of the $\bm{b}$-vector components on
$[zT]$ and also dependency of $\bb$, $\sigb$, $\bf$, $\sigf$ on $\tau$, even though it is still present. The $\tau \times \tau$ covariance matrices $\sigb$ and $\sigf$ are given in Appendix~\ref{sec:covmat}.
\end{proposition}
\begin{proof}
  See Section~\ref{app:prop:lpacf2}.
\end{proof}

We will use  expression~\eqref{eq:qTexpress}
as the basis of an estimator in Section \ref{sec:lpacf:est}. The last element of the vector $\PhisT$, denoted $\phisT$, can be obtained as the solution to the (local) Yule-Walker equations $\sig \PhisT=\rhsT$, where $\sig$ is the $\tau \times \tau$ covariance matrix
given in Appendix~\ref{sec:covmat} and $\rhsT$ is the $\tau \times 1$ covariance vector of
$X_{[zT]+\tau,T}$ with $\left(
X_{{[zT]+\tau-1},T},\ldots,X_{{[zT]},T} \right)$. This is equivalent to obtaining a solution
$\hat{X}_{[zT]+\tau,T}$ that achieves minimum
mean squared prediction error over the class of linear predictors.
For stationary processes the covariance matrix $\Gamma_\tau$ is Toeplitz. However, for locally stationary
processes the covariance matrix $\sig$ only has an approximate
Toeplitz structure. Once again, for ease of notation, we have suppressed the dependency on the lag
$\tau$ from the
vector $\rhsT$ and
covariance matrix $\sig$, the latter given in Appendix~\ref{sec:covmat}.

\subsection{The wavelet local partial autocorrelation function, $q$}
\label{sec:lpacf:lpaf}
The local (process) partial autocorrelation introduced
in Section~\ref{sec:lpacf:process} can be applied to any zero-mean locally
stationary process. However, for the theory we develop below, we need to establish the underlying asymptotic quantity, which is intimately related to the data generating model. Hence, from now on, we  assume that the
process $\{X_{t,T}\}$ is a zero-mean locally stationary wavelet process and
define the local partial autocorrelation
function, $q$,  which  we show later to be the asymptotic limit of $q_T$ from~\eqref{eq:deflpacf2}.
\begin{definition}
\label{def:lpacf:lpacf}
	Let $\{X_{t,T}\}$ be a zero-mean locally stationary wavelet process as defined in
\cite{fryz03:forecasting} with local autocovariance $c(z, \tau)$ and spectrum $\{ S_{j}(z)\}_j$ that satisfy
\begin{align}
\sum_{\tau = 0}^{\infty} \sup_z | c(z, \tau)| < \infty, \nonumber  
C_1 := \ess \inf_{z,\omega} \sum_{j > 0} S_j(z) | \hat{\psi}_j (\omega) |^2 > 0, \nonumber 
\end{align}
where $\hat{\psi}_j (\omega) = \sum_s \psi_{j, 0} (s) \exp(i \omega s)$.  Then, the local partial
autocorrelation function is
	\begin{equation}\label{deflpacf}
		\p{z}{\tau} =\phis \left\{ \frac{(\bb)^T\Bb\bb}{(\bf)^T\Bf\bf} \right\}^{1/2},
	\end{equation}
where
\begin{enumerate}
\item the quantity $\phis$ is the last element in the vector $\PPhis$ (of length $\tau$) obtained as
the solution to the local Yule-Walker equations i.e.\ $\BT \PPhis=\rhs_{[zT]}$,

\item the matrices $\Bf$ and $\Bb$ are  the local approximations of $\sigf$ and
$\sigb$, as  in the proof of Lemma A.1 from \cite{fryz03:forecasting}. The
$\rhs_{[zT]}$ are also local approximations to $\rhsT$ from
Section~\ref{sec:lpacf:equivalent}
but using $c(z, \tau)$.

\item the coefficient vectors $\bf$ and $\bb$ are obtained as the solution to the
forecasting and back-casting prediction equations, or equivalently through minimisation
of the $\mmspe$. See Section~3.1 and Proposition~3.1 from \cite{fryz03:forecasting} for details.
\end{enumerate}
\end{definition}
Next, Proposition~\ref{prop:pTp} shows that the (process) local partial autocorrelation,
$q_T$, converges to the local partial autocorrelation, $q$, defined by \eqref{deflpacf}.
\begin{proposition}\label{prop:pTp}
Let $\{ X_{t,T} \}$ be a zero-mean locally stationary wavelet process as defined by Definition \ref{def:lpacf:lpacf}, with
spectrum $\{ S_j(z) \}_{j=1}^{\infty}$ constructed with nondecimated discrete wavelet system $\{
\psi_{j, k}(t)\}$.
Let local partial autocorrelations $q_T$ and $q$ be defined as in~\eqref{eq:deflpacf3}
and~\eqref{deflpacf} respectively.
Then, as $T\to\infty$, uniformly in $\tau\in\ZZ$ and $z\in(0,1)$, we have
$|\pT{z}{\tau}-\p{z}{\tau}|=\OTinv$.
\end{proposition}
\begin{proof}
See Section~\ref{app:prop:pTp}.
\end{proof}

This result parallels the local autocovariance result of \cite{nason00:wavelet}, where it is shown that $| c_T(z, \tau) - c(z, \tau) | = \OTinv$ as $T \rightarrow \infty$
uniformly in $\tau \in \ZZ$ and $z \in (0,1)$.

\subsection{Wavelet local partial autocorrelation estimation}\label{sec:lpacf:est}
We now  consider the important problem of  local partial autocorrelation estimation.  We begin by first noting that
all the quantities on the right-hand side of \eqref{deflpacf} for $q(z, \tau)$ are based on
the local autocovariance $c(z, \tau)$. A natural estimator of $q$ can thus be
obtained by replacing all  occurrences of  $c(z, \tau)$ by the  wavelet-based estimator
$\hat{c}(z, \tau)$ from \citet[Section 3.3]{nason13} as follows.
\begin{definition}
\label{def:wbe}
The wavelet-based local partial autocorrelation  estimator is defined as
\beq\label{eq:estq}
\ptilde{z}{\tau}=\phistilde  \left\{ \frac{(\bbtilde)^T\Bbtilde\bbtilde}{(\bftilde)^T\Bftilde\bftilde} \right\}^{1/2},
\eeq
\begin{changebar}
where  the  matrix estimates, $\Bbtilde$, $\Bftilde$, and vector estimates $\bbtilde$, $\bftilde$,
\end{changebar}
are obtained from their population quantities in Sections~\ref{sec:lpacf:equivalent} and~\ref{sec:lpacf:lpaf} by
plugging in the wavelet-based local autocovariance estimator $\hat{c}$
from \cite{nason13}.  Similarly, the vector $\Phistilde_{[zT], \tau}$ is
obtained as the solution to the local Yule-Walker equations in
Definition~\ref{def:lpacf:lpacf} again replacing $c$ by $\hat{c}$.
\end{definition}
We next
establish the consistency of $\tilde{q}$ for $q$.
\begin{proposition}\label{prop:ptildep}
Let $\{ X_{t,T} \}$ be a zero-mean locally stationary wavelet process under the assumptions given in
Definition~\ref{def:lpacf:lpacf}.
The local partial autocorrelation estimator $\ptilde{z}{\tau}$ from~\eqref{eq:estq} is consistent
for the true local partial autocorrelation $\p{z}{\tau}$, in that $\ptilde{z}{\tau}-\p{z}{\tau}=o_{p}(1)$ as $T \rightarrow \infty$.
\end{proposition}
\begin{proof}
  See Section~\ref{app:prop:ptildep}.
\end{proof}
Our wavelet-based estimator, $\ptilde{z}{\tau}$, develops earlier
work on forecasting by \cite{Fryzlewicz03} in a new direction.
However, the estimator is not simple to implement and,
as we will see later, does not perform  as well as the following
alternative approach, which applies a window to the classical partial autocorrelation.

\section{Windowed Estimation of Local Partial Autocorrelation}\label{sec:3}

\subsection{The integrated local wavelet periodogram}

We introduce an alternative estimator, $\ptildeW{z}{\tau}$,
 that is simpler to implement than $\ptilde{z}{\tau}$, and turns out to perform better. This new estimator is constructed by windowing the classical partial autocorrelation (designed for stationary processes)
 over an interval centred  at time $[zT]$ with  length $L(T)$, where $L(T)
\rightarrow \infty$ and $L(T)/T \rightarrow 0$, as $T\rightarrow\infty$. Proposition~\ref{coro:pwindow}, in Section \ref{sec:windowlpacf}, establishes the asymptotic behaviour of $\ptildeW{z}{\tau}$ by approximating
the integrated local wavelet periodogram of a (zero-mean) locally stationary wavelet process by its equivalent stationary version at a fixed rescaled time (see Theorem~\ref{thm:JNphi}).
  The proof of the theorem introduces new bounds for quantities involving cross-correlation wavelets, as well as a new exact formula for cross-correlation Haar wavelets.
Key definitions and results are presented below, while full proofs are provided in Section~\ref{app:sec3proofs}.
\begin{definition}
\label{def:ILWP}
 Let $\{X_{t,T}\}$ be a locally stationary wavelet process as in
 Definition~1 from \cite{nason00:wavelet}
 with evolutionary wavelet spectrum $\{ S_j (z) \}_{j=1}^\infty$ for $z\in (0,1)$,
 Lipschitz constants $\{ L_j \}_{j=1}^\infty$, process constants $\{ C_j \}_{j=1}^\infty$
 and underlying discrete nondecimated wavelets $\{ \psi_{j, k} \}$.
 The integrated local periodogram on
the interval $\left[ [zT]-{L(T)}/2+1,[zT]+{L(T)}/2 \right]$ is given by
\begin{equation}\nonumber
J_{L(T)}(z,\phi) = \sum_{j=1}^{\infty} \phi_j I^{\ast}_{L(T)}(z, j).
\end{equation}
Here $\{ \phi_j \}_{j=1}^{\infty} \in \Phi$ and $\Phi$ is a set of complex-valued bounded sequences equipped with uniform norm $|| \phi ||_{\infty} := \sup_j | \phi_j |$, $z \in (0, 1)$ and, for $j \in \mathbb{N}$,
$I^{\ast}_{L(T)} (z, j)$ is the uncorrected, tapered local wavelet periodogram given by
\beq\label{eq:uncortapper}
I^{\ast}_{L(T)} (z, j) = H_{L(T)}^{-1} \left| \sum_{t=0}^{{L(T)}-1} h \left\{
t/{L(T)} \right\} X_{[zT]+ t - {L(T)}/2  +1, T}\, \psi_{j, [zT]}(t) \right|^2,~\mbox{for $j\in \mathbb{N}$},
\eeq
with  
$h: [0,1] \rightarrow {\mathbb R}_{+}$
and normalizing factor
$H_{L(T)} := \sum_{j=0}^{{L(T)}-1} h^2 \{ j/{L(T)}\} \sim {L(T)} \int_{0}^1 h^2(x) \,
dx$.
\end{definition}
We next approximate the integrated local (wavelet) periodogram, $J_{L(T)} (z,\phi)$, by the corresponding statistics of a stationary process $\{Y_t\}$ with the same local corresponding statistics at $t=zT$, for fixed $z$.  Conceptually, this is a common approach useful in establishing asymptotic properties for functions of locally stationary processes \citep{dahlhaus98:on}, which in this work we advance to include {\em wavelet}-based expansions. Specifically, define

\begin{equation}\nonumber
J^Y_{L(T)} (\phi) = \sum_{j=1}^{\infty} \phi_j I^{\ast,Y}_{{L(T)}}(j),
\end{equation}
where
\begin{equation}\nonumber
I^{\ast,Y}_{{L(T)}} (j) := H_{L(T)}^{-1} \left| \sum_{s=0}^{{L(T)}-1} h \left\{ s/L(T)
\right\} Y_{ [zT] - {L(T)}/2  +1+s, T}\, \psi_{j, [zT]}(s) \right|^2
\end{equation}
is the wavelet periodogram on the segment $[zT]-{L(T)}/2+1, \ldots, [zT]+{L(T)}/2$ of the {\em
stationary} process
\begin{equation}\label{eq:styY}
Y_s = \sum_{j=1}^{\infty} W_j(z)  \sum_{k=-\infty}^\infty  \psi_{j, k}(s) \xi_{j, k}.
\end{equation}
Here $\psi_{j, k}$ is the same wavelet sequence as previously, $\{ \xi_{j, k} \}$
a set of independent identically distributed random variables with mean zero and unit variance and
$W_j(z)$ is such that $W^2_j(z) = S_j(z)$ for all $z\in(0,1)$ and $j \in \N$. The next theorem is the key result establishing the asymptotic properties of the
integrated local wavelet periodogram.

\begin{theorem}\label{thm:JNphi}
Let $\{ X_{t,T} \}$ be a zero-mean Gaussian locally stationary
wavelet process as defined by Definition~\ref{def:ILWP}.
Suppose $\sum_j C_j^2 2^{2j} < \infty$,
$\{ W_j \}_j$ is Lipschitz continuous with Lipschitz constants
$L_j$ such that $\sum_{j} L_j^2 2^{2j}<\infty$ and $\sum_{j} W_j^2(z) 2^{2j}<\infty$ at any rescaled time $z$, $L(T)/T\rightarrow 0$ as $T\rightarrow \infty$, and $\phi \in \Phi$ is a sequence of bounded variation.
Further, assume $h$ is a rectangular kernel. Then, using the family of discrete  Haar wavelets, we have

\begin{align}
\E\left\{ J_{L(T)}(z,\phi) \right\} = \E\left\{ J^Y_{L(T)} (\phi)\right\} &+ \calO \left\{ {L(T)}^{-1}
\right\},
\label{eq:approx1}\\
J_{L(T)}(z,\phi)-\E\left\{ J_{L(T)}(z,\phi)\right\} &=o_{p}\left({L(T)}^{-1/2}\right),
\label{consist1}\\
J^Y_{L(T)}(\phi)-\E\left\{J^Y_{L(T)}(\phi)\right\} &=o_{p}\left({L(T)}^{-1/2}\right).
\label{consist2}
\end{align}

\end{theorem}
\begin{proof}
Section~\ref{app:thm:JNphi} contains the full proof.
\end{proof}

An important difference between earlier literature in {\em this} area and our work is the
introduction of windowing.
We provide new  results on windowed versions of the cross-correlation wavelets, which we denote $i_{N, z}$, where,
to simplify notation, we replace $L(T)$ by $N$ and sometimes omit $z$.
To prove Theorem~\ref{thm:JNphi} we need
bounds on quantities involving $i_{N, z}$ which we can obtain via their connection with
cross-correlation wavelets and, in particular, our new {\em closed form} expression
for the cross-correlation Haar wavelet. For completeness, we
define the truncated
cross-correlation wavelet here and some of the key bounds.
\begin{definition}\label{def:iNz}
For $N \in \nats$, scales $j, \ell \in \nats$ and rescaled time $z \in (0, 1)$, the windowed cross-scale autocorrelation wavelets $i_{N, z} (j, \ell, \cdotp)$ over the interval $\left[ [zT]-N/2+1, [zT]+N/2 \right]$ are
\begin{equation}\label{eq:iNz}
i_{N, z} (j, \ell, k) = \sum_{t=0}^{N-1}  \psi_{j, [zT]-t} \psi_{\ell, k- [zT]-t-1+N/2},
\end{equation}
where $\{ \psi_{j, m} \}_{j,m}$ is a family of discrete wavelets and $k\in\ints$.
\end{definition}
The similarity between the cross-scale autocorrelation wavelets $\Psi_{j, \ell}(\cdotp)$, defined
in \citet[Definition~5.4.2]{Fryzlewicz03} as $\Psi_{j, \ell} (\tau) = \sum_{k \in \ints} \psi_{j, k} \psi_{\ell, k + \tau}$ for
$j, \ell \in \nats$ and $\tau \in \ints$, and their
windowed version, $i_{N, z}(j, \ell, \cdotp)$ defined above, is key to how we subsequently bound
quantities involving $i_{N, z}$.
The exact new formulae for Haar cross-scale autocorrelation wavelets are established
in Appendix~\ref{proof:psilj}, along with a pictorial description in Figure~\ref{fig:omegai} in
Appendix~\ref{app:add}.

As  bounds for $i_N(j, \ell, \cdotp)$ are a key component of the proof of Theorem~\ref{thm:JNphi}, these are provided by the next three results. The first bound for $i_N$ is valid for all discrete wavelets
based on  \cite{daubechies92:ten} compactly supported wavelets, although we later only
use it for Haar wavelets.

\begin{lemma}\label{lem:boundi}
Using  previous notation and assumptions,
 let $b_1 = [zT] + N/2 +1$ and $b_2 = [zT] + N/2 + N_\ell - 1$. Then
\begin{equation}
\label{eq:maineq}
| i_{N, z} (j, \ell, k) | \leq | \Psi_{j, \ell} ( k - 2[zT] + N/2 - 1)|,
\end{equation}
holds when
\begin{equation}
\label{eq:b1}
[zT] > N_j - 2 \mbox{ and } k <  b_1
\end{equation}
or when
\begin{equation}
\label{eq:b2}
k >  b_2,
\end{equation}
for integers $k$, $z \in (0,1)$, $j, \ell \in \nats$ and $N_j$ is the length of the discrete
wavelet
$\psi_{j, \cdot}$ for all
Daubechies compactly supported wavelets.

When $b_1 \leq k \leq b_2$ we have (i) for Daubechies' wavelets with two or more vanishing moments:
\begin{align}
| i_{N, z} (j, \ell, k) | &\leq | \Psi_{j, \ell} ( k - 2[zT] + N/2 - 1)|\nonumber\\
&+ \frac{2^{-(j+\ell)/2} }{[zT] - N}
	\left\{ \gamma +  \log ( [zT] -N ) - \log ( [zT] ) + \log N \right. \nonumber\\
&+ \left. \calO ( [zT]^{-1}) + \calO \{ ([zT]-N)^{-1} \} + \calO(N^{-1}) \right\}, \label{eq:part3}
\end{align}
where $\gamma$ is the Euler-Mascheroni constant and (ii) for Haar wavelets we have:
\begin{equation}
\label{eq:otherkbound}
| i_{N, z} (j, \ell, k) |  \leq 2^{-(j+\ell)/2}\left\{ \min(N_\ell,N_j)+N_\ell\right\}.
\end{equation}
\end{lemma}
\begin{proof}
See Section~\ref{app:boundi}.
\end{proof}
We use Lemma~\ref{lem:boundi} to prove the next two useful results about
$i_N$.
\begin{lemma}\label{lem:order2l}\label{lem:orderlpm}
Using previous notation and assumptions, and assuming $\{ \psi_{j, k} \}$ are discrete
Haar wavelets
\begin{align}
\sum_{k=-\infty}^\infty \sum_{j=1}^{\infty} \av i_{N,z}(j, \ell, k) \av^2 &=
\calO(2^{2\ell}),\label{eqn:order2I} \\
%
%
\sum_{k=-\infty}^\infty \sum_{n=-\infty}^\infty \left\{\sum_{j=1}^{\infty}\av  i_{N,z}(j, \ell, k)
i_{N,z}(j,m,n) \av\right\}^2 &= \calO \{ 2^{(\ell+m)}\}.\label{eqn:2lpm}
\end{align}
\end{lemma}
\begin{proof}
See Section~\ref{app:orderlpm}.
\end{proof}
These properties of the integrated local wavelet periodogram allow us to establish the asymptotic
behaviour of $\ptildeW{z}{\tau}$ in the following section.

\subsection{Windowed local partial autocorrelation estimation}\label{sec:windowlpacf}
We now define a  local partial autocorrelation estimator by using the classical (stationary) partial autocorrelation computed on a window of length $L(T)$ centred at time $[zT]$. The theoretical properties of this windowed estimator are derived and we investigate its empirical behaviour.

\begin{definition}
\label{def:windowed}
Let $\hat{q}$ be the usual partial autocorrelation estimator as defined by 
\citet[Definition~3.4.3]{brockwell91:time} for example.
Define the window $\calI(z, L) := \left[ z- L(T)/2T, z+ L(T)/ 2T \right]$ for some interval length
function $L(T)$ and location $z\in (0,1)$.
Define the windowed estimator, $\ptildeW{z}{\tau}$, of the local partial autocorrelation
function at rescaled time $z$ and lag $\tau$, to be the classical partial autocorrelation function
evaluated on observations contained in $\calI(z, L)$ and denoted by
\beq\nonumber\label{eq:estqw}
\ptildeW{z}{\tau}=\hat{q}_{\calI (z, L)}(\tau).
\eeq
\end{definition}
Our definition uses a rectangular window, but some of our applications later use
an  Epanechnikov window. Other variants could also be substituted.

The integrated wavelet periodogram approximation derived in Theorem~\ref{thm:JNphi} ensures that our windowed estimator can benefit from the established asymptotic distributional properties of the partial autocovariance \bcb estimator \ecb in the stationary setting, \bcb including its standard deviation, relevant for practical tasks. \ecb

\begin{proposition}\label{coro:pwindow}
Let $\{X_{t,T}\}$ be a zero-mean Gaussian locally stationary wavelet process under the conditions
set out by Theorem~\ref{thm:JNphi}. Then, for the windowed local partial autocorrelation estimator
$\ptildeW{z}{\tau}$ from Definition~\ref{def:windowed},
assuming $L(T) \rightarrow \infty$ and $L(T)/T \rightarrow 0$, as $T\rightarrow\infty$, we have that $\ptildeW{z}{\tau} $ converges in distribution to $\hat{q}^{Y}(\tau)$, where $Y$ is a stationary process with the same characteristics at rescaled time $z$ as the process $\{X_{t,T}\}$ (constructed as in equation~\eqref{eq:styY}) and $\hat{q}^{Y}(\tau)=\hat{\ph}^Y_{\tau,\tau}$ is the classical Yule-Walker partial autocorrelation function estimator.
\end{proposition}
\begin{proof}
Section~\ref{app:coro:pwindow} contains the proof, which
relies on the integrated wavelet periodogram approximation
from Theorem~\ref{thm:JNphi}.
\end{proof}

When dealing with processes that can be locally well modelled by an autoregressive structure, the result above amounts to establishing the asymptotic normality of our windowed local partial autocovariance estimator for large lags (see next corollary).

\begin{corollary}\label{coro:pnormal}
Under the assumptions from Proposition \ref{coro:pwindow} and assuming that $\{X_{t,T}\}$
can be locally well modelled by an autoregressive structure of order say $p$, then for lags $\tau$ larger than $p$ we have that $L(T)^{1/2}\ptildeW{z}{\tau}$ converges in distribution to a standard normal random variable.
\end{corollary}
\begin{proof}
The proof follows directly from Proposition \ref{coro:pwindow} and classical theory on the asymptotic behaviour of Yule-Walker estimates for stationary autoregressive processes (see for instance Theorem 8.1.2 from \cite{brockwell91:time}).
\end{proof}

\subsection{Choice of Control Parameters}
\label{sec:choicecontrol}
As with many nonparametric estimation methods in the literature,
we have to make various choices in an attempt
to obtain good estimators $\tilde{q}_W(z, \tau)$. Unfortunately there is no universal automatic best choice, at least in the real world.
For the wavelet estimator, $\tilde{q}$, we have to specify an underlying wavelet, a method
for handling boundaries and also a smoothing parameter, e.g.\ $s$ in Section~3.3 of
\cite{nason13}. However, a further advantage of the windowed estimator is that
we really only have to choose the window width $L(T)$
and the window kernel. \cite{dahlhaus98:on} show that
the Epanechnikov window is a good choice,
which we also advocate here.

Unfortunately, rates of convergence of the estimator, although providing theoretical
insight, do not really help
with the practical selection of the window width. A promising direction for practical bandwidth selection might be
via methods such as the locally stationary process bootstrap for pre-periodogram-like quantities,
as proposed by \cite{kreiss14:bootstrapping}, but development of this is beyond the scope of the current paper.

Below, we use a
manually-selected window width, by observing choices that achieve a good balance between
estimates that are too rough, and those that appear too smooth (and change little on further
smoothing). Section~\ref{sec:eastport} and Appendix~\ref{sec:eastportwidths} provide some empirical
evidence that the window width choice is not too hard, and the results are not particularly sensitive to it.
Such manually-selected procedures are well-acknowledged in the literature,\begin{changebar}
e.g. 
\cite{chaudhuri99:sizer}, although \bcb a cross-validation method for bandwidth selection is available in our associated software at increased computational cost. \ecb
This cross-validation combines a series of dyadic cross-validations, each
a simple extension of the even/odd dyadic cross-validation for wavelet shrinkage found
in \cite{nason96}.\end{changebar}


\section{Local partial autocorrelation estimates in practice}\label{sec:4}

\subsection{Simulated nonstationary autoregressive examples}

\bcb We illustrate our local partial autocorrelation function estimators on two simple, well-understood examples: (a) simulated time-varying autoregressive process TVAR$(1)$ and(b) piecewise AR$(p)$. \ecb

\bcb Consider a single $T=512$ realization from a time-varying autoregressive process with lag one coefficient linearly changing from $0.9$ to $-0.9$ over the series.  Figure~\ref{tvar1} shows the  partial autocorrelation function estimators, under the classical assumption of process stationarity (top left plot) and our two (time-dependent) estimators (top right and bottom plots). \ecb  The 95\% confidence bands are constructed under the null hypothesis of white noise and are the standard ones as displayed by, e.g., established R software. \bcb The red dotted lines show the true partial autocorrelation, a linear function of time at lag $1$, and constant ($0$) through time from lag $2$ onwards. \ecb
\begin{figure}
\centering
\resizebox{0.45\textwidth}{!}{\includegraphics{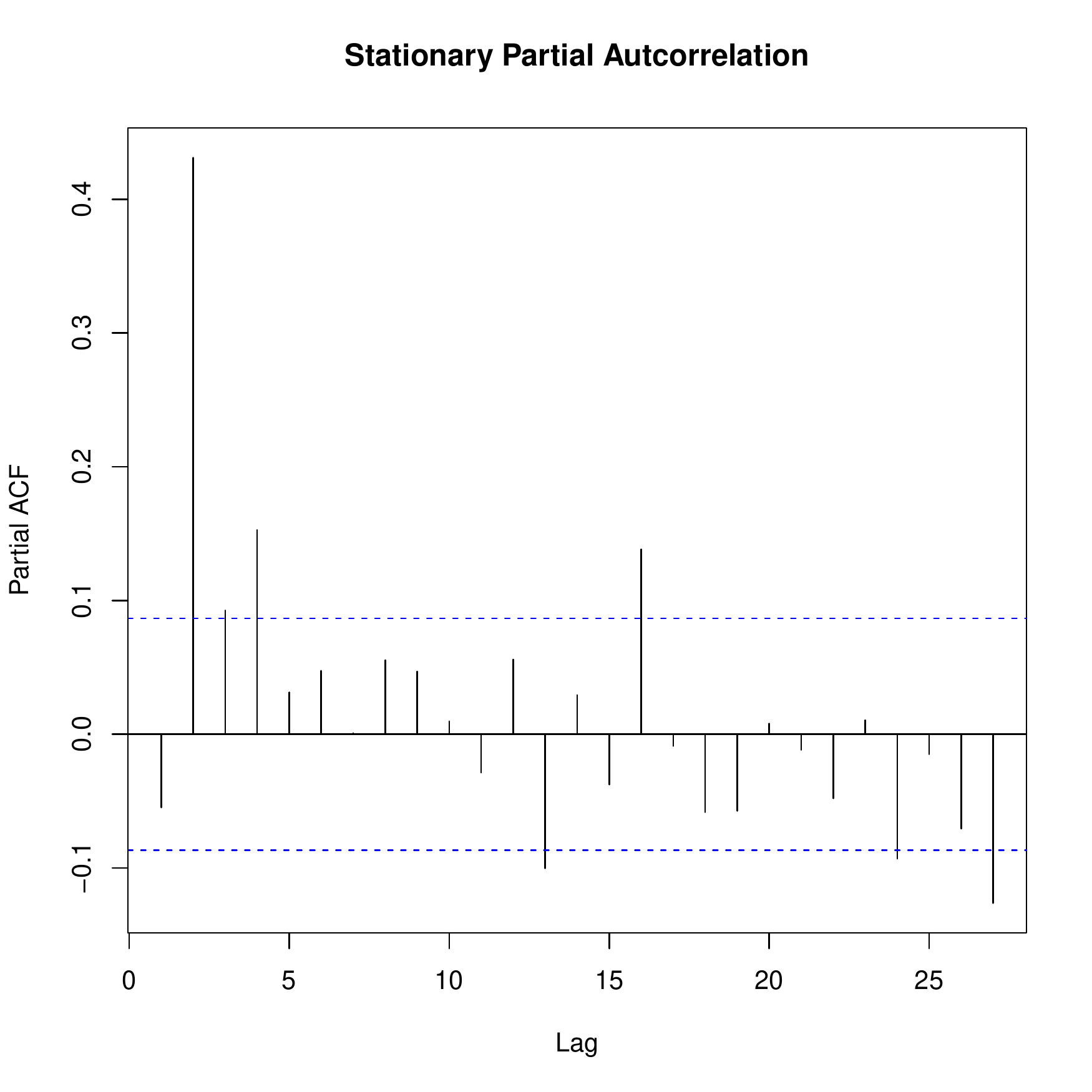}} \
\resizebox{0.45\textwidth}{!}{\includegraphics{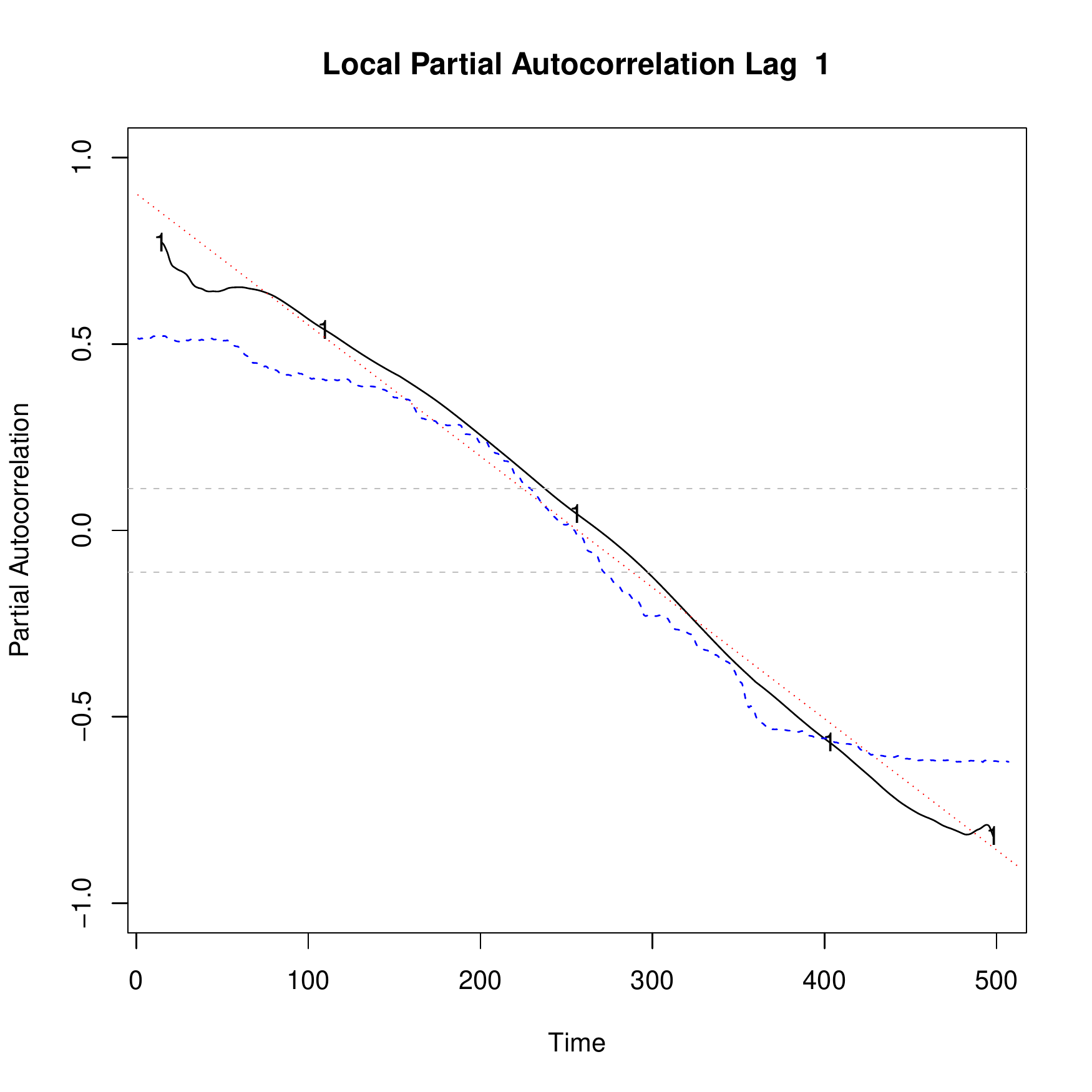}}\\
\resizebox{0.45\textwidth}{!}{\includegraphics{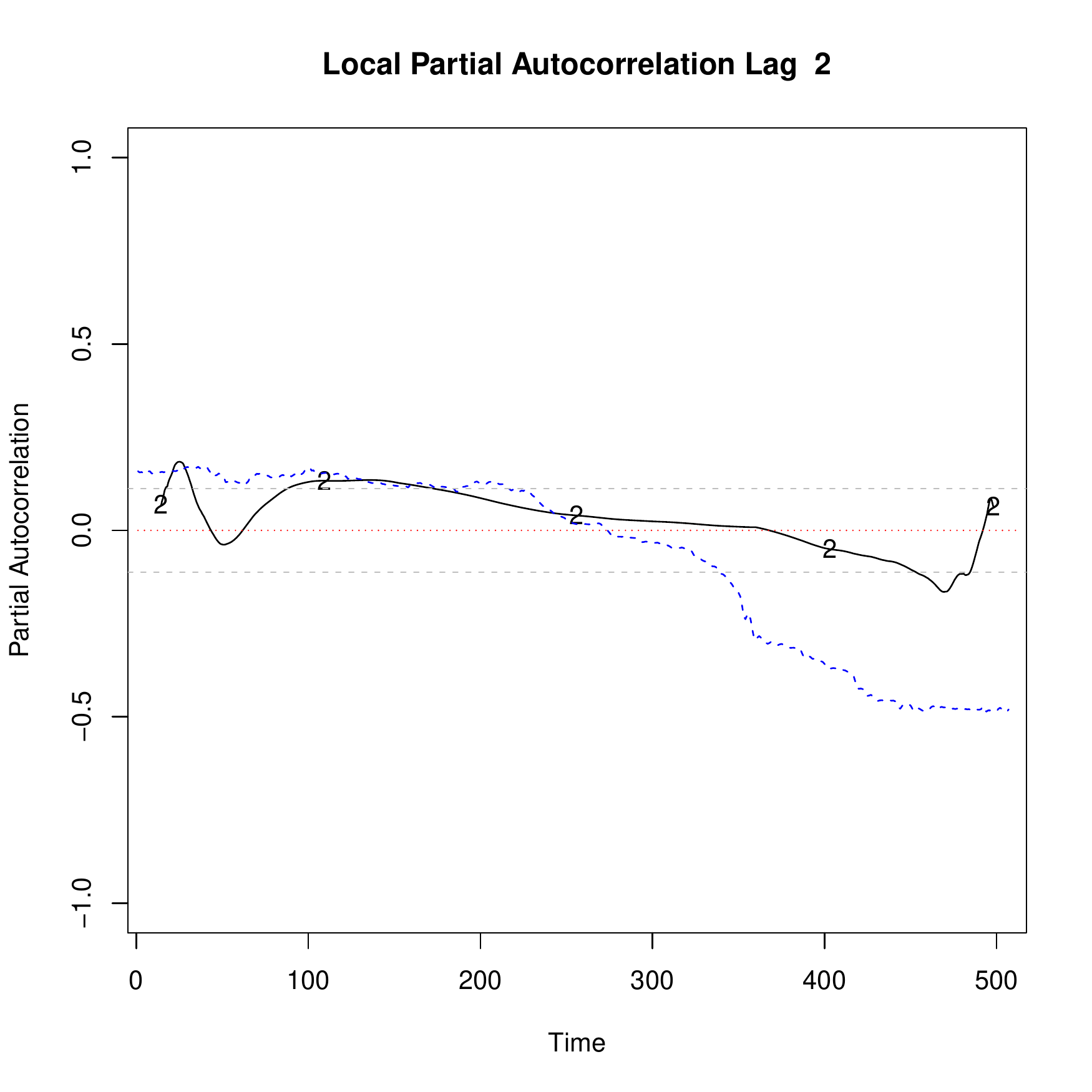}}
\caption{Partial autocorrelation function (pacf)
estimators applied to a single realization of a time-varying
autoregressive process. {\em Top left:} classical (stationary) pacf.
{\em Top right:} Lag one local pacf.  {\em Bottom:} Lag
two local pacf. Theoretical: red dotted line, Daubechies D5 wavelet-based estimate
($\tilde{q}$, Section \ref{sec:lpacf:est}) is dashed blue line,
Epanechnikov windowed estimate ($\tilde{q}_W$,
Section \ref{sec:3}) is solid black line.\label{tvar1} Bandwidth selected
using {\tt AutoBestBW} from {\tt locits} package.}
\end{figure}

Unsurprisingly, the classical partial autocorrelation is  misleading\bcb, indicating a significant {\em incorrect} strong lag two structure, and entirely failing to detect the existing (true) lag $1$ dependence. By contrast, our two developed local partial autocorrelation estimators correctly track the true {\em time-dependent} autoregressive parameters, thereby showing the importance of not using techniques designed for stationary series on nonstationary ones.  Amongst our two proposals, \ecb the wavelet-based estimate seems a bit worse, particularly for the lag two partial autocorrelation after about time $350$. This was confirmed by a small simulation study, based on $100$ realizations drawn from the TVAR process. The average root-mean-square error for the wavelet estimator \bcb (times $10^{2}$, standard errors in parantheses)
at lags one and two was $2.4 (0.76)$ and  $18.0 (3.9)$, respectively, whereas for the windowed estimator it was $1.5 (0.70)$ and $27.9 (4.7)$ respectively. \ecb Both estimators are less accurate near the ends of the series, which is a common problem with such estimators, see~\cite{CH03}, for example. However, the windowed estimator usually appears less affected, \bcb and thus is the estimator we propose to use in practice. \ecb


The TVAR process used in Figure~\ref{tvar1} exhibits a large range of time-varying parameter values
from $-0.9$ to $+0.9$.
However, we repeated the example for less extreme parameter changes. Unless the parameter change is very
small, and the process is close to stationary, the classical partial autocorrelation still misleads. For smaller parameter changes, the
classical partial autocorrelation often gets the process order correct, but gives a partial autocorrelation value that is often close
to the average of the local partial autocorrelations.

\bcb
Our second example considers a piecewise stationary AR$(p)$ process of length $T=256$.  The first and last segments (each of length 85) are realizations of an AR$(1)$ process with $\phi=-0.2$, and the middle segment (of length 86) follows an AR$(2)$ process with $\phi=(0.5,0.2)$.  Note the middle segment has a significantly different structure to the first and last.  Our estimators correctly identify the process structure, otherwise invisible to classical approaches. This is verified by performing a small simulation study and drawing from this process 100 times. The average root-mean-square error for the wavelet estimator at lags one and two (times $10^2$, standard errors in parentheses) is $11\,(2)$ and  $3\,(2)$, respectively, whereas for the windowed estimator it is $7\,(1)$ and $3\,(1)$ respectively.  The process lag 2 structure is closer to stationarity (with corresponding true pacf 0, 0.2, 0 in the three segments) and this is reflected in the similar results for the two estimators.
\ecb


\subsection{U.K.\ National Accounts data}\label{example:abml}
The ABML time series obtained from the Office for National Statistics contains
values of the U.K.\ gross-valued added (GVA), which is a component of gross
domestic product (GDP).
Our ABML series is recorded quarterly from quarter one 1955 to quarter three 2010 and consists
of $T=223$ observations. \begin{figure}[ht]
\centering
 \subfigure{\includegraphics[width=0.45\textwidth]{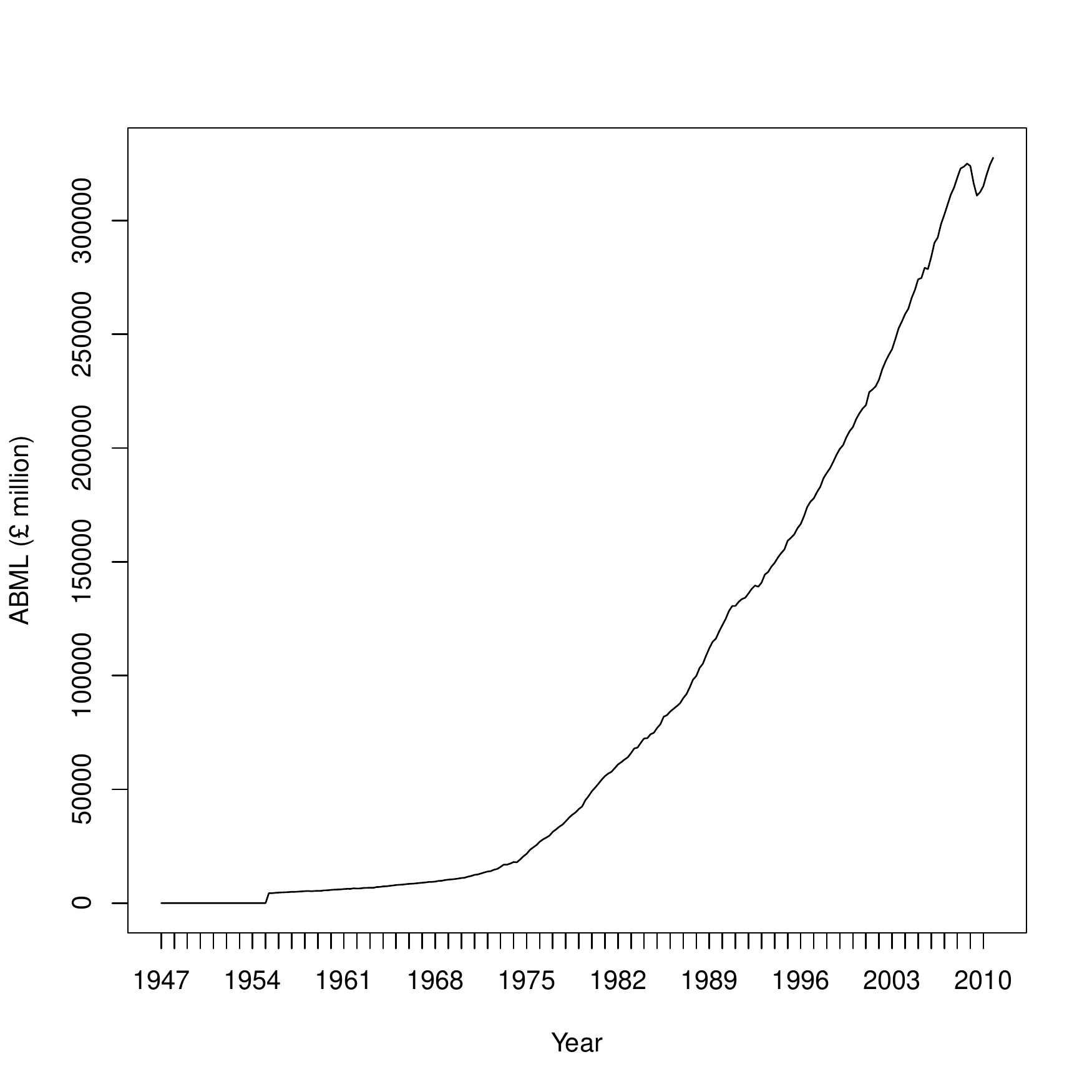}}
 	\hfill
  \subfigure{\includegraphics[width=0.45\textwidth]{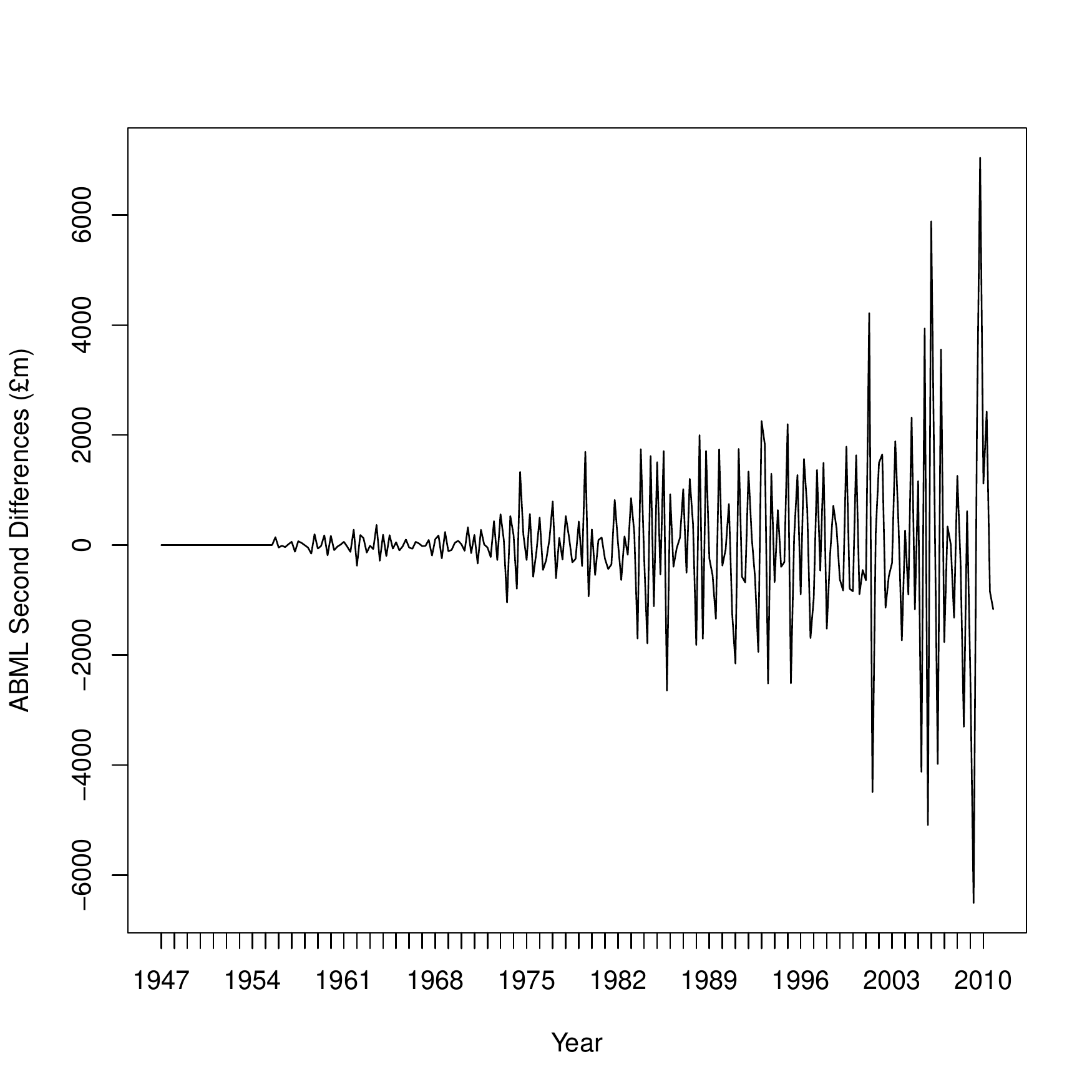}}
\caption{Left: ABML time series. Right: second-order differences of ABML.}\label{fig:abml}
\end{figure}
As with many economic time series, ABML exhibits a clear trend, which we removed using second-order differences; these are shown in Figure~\ref{fig:abml}.
 Naturally, other methods for removing the trend could be tried.
The second-order  differences strongly suggest that the series is not second-order stationary, with the
series variance increasing markedly over time.
Use of methods from~\cite{nason13} show that
 the autocorrelation also changes over time. In particular, the lag one autocorrelation undergoes a major and rapid shift around 1991.

Much of the increase in variance observed in Figure~\ref{fig:abml} is probably due to inflation. However, we have also analysed two different inflation-corrected versions of ABML, one provided by the U.K. Office of National Statistics, and both of these are also not second-order stationary, as determined by tests of
stationarity in \cite{PriestleySubbaRao1969} and \cite{nason13}.
%
\begin{figure}[ht]
\centering
\resizebox{0.9\textwidth}{!}{\includegraphics{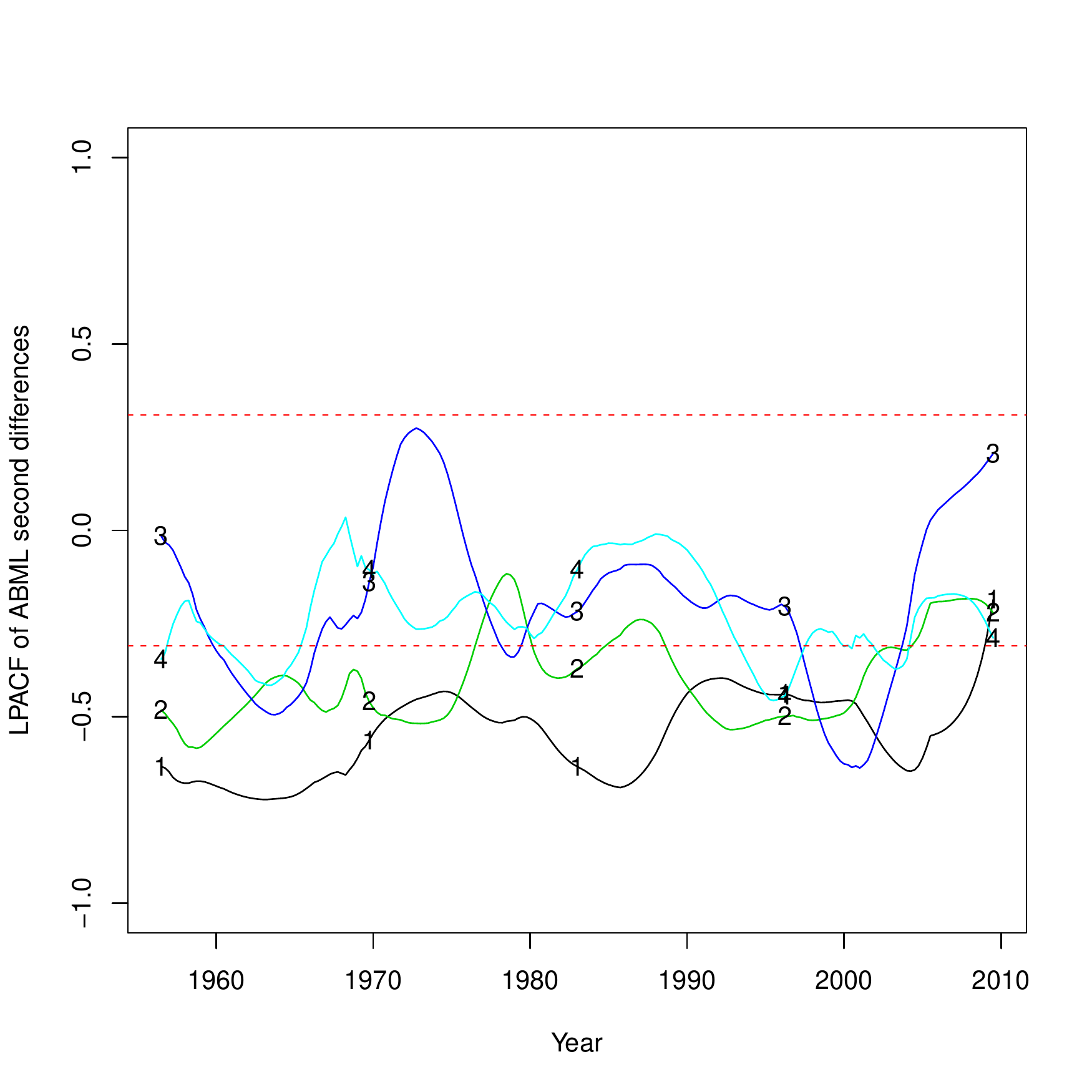}}
\caption{Windowed partial autocorrelation, $\tilde{q}_W(z, \tau)$, of ABML second difference series
	for lags one to four indicated on each curve
	(Epanechnikov window with $L=40$).
	Horizontal red dotted lines are approximate
	95\% confidence intervals.
	\label{fig:abmld2lpacf}}
\end{figure}

Our new estimation methodology enables us to obtain the
windowed local partial autocorrelation estimator, $\tilde{q}_W (z, \tau)$, shown in Figure~\ref{fig:abmld2lpacf} and computed on the ABML second-order  differences.
\bcb Note that crucially the local partial autocorrelation estimates within each lag ($\tau$) are time-dependent ($z$), and here these estimates suggest significant dependencies up to lag $\tau=4$. \ecb There are times, such as the 1970s, when the higher-order partial autocorrelations are not outside of the approximate significant bands, indicating that a \bcb lower lag, $2$, \ecb  might be appropriate. These results are (i) economically interesting as the local variance, autocorrelation and partial autocorrelation all  change over time, (ii) highlight the concerns with having no access to second-order conditional information (as was the case until now) and (iii) further pose the challenge of accurately forecasting such data. Although the topic of time series forecasting is outside the scope of this paper, many authors acknowledge the superiority of wavelet-based forecasting \citep{aminghafari07:forecasting,schlueter10:using} and we envisage the proposed local partial autocorrelation estimator could further improve results.

\subsection{Precipitation in Eastport, U.S.}
\label{sec:eastport}
Understanding precipitation patterns is important for detecting climate change indications and for policy decisions. The left panel in Figure~\ref{fig:eastport} shows monthly precipitation in millimetres from January 1887 until
December 1950 (768 observations) at a location in Eastport.  The data can
be found in \cite{hipel1994time} and have been analysed in many publications including \cite{rao2012nonstationarities,WRCR:WRCR21534}.
Our windowed local partial autocorrelation estimate of the Eastport data is
given in the right panel of Figure~\ref{fig:eastport} and shows clear
nonstationarity at lags one through three.
\begin{figure}
  \begin{center}
  {\includegraphics[width=0.45\textwidth]{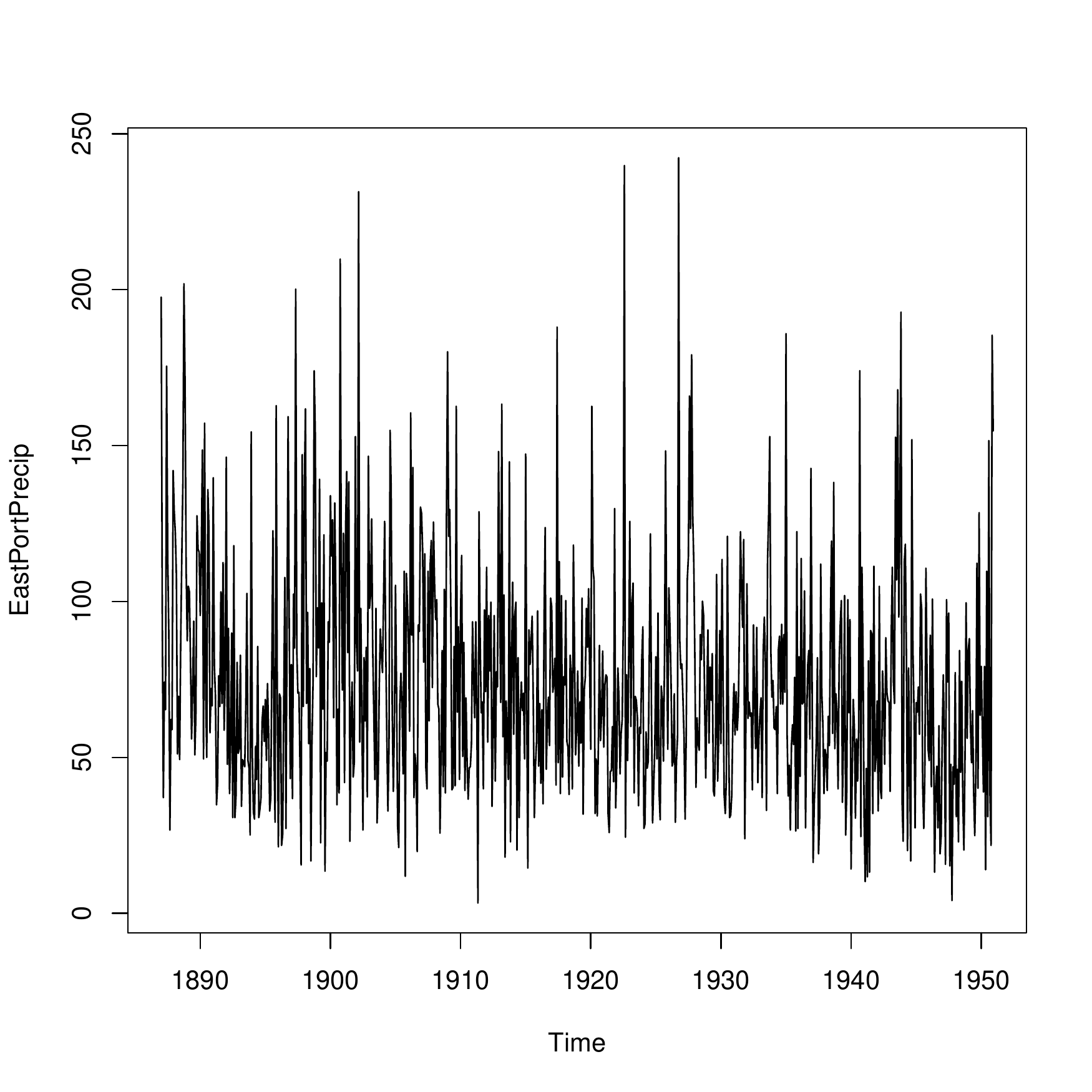}}
  \hfill
  {\includegraphics[width=0.45\textwidth]{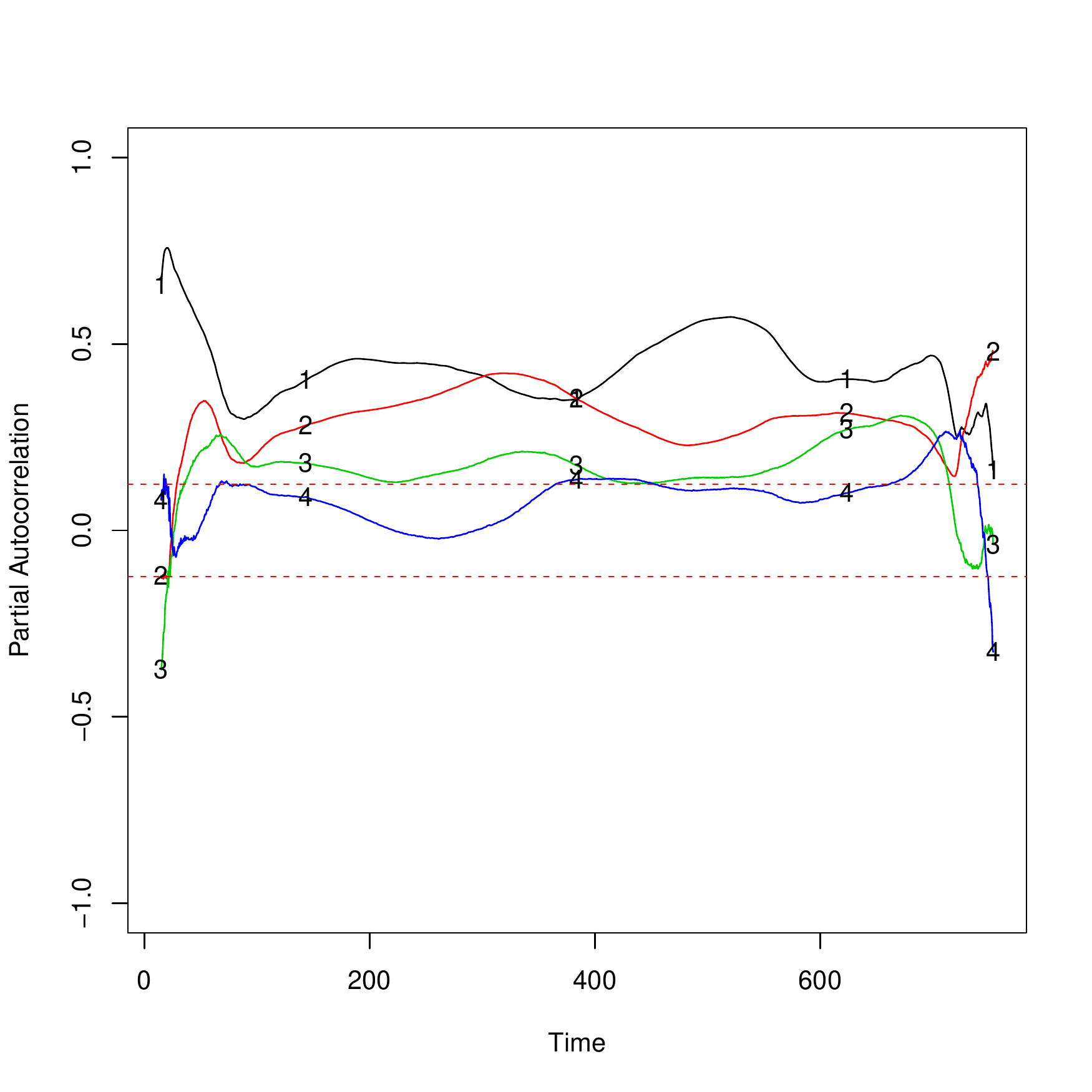}}
  \end{center}
  \caption{Left: Precipitation (mm) in Eastport, U.S. Right: Windowed partial autocorrelation,
$\tilde{q}_W(z, \tau)$, of left for lags one to four indicated on each curve
(Epanechnikov window with $L=250$, chosen by {\tt AutoBestBW} from {\tt locits} package.).
Horizontal red dotted
lines are approximate 95\% confidence intervals.
}\label{fig:eastport}
\end{figure}
Some authors analyse this series as if it were stationary and our analysis suggests that this is inappropriate.
Indeed, if one applies a formal hypothesis test of nonstationarity on appropriate lengths of the series, such as that proposed by \cite{CardinaliNason18},  there is strong evidence for nonstationarity. From a modelling point of view, the estimated local partial autocorrelation behaviour might support fitting a time-varying AR(3) model.

To provide some empirical support to the notion that window width is not visually critical to the interpretation of the
local partial autocorrelation, Appendix~\ref{sec:eastportwidths} shows the smoothed local partial autocorrelation
plots similar to that in the right-hand plot of Figure~\ref{fig:eastport}, but at three smaller window widths of 160, 80 and 40.
The plot at window width of 160 is not that different to the one above at $L=250$ and, indeed, the $L=80$ plot is not
that dissimilar. However, the $L=40$ plot almost certainly contains too much `noise' and should be disregarded.

\subsection{Euro-Dollar exchange rate}
Following the introduction of the Euro currency in 1999 several authors, including
\cite{AHAMADA2002177} and \cite{Garcin2017}, have considered different properties of this series,
which have an influence on setting monetary policy in various jurisdictions.

We analyze log returns of the monthly Euro-Dollar exchange rate as provided by EuroStat at
\\ \url{http://ec.europa.eu/eurostat/web/products-datasets/-/ei_mfrt_m} \\ from January 1999 until
October 2017.  The log returns and corresponding local partial autocorrelation function estimates
are given in Figure \ref{fig:eder}.  This demonstrates that the log returns do not appear to be
time varying (outside of the boundary locations) and exhibit only lag one partial autocorrelation.
This apparent stationarity is confirmed with formal tests using the \texttt{locits}
(\cite{R:locits}, \cite{nason13}) and \texttt{fractal} (\cite{R:fractal},
\cite{PriestleySubbaRao1969}) packages in \texttt{R}. Interestingly this relationship holds
throughout the financial crisis, from 2008 to 2011.

\bcb These examples highlight the versatility of our method and its potential use to identify stationary behaviour, manifest through local partial autocorrelation estimates that are constant through time. In addition, it highlights how the approach can identify departures from stationarity, evident through explicit time-dependent profiles at particular lags. \ecb
\begin{figure}
  \begin{center}
  {\includegraphics[width=0.45\textwidth]{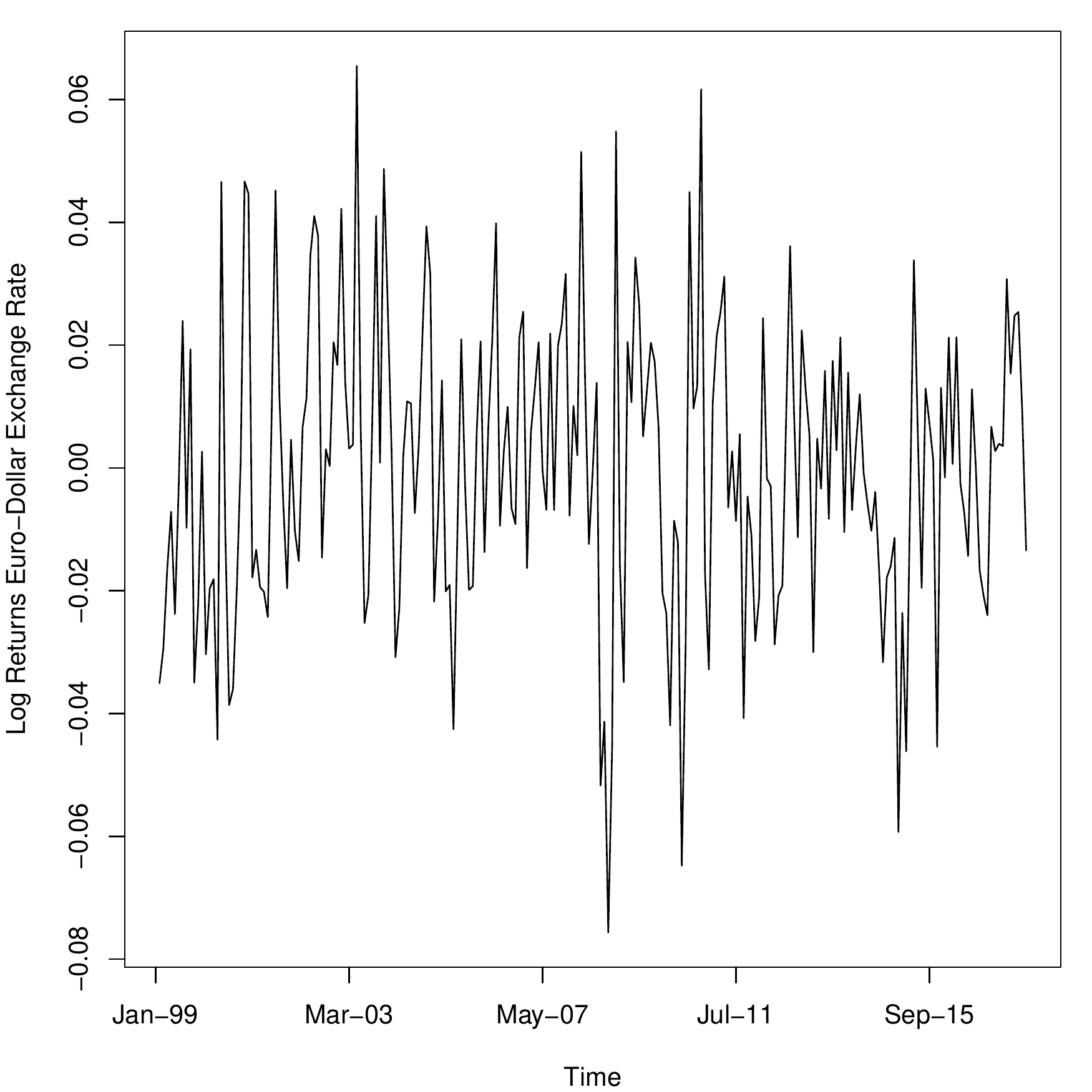}}
  \hfill
  {\includegraphics[width=0.45\textwidth]{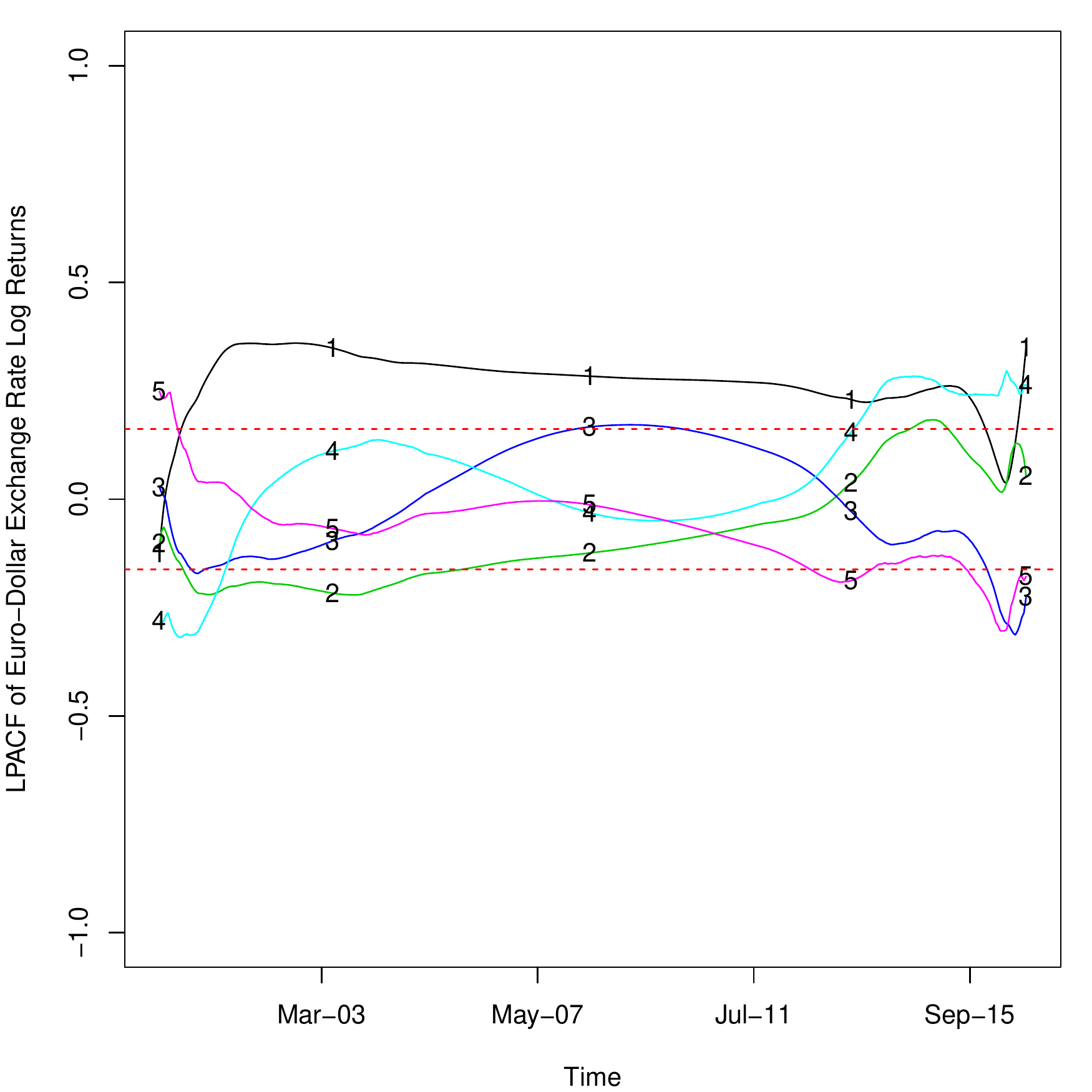}}
  \end{center}
    \caption{Left: Monthly Euro-Dollar log returns. Right: Windowed partial autocorrelation,
$\tilde{q}_W(z, \tau)$, of left for lags one to five indicated on each curve. Horizontal red dotted
lines are approximate 95\% confidence intervals. (Epanechnikov window with $L=147$)
}\label{fig:eder}
\end{figure}

%
%

\section{Discussion}
\label{sec:7}
This article develops  two new estimators of the local partial autocorrelation function and studied their theoretical properties when applied to a locally stationary wavelet process. We established consistency for the wavelet-based estimator and asymptotic distribution for the windowed estimator.  The latter result relied on new results on the integrated local wavelet periodogram, the (windowed) Haar cross-correlation wavelets and related quantities. \bcb For practical reasons, we promote the use of the windowed estimator. \ecb  We demonstrated the utility of these estimators for eliciting local \bcb second-order \ecb structure on simulated data, the U.S.\ Eastport precipitation time series and the U.K.\ ABML time series. We also demonstrated the versatility of our method in the (desirable) presence of stationarity for the Euro-Dollar exchange rates. \bcb On a practical note, should the practitioner believe that higher process powers also display a locally stationary behaviour, the proposed local partial autocorrelation function could then be used to additionally uncover higher-order dependency structures. \ecb
Most of the theoretical results relating to the generic local partial autocorrelation function estimator presented here are based on Haar wavelets, but many results and definitions also apply to other Daubechies' compactly supported wavelets.
The associated software package \texttt{lpacf} contains functionality to compute the estimators for all such wavelets, up to ten vanishing moments
as contained within the \texttt{wavethresh} package \citep{wavethresh}, as well as a cross-validation method for automatic bandwidth selection. The \texttt{lpacf} package
will be released on to the Comprehensive R Archive Network (CRAN) in due course.

\section*{Acknowledgements}
The authors were partially supported by the Research
Councils UK Energy Programme. The Energy Programme is an RCUK
cross-council initiative led by EPSRC and contributed to by ESRC, NERC,
BBSRC and STFC. GPN gratefully acknowledges support from EPSRC
grant K020951/1.

\bibliography{bibs}

\appendix

\section{R\'{e}sum\'{e}: partial autocorrelation for stationary  series}\label{sec:pacf}
Let $\{X_t\}_{t\in\ints}$ be a zero-mean second-order stationary process with autocovariance function
$\gamma(\tau)$. Loosely speaking,  the partial autocorrelation function at lag $\tau$ is
  the correlation between $X_1$ and $X_{\tau+1}$ whilst adjusting for the ``in-between''
observations, $X_2,\ldots, X_\tau$.
 \citet[p.~54]{brockwell91:time} define the closed span $\overline{\mbox{sp}} ( X_t, t \in H )$
of any subset $\{ X_t, t \in H\}$ of a Hilbert space ${\cal H}$ to be the smallest closed subspace
of ${\cal H}$ which contains each $X_t, t\in H$.
Then, following \citet[p.~98]{brockwell91:time}, the  lag $\tau$ partial autocorrelation function
$q(\tau)$ of $\{ X_t \}$ is defined by
		$q(\tau)=
\gpncor \left\{ X_{{\tau}+1}-P_{1,\tau} (X_{{\tau+1}}), X_{1}-P_{1,\tau}(X_{{1}}) \right\},$
where $P_{1,\tau}(\cdotp)$ denotes the projection operator onto
$\overline{\mbox{sp}}(X_{2}, \ldots, X_{\tau})$.
See also \citet[p. 43]{FanYao03}

Alternatively, if  $\gamma(0)>0$ and $\gamma(h)\rightarrow 0$ as $h\rightarrow \infty$, then the partial
autocorrelation function, $q(\tau)$, can be obtained as the final entry of the vector $\bphi_\tau$
which is the
solution to the well-known Yule-Walker equations $\Gamma_\tau\bphi_\tau = \bgamma_\tau$.
Here  $\Gamma_\tau=\{\gamma(i-j)\}_{i,j=1}^\tau$ is a $\tau \times \tau$ covariance matrix and
$\bgamma_\tau=\{ \gamma(i) \}_{i=1}^\tau$ is a vector of covariances. Equivalently,
$q(\tau)=\ph_{\tau,\tau}$ where $\ph_{\tau,\tau}$ is the coefficient of $X_1$ when projecting
$X_{\tau+1}$ on the space spanned by $X_1, \ldots, X_\tau$, i.e.\ the projection $\hat{X}_{\tau +
1} =\ph_{\tau,1}X_{\tau}+\ldots+\ph_{\tau,\tau}X_{1}$.

For a sampled series $\{X_{t}\}_{t=1}^{T}$, the sample partial autocorrelation at lag $\tau$ is often
estimated by solving $\hat{\Gamma}_\tau \hat{\bphi}_\tau = \hat{\bgamma}_\tau$, where $\hat{\gamma}$
are the usual sample autocovariances, and taking
$\hat{q}(\tau):=\hat{q}_{[1,T]}(\tau)=\hat{\ph}_{\tau,\tau}$. Here we use  the index notation
$[1,T]$ in order to indicate the range of observations on which the estimation of
$\hat{\Gamma}_\tau$ and $\hat{\bgamma}_\tau$ is based. The properties of $\hat{q}(\cdotp)$ are
well-known, see \citet[Section 8.10]{brockwell91:time}. In particular, $T^{1/2}\left\{
\hat{q}(\tau)-q(\tau) \right\}$ has a limiting Gaussian distribution, as $T \rightarrow \infty$,
with mean zero and variance proportional to the last term on the diagonal of $\Gamma_\tau^{-1}$.

\section{Definitions of locally stationary processes}
\label{sec:defls}
Definition~2.1 of \cite{dahlhaus97:fitting} is as follows.
\begin{quote}
``A sequence of stochastic processes $X_{t, T}$ $(t=0, \ldots, T-1)$ is called
locally stationary with transfer function $A^0$ and trend $\mu$ if there exists a representation
\begin{equation}
X_{t, T} = \mu ( t/ T) + \int_{-\pi}^\pi \exp (i \lambda t) A^0_{t,T} (\lambda) d\xi (\lambda),
\end{equation}
where the following holds.
\begin{enumerate}[(i)]
\item $\xi(\lambda)$ is a stochastic process on $[-\pi, \pi]$ with $\overline{\xi (\lambda)} = \xi(-\lambda)$ and
\begin{displaymath}
\cum \{ d\xi( \lambda_1), \ldots, d\xi(\lambda_k)\} = \eta\left( \sum_{j=1}^k \lambda_j \right)
	g_k (\lambda_1, \ldots, \lambda_{k-1}) d\lambda_1 \ldots d\lambda_k,
\end{displaymath}
where $\cum\{ \cdots \}$ denotes the cumulant of $k$th order, $g_1 =0$, $g_2(\lambda)=1$,
$| g_k (\lambda_1, \ldots, \lambda_{k-1})| \leq \mbox{const}_k$ for all $k$
and $\eta(\lambda) = \sum_{j=-\infty}^\infty \delta(\lambda + 2\pi j)$ is the periodic $2\pi$ extension
of the Dirac delta function.

\item There exists a constant $K$ and a \mbox{$2\pi$-periodic} function $A:[0,1] \times \R \rightarrow \C$
with $A(u, -\lambda) = \overline{A(u, \lambda)}$ and
\begin{equation}
\sup_{t, \lambda} \left| A^0_{t,T} (\lambda) - A \left( \frac{t}{T}, \lambda \right) \right| \leq K T^{-1}
\end{equation}
for all $T$; $A(u, \lambda)$ and $\mu(u)$ are assumed to be continuous in $u$.''
\end{enumerate}
\end{quote}

Definition~1 of \cite{nason00:wavelet}, including improvements from \citet{Fryzlewicz03}, is as follows.
\begin{quote}
``The locally stationary wavelet processes are a sequence of doubly indexed stochastic processes
$\{ X_{t,T} \}_{t=0, \ldots, T-1}$, $T = 2^J \geq 1$, having the representation in the mean-square sense
\begin{equation}
\label{eq:lswdef}
X_{t, T} = \sum_{j=1}^J \sum_k w^0_{j,k; T}  \psi_{j,k}(t) \xi_{j, k},
\end{equation}
where $\xi_{j,k}$ is a random orthonormal increment sequence and where $\{ \psi_{j, k} (t) \}_{j, k}$
is a discrete non-decimated family of wavelets based on a mother wavelet $\psi(t)$ of compact support.
The quantities in~\eqref{eq:lswdef} have the following properties:
\begin{enumerate}[(a)]
\item $\E (\xi_{j,k}) = 0$ for all $j, k$. Hence $E(X_{t, T}) = 0$ for all $t, T$,

\item $\cov{\xi_{j, k}, \xi_{\ell, m}} = \delta_{j, \ell} \delta_{k, m}$ where $\delta_{i,j}$ is the Kronecker delta.

\item There exists, for each $j \geq 1$ a Lipschitz continuous function $W_j(z)$ for $z \in (0,1)$ which fulfils
	the following properties:
	\begin{equation}
	\sum_{j=1}^\infty | W_j (z) |^2 < \infty \mbox{\ \ \ uniformly in $z\in (0,1)$};
	\end{equation}
	the Lipschitz constants $L_j$ are uniformly bounded in $j$ and
	\begin{equation}
	\sum_{j=1}^\infty 2^{j} L_j < \infty;
	\end{equation}
	there exists a sequence of constant $C_j$ such that for each $T$
	\begin{equation}
	\sup_{k} \left| w^0_{j, k;T} - W_j (k/T) \right| \leq C_j/ T,
	\end{equation}
	where $\{ C_j \}$ fulfils $\sum_{j=1}^\infty C_j < \infty$.''
\end{enumerate}
\end{quote}

\section{Miscellaneous covariance matrices}
\label{sec:covmat}
The $\tau \times \tau$ covariance matrices are given by
\beqann
\sigb&=&\left(\begin{array}{ccc}
\mbox{Cov}(X_{{[zT]},T},X_{{[zT]},T}) &\cdots& \mbox{Cov}(X_{{[zT]},T},X_{{[zT]+\tau-1},T})\\
\vdots &\cdots& \vdots\\
\mbox{Cov}(X_{{[zT]+\tau-1},T},X_{{[zT]},T}) &\cdots&
\mbox{Cov}(X_{{[zT]+\tau-1},T},X_{{[zT]+\tau-1},T})\end{array}\right)\\
 \mbox{ and}\\
\sigf&=&\left(\begin{array}{ccc}
\mbox{Cov}(X_{{[zT]+1},T},X_{{[zT]+1},T}) &\cdots& \mbox{Cov}(X_{{[zT]+1},T},X_{{[zT]+\tau},T})\\
\vdots &\cdots& \vdots\\
\mbox{Cov}(X_{{[zT]+\tau},T},X_{{[zT]+1},T}) &\cdots&
\mbox{Cov}(X_{{[zT]+\tau},T},X_{{[zT]+\tau},T})\end{array}\right).
\eeqann

$$
\sig=\left(\begin{array}{ccc}
\mbox{Cov}(X_{{[zT]+\tau-1},T},X_{{[zT]+\tau-1},T}) &\cdots& \mbox{Cov}(X_{{[zT]},T},X_{{[zT]+\tau-1},T})\\
\vdots &\cdots& \vdots\\
\mbox{Cov}(X_{{[zT]+\tau-1},T},X_{{[zT]},T}) &\cdots&
\mbox{Cov}(X_{{[zT]},T},X_{{[zT]},T})\end{array}\right)
$$

\section{Cross-scale autocorrelation Haar wavelets}
\label{proof:psilj}

First note that by substituting $s = [zT] -t$ in~\eqref{eq:iNz} and by denoting the rectangular kernel by $h$, we obtain:
\begin{equation}\label{eq:iNzalt}
i_{N,z} (j, \ell, k) = \sum_{s=[zT]-N+1}^{[zT]} \psi_{j, s} \psi_{\ell, s+k - 2[zT]+N/2-1}
h\left(\frac{[zT]-s}{N}\right).
\end{equation}
The above expression is by no means restricted to $h$ being a rectangular kernel, and other kernels may be used, as explained in the article main text.

This new formulation of $i_{N, z}$ is very similar to that of the
cross-correlation wavelet $\Psi_{j, \ell}$, except that the summation limits are $[zT-N+1]$ and
$[zT]$ instead of $-\infty$ and $\infty$. This similarity is true for all Daubechies' compactly
supported wavelets. We shall use this similarity to bound $i_{N, z}$ using Lemma~\ref{lem:boundi}.

\begin{proposition}
\label{prop:psilj}
The cross-scale autocorrelation Haar wavelets $\Psi_{j, \ell} (\cdotp)$ at scales $\ell < j$ are given by
\begin{equation}
\Psi_{j, \ell} (\tau) = 2^{-(j-\ell)/2} \begin{cases}
0 & \mbox{for } \tau < -2^{\ell},\\
-(2^{-\ell}\tau + 1) & \mbox{for } -2^{\ell} \leq \tau <  -2^{\ell-1},\\
2^{-\ell} \tau & \mbox{for } -2^{\ell-1} \leq \tau < 0,\\
0 & \mbox{for } 0 \leq \tau < 2^{j-1} - 2^\ell,\\
2^{-\ell} (2\tau - 2^j + 2^{\ell+1}) & \mbox{for } 2^{j-1} - 2^\ell \leq \tau < 2^{j-1} -
2^{\ell-1},\\
2^{-\ell} (2^j - 2\tau) & \mbox{for } 2^{j-1} - 2^{\ell-1} \leq \tau < 2^{j-1},\\
0 & \mbox{for } 2^{j-1} \leq \tau < 2^j - 2^{\ell},\\
2^{-\ell} (2^j - \tau - 2^\ell) & \mbox{for } 2^j - 2^\ell \leq \tau < 2^j - 2^{\ell-1},\\
2^{-\ell}(\tau - 2^j) & \mbox{for } 2^j - 2^{\ell - 1} \leq \tau < 2^j,\\
0 & \mbox{for } 2^j \leq \tau.
\end{cases}
\label{eq:HaarXCorr}
\end{equation}
For $\ell > j$ we have
 $\Psi_{j, \ell}(\tau) = \Psi_{\ell, j} (-\tau)$  and a precise formula appears
in equation~\eqref{eq:Psijlessell} (Appendix~\ref{proof:psilj}).
Note $\Psi_{j, j} (\tau) = \Psi_j(\tau)$, the regular autocorrelation wavelet.
\end{proposition}

\begin{proof}
See Section~\ref{suppmat:proof:psilj}.
\end{proof}

\begin{corollary}
\label{corr:psiljREMAIN}
The regions on the right-hand side of~\eqref{eq:HaarXCorr}
correspond exactly to the regions $III_a$ to $III_j$ in
\eqref{eq:regionsIII}.
For completeness (and usefulness in working out derived quantities) we
can write down $\Psi_{j, \ell} (\tau)$ for $\ell > j$ explicitly as
\begin{equation}
\Psi_{j, \ell} (\tau) = 2^{-(\ell - j)/2} \begin{cases}
0 & \mbox{for } \tau \leq -2^{\ell},\\
-2^{-j} (\tau + 2^\ell) & \mbox{for } -2^\ell < \tau \leq -2^\ell + 2^{j-1},\\
2^{-j} (2^\ell + \tau - 2^j) & \mbox{for } -2^\ell + 2^{j-1} < \tau \leq -2^\ell + 2^j,\\
0 & \mbox{for } -2^\ell + 2^j < \tau \leq -2^{\ell - 1},\\
2^{-j} (2^\ell + 2\tau) & \mbox{for } -2^{l-1}  < \tau \leq -2^{\ell-1} + 2^{j-1},\\
2^{-j} (2^{j+1} - 2^\ell - 2\tau) & \mbox{for } -2^{\ell-1} + 2^{j-1} < \tau \leq -2^{\ell-1} +
2^j,\\
0 & \mbox{for } -2^{\ell-1} + 2^j < \tau \leq 0,\\
-2^{-j} \tau & \mbox{for } 0 < \tau \leq 2^{j-1},\\
-(1 - 2^{-j} \tau) & \mbox{for } 2^{j-1} < \tau \leq 2^j.
\end{cases}
\label{eq:Psijlessell}
\end{equation}
\end{corollary}

\section{Subsidiary result used in the proof of Lemma~\ref{lem:boundi}}
\label{app:intres}

\begin{lemma}\label{lem:intres}
For $a, b$ such that $2a, 2b \in \nats$ and $a, b > 0$ it is the case that
\begin{equation}
\label{eq:I1}
 \int_{-\pi}^\pi \frac{ \{ 1 - \cos(2 a \omega) \} \{ 1 - \cos(2 b \omega) \}}{1 - \cos(\omega)}
\, d\omega = 4\pi \min(a,b).
\end{equation}
\end{lemma}

\begin{proof}
See Section~\ref{proof:lem:intres}.
\end{proof}

\section{Additional Results Required For the Proofs from Section \ref{sec:3}}\label{app:add}
\begin{lemma}
\label{lem:omegai}
The core function, $\Omega_i(u)$ for Haar wavelets is given by
\begin{equation}
\label{eq:omegaihaar}
\Omega_i (u) = 2^{-i/2} \begin{cases}
0 & \mbox{for } u < -1,\\
-(u+1) & \mbox{for } -1 \leq u < -\frac{1}{2},\\
u & \mbox{for } -\frac{1}{2} \leq u < 0,\\
0 & \mbox{for } 0 \leq u < 2^{i-1} - 1,\\
2u - 2^i + 2 & \mbox{for } 2^{i-1} - 1 \leq u < 2^{i-1} - \frac{1}{2},\\
2^i - 2u & \mbox{for } 2^{i-1} - \frac{1}{2} \leq u < 2^{i-1},\\
0 & \mbox{for } 2^{i-1} \leq u < 2^i - 1,\\
2^i - u  -1 & \mbox{for } 2^i - 1 \leq u < 2^i - \frac{1}{2},\\
u - 2^i & \mbox{for } 2^i - \frac{1}{2} \leq u < 2^i,\\
0 & \mbox{for } 2^i \leq u,
\end{cases}
\end{equation}
for $u \in \reals$ and $i \in \nats \cup \{0\}$.
\end{lemma}
Figure~\ref{fig:omegai} shows a depiction of $\Omega_i(u)$.
\begin{figure}[ht]
\centering
\resizebox{0.8\textwidth}{!}{\includegraphics{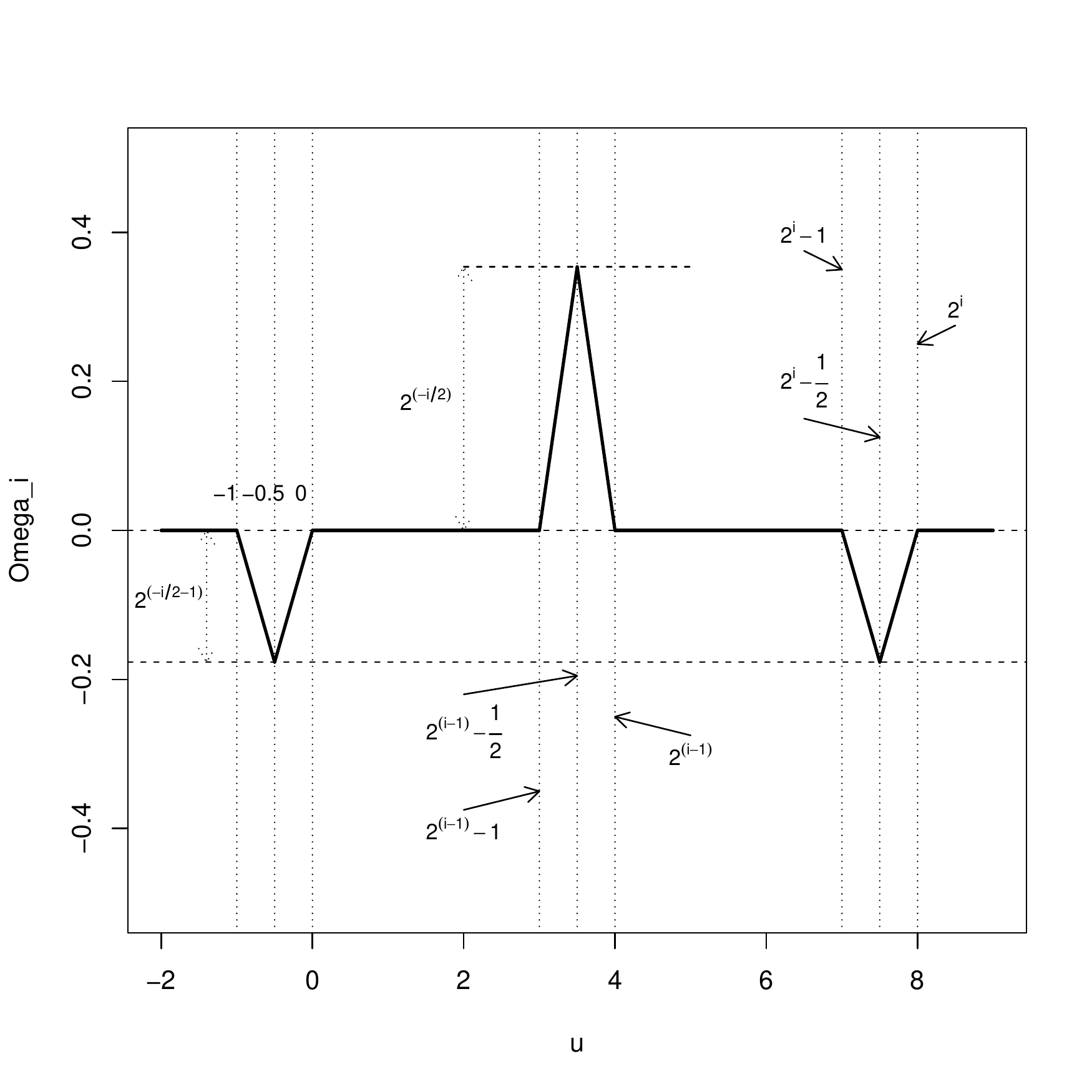}}
\caption{Depiction of $\Omega_i(u)$. The function is symmetric about $2^{(i-1)}-\tfrac{1}{2}$ and
the
extent of the function is from $-1$ on the left to $2^i$ on the right. The width of all of the
triangles is always
1 for all $i$. As $i$ increases the function gets stretched to the right (but also anchored on the
left at $u=-1$), the peaks decrease in size like $2^{-i/2}$.
\label{fig:omegai}}
\end{figure}

\begin{proof}
See Section~\ref{proof:lem:omegai}.
\end{proof}

\begin{lemma}\label{lem:bothB}
Under the conditions and notations set out so far, for the nondecimated family of discrete Haar
wavelets we have

\[T_{{\cal B}{\cal B}}(\ell, m) = \sum_{k \in {\cal B}} \sum_{n \in {\cal B}}  \left\{
		\sum_{j=1}^\infty  | i_{N, z} (j, \ell, k) i_{N, z} (j, m ,n)| \right\}^2 = {\cal O} \{
2^{2(\ell + m)} \}.\]
\end{lemma}
\begin{proof}
See Section~\ref{proof:lem:bothB}.
\end{proof}

\begin{lemma}
\label{lem:cross}
Under the conditions and notations set out so far, for the nondecimated family of discrete Haar
wavelets we have that the order of the cross terms is:
\begin{align*}
  T_{\not{\cal B}{\cal B}}(\ell, m) &= \sum_{k \not\in {\cal B}} \sum_{n \in {\cal B}}  \left\{
		\sum_{j=1}^\infty  | i_{N, z} (j, \ell, k) i_{N, z} (j, m ,n)| \right\}^2 = {\cal
O}(2^{2(\ell+m)}) \\
  T_{{\cal B}{\not\cal B}}(\ell, m) &= \sum_{k \in {\cal B}} \sum_{n \not\in {\cal B}}  \left\{
		\sum_{j=1}^\infty  | i_{N, z} (j, \ell, k) i_{N, z} (j, m ,n)| \right\}^2 = {\cal O} (
2^{2(\ell+m)}).
\end{align*}
\end{lemma}

\begin{proof} 
See Section~\ref{proof:lem:cross}.
\end{proof}

\begin{lemma}
\label{lem:PsiBound}
Under the conditions and notations set out so far, for the nondecimated family of discrete Haar
wavelets we
have that
\[ T_{\not{{\cal B}}\not{{\cal B}}}(\ell, m) = \sum_{k \not\in {\cal B}} \sum_{n \not\in {\cal B}}
\left\{\sum_{j=1}^\infty|i_{N,z}(j,\ell,k)i_{N,z}(j,m,n)|\right\}^2 = {\cal O} \{ 2^{2(\ell +
m)}\}.\]
\end{lemma}
\begin{proof}
See Section~\ref{proof:lem:PsiBound}.
\end{proof}

\section{Some Exact Formulae for Haar wavelets}
\label{app:Haarexact}
{\em Fourth-order absolute value wavelet cross-correlations for Haar wavelets.}
In what follows we demonstrate new results on the the fourth-order absolute value wavelet cross-correlations, for Haar wavelets which were used in showing the previous results in Appendix~\ref{app:thm:JNphi}.

Recall these were defined as
$B^{(r)}_\ell (j, i) = \sum_{p=-\infty}^\infty |p|^r | \Psi_{j, \ell}(p) \Psi_{i, \ell} (p)|$
for $r=0, 1$ and scales $\ell, j, i \in \nats$.

The $B$ products are symmetric in their arguments, $B^{(r)}_\ell (j, i) = B_\ell^{(r)} (i, j)$. Note that for $r=0$, for ease of notation, these $B^{(0)}$ quantities appeared as $B$ in the previous proofs.

\begin{proposition}\label{prop:B}
For Haar wavelets. (Part A) For $i, j > \ell$:
\begin{equation}\nonumber
B^{(0)}_\ell (j, i) =  \begin{cases}
2^{-j} (2^{2\ell-1} + 1) & \mbox{for $j=i$},\\
2^{-j}  (2^{2\ell-1} + 1) 2^{-3/2} & \mbox{for $i=j+1$},\\
2^{-j/2} 2^{-i/2}  (2^{2\ell-1}+ 1)/6 & \mbox{for $|j-i| > 1$.}
\end{cases}
\end{equation}
Also, for all $i, j, \ell$ such that $i, j > \ell$,
$B_\ell^{(0)}(j,i)$ is bounded by
\begin{equation}\nonumber
B_\ell^{(0)} (j, i) \leq 2^{-j/2} 2^{-i/2} 2^{2\ell}.
\end{equation}
(Part B) For $i, j < \ell$:
\begin{equation}\nonumber
B^{(0)}_\ell (j, i) = \begin{cases}
 2^{-\ell} (2^{2j-1} +1 ) & \mbox{for } i=j < \ell,\\
\tfrac{3}{2} 2^{-\ell} 2^{-j/2} 2^{5i/2 - 1} & \mbox{for } i < j < \ell.
\end{cases}
\end{equation}
(Part C) For $i < \ell < j$:
\begin{equation}
B^{(0)}_\ell (j, i) = \begin{cases}
\tfrac{1}{8} 2^{-(\ell+1)/2} 2^{3i/2} ( 2^{i-\ell} + 2 ) & \mbox{for } j = \ell+1,\\
\tfrac{1}{8} 2^{-j/2} 2^{3i/2} ( 2 - 2^{i-\ell}) & \mbox{for } j > \ell + 1.
\end{cases}
\label{eq:sandwichB}
\end{equation}
(Part D) For $\ell=j$ and $i > \ell$:
\begin{equation}
B^{(0)}_\ell (\ell, i) = 2^{-(i - \ell)/2} \begin{cases}
\frac{17}{9} 2^{\ell-3} & \mbox{for } i = \ell+1,\\
\frac{17}{27} 2^{\ell-3}& \mbox{for } i > \ell + 1.
\end{cases}
\label{eq:oneequal}
\end{equation}
For $i < \ell$ we have the following bound:
\begin{equation}\nonumber
B^{(0)}_{\ell} (\ell, i) \leq 2^{3i/2} 2^{-\ell/2}.
\end{equation}
(Part E) Finally, when all indices are equal we can use (34) from \cite{nason00:wavelet}
to show
\begin{equation}\nonumber
B^{(0)}_{\ell} (\ell, \ell) = \sum_p \Psi^2_\ell (p) = A_{\ell, \ell} = \frac{1}{3} 2^{-\ell}
(2^{2\ell} + 5),
\end{equation}
for $\ell > 0$ and $A$ is the matrix from~\cite{nason00:wavelet}.

The symmetry of $B$  permits evaluation of $B_\ell^{(0)} (j, i)$
for other orderings of $(i, j)$.

An overall bound for all $i,j,\ell$ is $B^{(0)}_\ell (j,i) \leq K  2^{-(j+i)/2} 2^{2\ell}$ for some
positive constant $K$.
\end{proposition}

\begin{proof}
See Section~ \ref{proof:prob:B}.
\end{proof}

\section{LPACF of Eastport Precipitation Data at Different Window Widths}
\label{sec:eastportwidths}

The plots below were produced by the following functions executed using the {\tt lpacf} package with
binwidths of 160, 80 and 40.
\begin{verbatim}
function(binwidth=250){
#
# Compute the Epanechnikov kernel smoothed local PACF using
# a specified binwidth, using parallel processing function
# mclapply on all points.
#
# Then plot the answer: only the first four lags
#  using colours 1 thru 4.
#
plot(lpacf.Epan(EastPortPrecip, allpoints=TRUE, binwidth=binwidth,
           		lapplyfn=mclapply), lags=1:4, lcol=1:4)
#
# Construct and plot "standard" confidence intervals
#
ci <- 1.96/sqrt(binwidth)
abline(h = c(-ci, ci), lty = 2, col = 2)
}
\end{verbatim}

\begin{figure}
  \begin{center}
  {\includegraphics[width=0.45\textwidth]{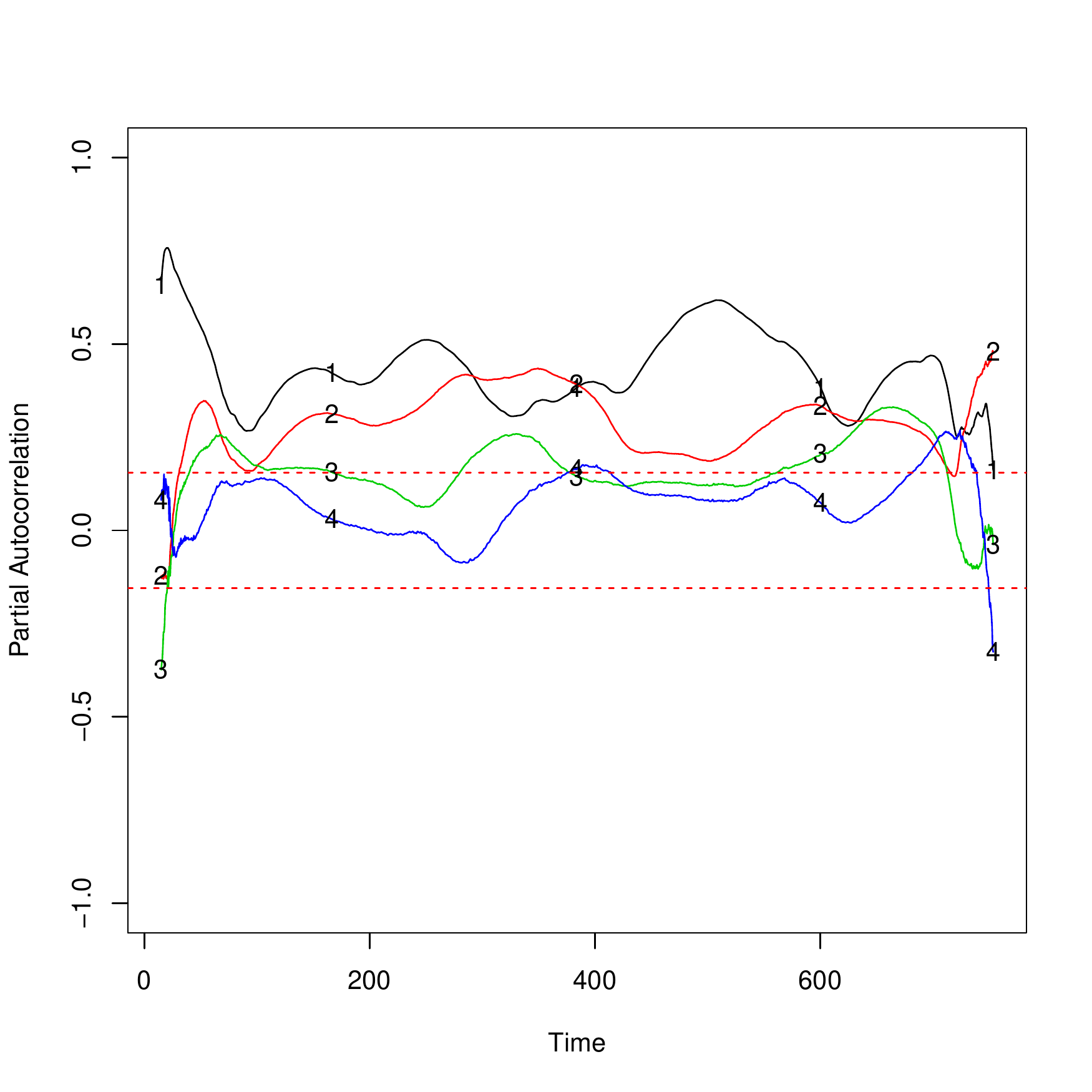}\label{fig:eastport160}}
  \hfill
  {\includegraphics[width=0.45\textwidth]{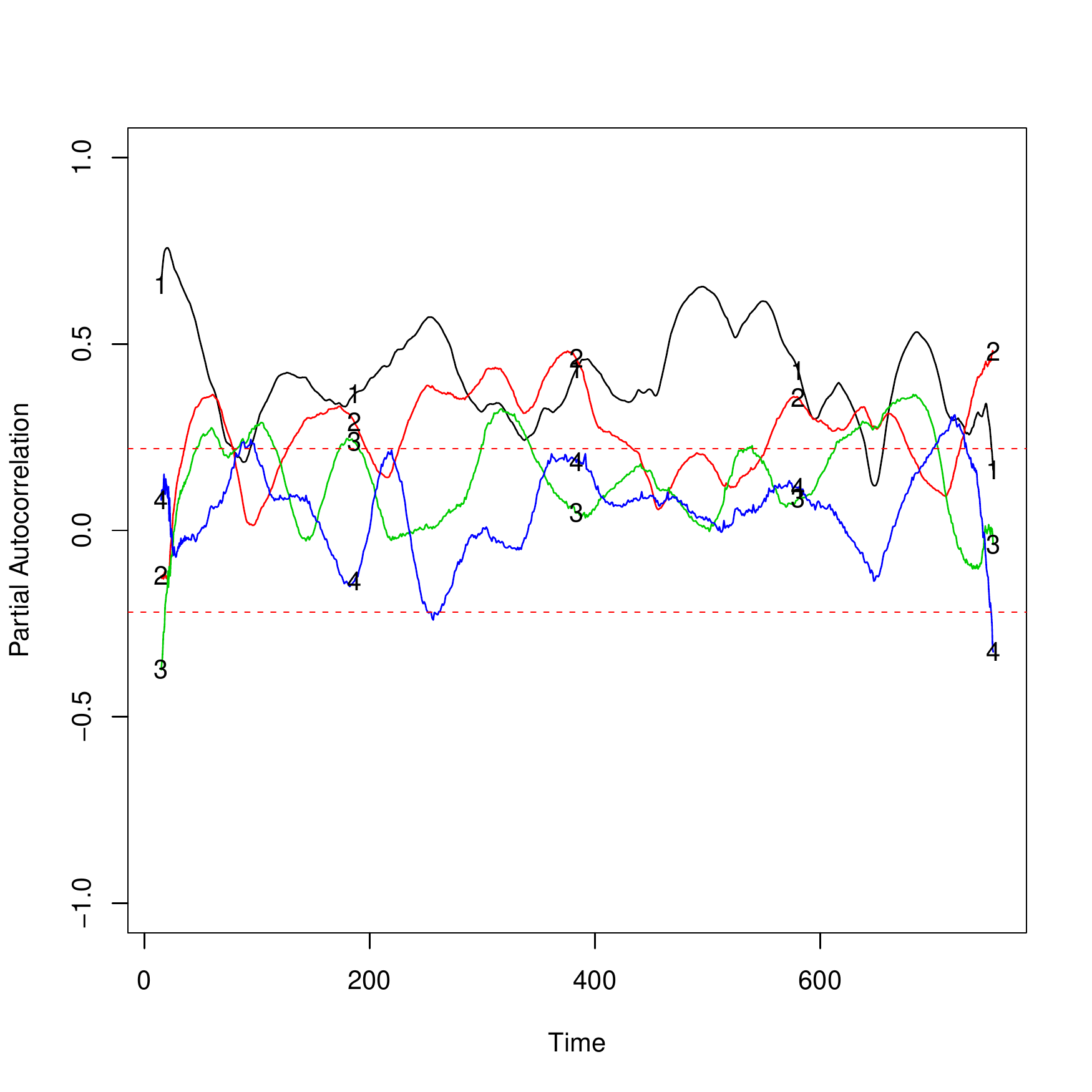}\label{fig:eastport80}}\\
  {\includegraphics[width=0.45\textwidth]{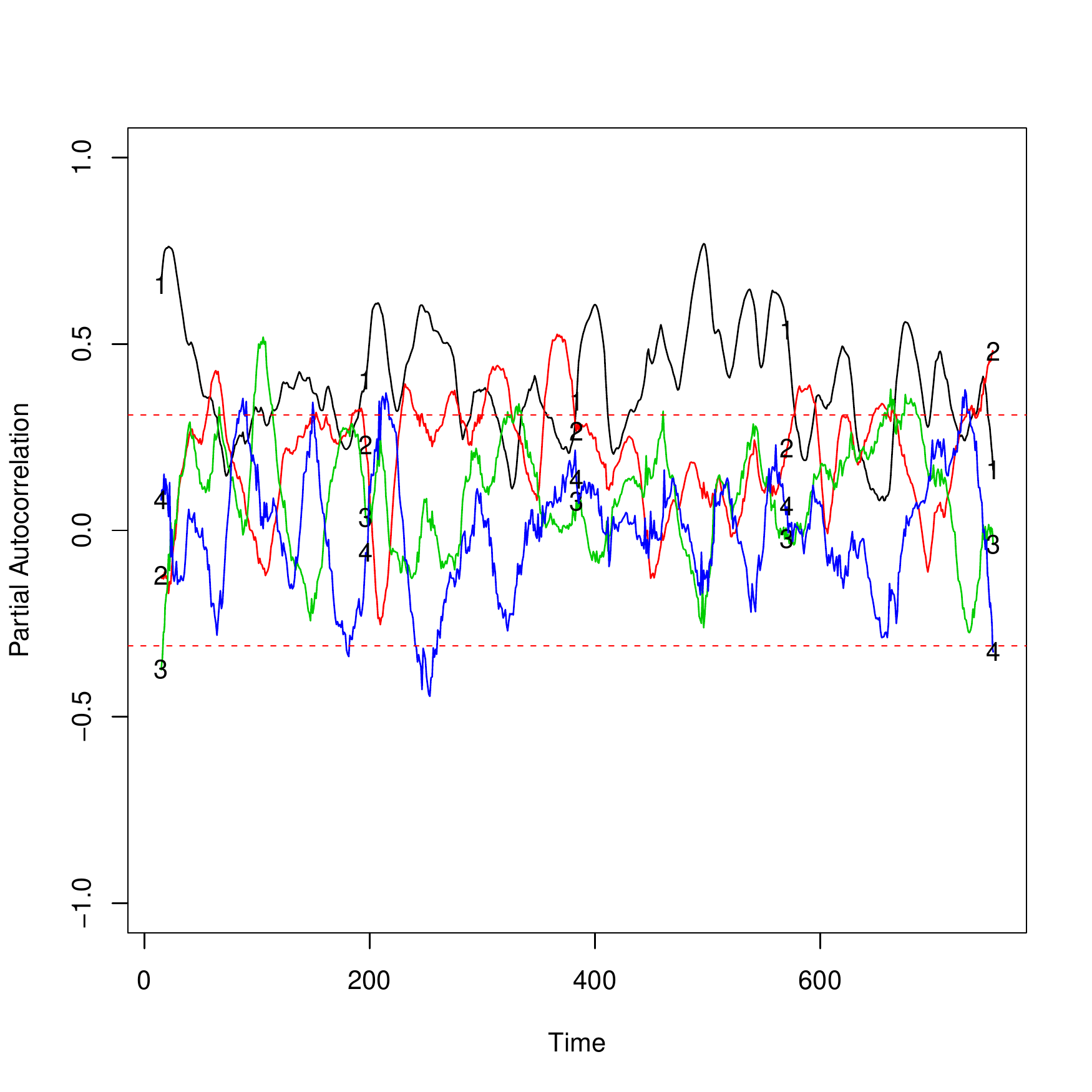}\label{fig:eastport40}}
  \end{center}
    \caption{Windowed partial autocorrelation of Eastport Precipitation Data,
$\tilde{q}_W(z, \tau)$, of left for lags one to four indicated on each curve.
Horizontal red dotted
lines are approximate 95\% confidence intervals. All plots were created with the Epanechnikov window with sizes a)
$L=160$, b) $L=80$, c) $L=40$.}\label{fig:eastportWindows}
\end{figure}

\section{Proofs from Section~\ref{sec:2}.}

  \subsection{Proof of Proposition \ref{prop:lpacf}}\label{app:prop:lpacf}
\begin{proof}
We shall use the notation $P_{[zT],\tau}(X_{{[zT]},T})=\hat{X}^{(b)}_{[zT],T}$ and
$P_{[zT],\tau}(X_{{[zT]+\tau},T})=\hat{X}^{(f)}_{[zT]+\tau,T}$, since these are the linear
predictors of $X_{{[zT]},T}$ ({\em b}ack-casted), respectively $X_{{[zT]+\tau},T}$ ({\em
f}orecasted), using the set of predictors $X_{[zT]+1,T}, \ldots, X_{[zT]+\tau-1,T}$.

Decomposing the projection space $\overline{\mbox{sp}}(X_{[zT],T}, \ldots, X_{[zT]+\tau-1,T})$ into
$\overline{\mbox{sp}}(X_{[zT]+1,T},$ $\ldots,$ \\$X_{[zT]+\tau-1,T})$ and its orthogonal complement,
we
can also write $\hat{X}_{[zT]+\tau,T}$ as
\begin{align}
\hat{X}_{[zT]+\tau,T}=\hat{X}^{(f)}_{[zT]+\tau,T}+P^{\perp}_{[zT],\tau}(X_{{[zT]+\tau},T}),
\label{eq:proj}
\end{align}
 where $P^{\perp}_{[zT],\tau}(\cdotp)$ denotes the projection onto the orthogonal complement space
above. Since this space is $\overline{\mbox{sp}}(X_{[zT],T}-\hat{X}^{(b)}_{[zT],T})$, it then
follows that $P^{\perp}_{[zT],\tau}(X_{{[zT]+\tau},T})=a(X_{{[zT]},T}-\hat{X}^{(b)}_{[zT],T})$, and
using equation \eqref{eq:proj} we obtain
$$
\Cov \left( \hat{X}_{[zT]+\tau,T},X_{{[zT]},T}-\hat{X}^{(b)}_{[zT],T} \right)= a \Var \left(
X_{{[zT]},T}-\hat{X}^{(b)}_{[zT],T} \right).
$$

Due to orthogonality of projection spaces, we have $\left(X_{{[zT]},T}-\hat{X}^{(b)}_{[zT],T}\right)
\perp \overline{\mbox{sp}}(X_{[zT]+1,T},$ $\ldots,$ $X_{[zT]+\tau-1,T})$ and from the equation above
and
equation \eqref{eq:calclpacf}, it follows that $a=\phisT$.

Hence $P^{\perp}_{[zT],\tau}(X_{{[zT]+\tau},T}))=\phisT (X_{{[zT]},T}-\hat{X}^{(b)}_{[zT],T})$ and
we obtain
\beqa \nonumber
\phisT&=&\frac{\Cov\left(P^{\perp}_{[zT],\tau}(X_{{[zT]+\tau},T}))
,X_{{[zT]},T}-\hat{X}^{(b)}_{[zT],T} \right)}{\Var\left( X_{{[zT]},T}-\hat{X}^{(b)}_{[zT],T}
\right)}\\ \nonumber
&=&\frac{\Cov\left( {X}_{[zT]+\tau,T},X_{{[zT]},T}-\hat{X}^{(b)}_{[zT],T} \right)}{\Var\left(
X_{{[zT]},T}-\hat{X}^{(b)}_{[zT],T} \right)},
\eeqa
as
$\Cov(X_{{[zT]+\tau},T}-P^{\perp}_{[zT],\tau}(X_{{[zT]+\tau},T}),Y)=0,\forall Y
\in\overline{\mbox{sp}}(X_{[zT],T}-\hat{X}^{(b)}_{[zT],T})$. Hence,
\begin{equation}\label{eq:lpacf1}
 \phisT =\frac{\Cov\left(
{X}_{[zT]+\tau,T}-\hat{X}^{(f)}_{[zT]+\tau,T},X_{{[zT]},T}-\hat{X}^{(b)}_{[zT],T}
\right)}{\Var\left( X_{{[zT]},T}-\hat{X}^{(b)}_{[zT],T} \right)}
\end{equation}
because $\hat{X}^{(f)}_{[zT]+\tau,T} \perp \overline{\mbox{sp}}(X_{[zT],T}-\hat{X}^{(b)}_{[zT],T})$.

Recall from equation~\eqref{eq:deflpacf} that
\begin{displaymath}
\pT{z}{\tau}=
\cor{X_{{[zT]+\tau},T}-P_{[zT],\tau}(X_{{[zT]+\tau},T}),X_{[zT],T}-P_{[zT],\tau}(X_{{[zT]},T})},
\end{displaymath}
or
equivalently $\pT{z}{\tau}=
\cor{X_{{[zT]+\tau},T}-\hat{X}^{(f)}_{[zT]+\tau,T},X_{[zT],T}-\hat{X}^{(b)}_{[zT],T}}$
which, combined with equation~\eqref{eq:lpacf1}, yields
$$
\pT{z}{\tau}= \phisT \left\{\frac{\Var\left( X_{{[zT]},T}-\hat{X}^{(b)}_{[zT],T} \right)}
{\Var\left( X_{{[zT]+\tau},T}-\hat{X}^{(f)}_{[zT]+\tau,T} \right)} \right\}^{1/2}
$$
as desired.
\end{proof}

 \subsection{Proof of Proposition \ref{prop:lpacf2}}\label{app:prop:lpacf2}

\begin{proof}

As $\hat{X}^{(b)}_{[zT],T}$ and $\hat{X}^{(f)}_{[zT]+\tau,T}$ are projections of
$X_{{[zT]},T}$ and $X_{{[zT]+\tau},T}$, respectively, on the space  $\overline{\mbox{sp}}(X_{[zT]+1,T},
\ldots, X_{[zT]+\tau-1,T})$ it follows that
$\E ( X_{{[zT]},T}-\hat{X}^{(b)}_{[zT],T} )=0$ and $\E ( X_{{[zT]+\tau},T}-\hat{X}^{(f)}_{[zT]+\tau,T} )=0$.
Hence,
the numerator and denominator in~\eqref{eq:deflpacf2} can  be
re-expressed as a Mean Squared Prediction Error (MSPE), since
\beqann
\Var ( X_{{[zT]},T}-\hat{X}^{(b)}_{[zT],T} )&=& \E (\hat{X}^{(b)}_{[zT],T}-X_{{[zT]},T})^2 = \mbox{MSPE}(\hat{X}^{(b)}_{[zT],T},X_{{[zT]},T}),\\
\Var ( X_{{[zT]+\tau},T}-\hat{X}^{(f)}_{[zT]+\tau,T} )&=& \E (\hat{X}^{(f)}_{[zT]+\tau,T}-X_{{[zT]+\tau},T})^2\\
&=& \mbox{MSPE}(\hat{X}^{(f)}_{[zT]+\tau,T},X_{{[zT]+\tau},T}).
\eeqann

Using these expressions we can rewrite $\pT{z}{\tau}$ from~\eqref{eq:deflpacf2} as
\beq\nonumber
\pT{z}{\tau}= \phisT \left\{ \frac
{\mbox{MSPE}(\hat{X}^{(b)}_{[zT],T},X_{{[zT]},T})}
{\mbox{MSPE}(\hat{X}^{(f)}_{[zT]+\tau,T},X_{{[zT]+\tau},T})} \right\}^{1/2}. 
\eeq

Now use the fact that the MSPE of a linear predictor of $X_{t,T}$ can be written as

\beq\nonumber
\mspeA{X}{{t,T}}=\mspeeA{X}{{t,T}}=\bm{b}_t^T \Sigma_{t,T} \bm{b}_{t},
\eeq
where $\bm{b}_t = ( b_{t-1,T},  \ldots, b_{0,T}, -1)^T$ and $\Sigma_{t,T}$ is the covariance  of $X_{0,T},\ldots,X_{t,T}$ \citep[Section 3.3]{fryz03:forecasting}. In our case, the back-casted and forecasted values of $X_{{[zT]},T}$ and $X_{{[zT]+\tau},T}$
are also linear predictors using the window of observations $X_{[zT]+1,T}, \ldots, X_{[zT]+\tau-1,T}$, and thus their corresponding MSPE can be expressed as
\begin{align}
\mbox{MSPE}(\hat{X}^{(b)}_{[zT],T},X_{{[zT]},T}) &= (\bb)^T \sigb \bb, \label{eq:mspeb}\\
\mbox{MSPE}(\hat{X}^{(f)}_{[zT]+\tau,T},X_{{[zT]+\tau},T}) & =  (\bf)^T \sigf \bf, \label{eq:mspef}
\end{align}
where, as above, the $\tau \times 1$ coefficient vectors are
$\bb= ( -1,  \tilde{b}^{(b)}_{1,T}, \ldots,  \tilde{b}^{(b)}_{{\tau-1},T})^T$ and
$\bf= ( {b}^{(f)}_{\tau-2,T}, \ldots, {b}^{(f)}_{0,T},  -1 )^T$ and the $\tau \times \tau$ covariance matrices $\sigb$ and $\sigf$ appear in Appendix~\ref{sec:covmat}.

Therefore, on combining equation~\eqref{eq:deflpacf3} with~\eqref{eq:mspeb} and~\eqref{eq:mspef}
we obtain as desired
	\begin{displaymath}
		\pT{z}{\tau} = \phisT \left\{ \frac{(\bb)^T\sigb\bb}{(\bf)^T\sigf\bf} \right\}^{1/2}.
	\end{displaymath}
\end{proof}

\subsection{Proof of Proposition \ref{prop:pTp}}\label{app:prop:pTp}
\begin{proof}
  The proof treats the convergence of $\phisT$ and the quotient that forms the square-root in
equation~\eqref{eq:deflpacf3} separately.  Firstly, we address the quotient convergence.

\underline{A: Quotient Convergence:}
By Proposition~3.1 from \cite{fryz03:forecasting} we have
\begin{displaymath}
  \mspeebA{X}{[zT],T} = (\bb)^T\sigb\bb = \left\{ (\bb)^T\Bb\bb\right\} \left\{ 1+\otone\right\}
  \end{displaymath}
  and
  \begin{eqnarray*}
	\mspeefA{X}{{[zT]+\tau,T}} &=& (\bf)^T\sigf\bf \\
		&=& \left\{ (\bf)^T\Bf\bf\right\}\left\{1+\otone\right\}.
\end{eqnarray*}
Hence
\begin{displaymath}
  \frac{\mspeebA{X}{[zT],T}}{\mspeefA{X}{[zT]+\tau,T}} =
\frac{\left\{ (\bb)^T\Bb\bb\right\}}{ \left\{ (\bf)^T\Bf\bf\right\}}  \left\{1+\otone\right\}
\end{displaymath}

\underline {B: Convergence of  $\phisT$}.
We defined $\phisT$ as the last element in  $\PhisT$ that is the solution to the
Yule-Walker equations $\sig\PhisT=\rhsT$.

Let $[\OTinv]$ be an appropriately-sized matrix whose elements are all $\OTinv$ and $\bm{\OTinv}$ be
a similarly defined vector.  

Now consider
\begin{align}
  \sig\left(\PhisT-\PPhis\right)&=\sig\PhisT - \sig\PPhis \nonumber\\
	&= \rhsT-\left(\BT+[\OTinv]\right)\PPhis \nonumber\\
	&= \rhsT-\rhs_{[zT]}-[\OTinv]\PPhis \tag{as $|c_T(z,\tau)-c(z,\tau)|=\OTinv$}\\
	&= \bm{\OTinv} - [\OTinv]\PPhis. \nonumber
\end{align}

Observe that $[\OTinv]\PPhis=\left(\sum_{j}\sum_i K_{j,i}\ph_{[zT],\tau,j}/T\right)\bm{1}$, where
$K_{j,i}$ is the $(j,i)$th constant, and in what follows we shall seek to bound this
quantity.

From the Cauchy-Schwarz inequality $\|\PPhis\|_1 \leq |\tau|^{1/2}\|\PPhis\|_2=C_{\tau}$ as $\tau$
is fixed, and by standard properties of the spectral norm
\begin{align*}
\|\PPhis\|_2^2 &= \PPhis^T\PPhis \\&\leq
\PPhis^T\sig\PPhis\|\sig^{-1}\|\\&=\mspeA{X}{{[zT]+\tau-1,T}}\|\sig^{-1}\|<\infty,
\end{align*}
as the spectral
norm $\|\sig^{-1}\|$ is bounded using Lemma A.3 from~\cite{fryz03:forecasting}.

Thus $[\OTinv]\PPhis=\bm{\OTinv}$ and it follows that
\begin{displaymath}
\sig\left(\PhisT-\PPhis\right)=\bm{\OTinv},
\end{displaymath}
which is equivalent to  $\PhisT-\PPhis=\sig^{-1}\bm{\OTinv}$.  By bounds of Rayleigh quotients
\citep[pg.181]{AbadirMagnus2005},
$\| \PhisT-\PPhis \|_2$ $=\|\sig^{-1}\bm{\OTinv}\|_2$ $\leq \mu^{1/2}\|\bm{\OTinv}\|_2$ $=
(\mu |\tau|)^{1/2}/T$ $=\OTinv$ as $\tau$ is fixed.  Here $\mu$ is the largest eigenvalue of
$\left(\sig^{-1}\right)^T\sig^{-1}$, i.e.
$\mu=\|\sig^{-1}\|^2$, and so $\mu<\infty$.  It follows that $\PhisT-\PPhis=\bm{\OTinv}$.

\underline{Putting parts A and B together:}
\begin{align*}
  \left|\pT{z}{\tau}-\p{z}{\tau}\right| &= \left| \phisT
\left\{ \frac{\mspeebA{X}{{[zT],T}}}{\mspeefA{X}{[zT]+\tau,T}} \right\}^{1/2} \right. \\
 & \left.  - \phis
\left\{ \frac{(\bb)^T\Bb\bb}{(\bf)^T\Bf\bf} \right\}^{1/2}\right| \\
	\begin{split}&= \left| \left\{\phis+\OTinv\right\} \right. \\
	& \left.
\left[ \frac{\left\{ (\bb)^T\Bb\bb\right\} \left\{ 1+\otone\right\}}{
\left\{ (\bf)^T\Bf\bf\right\} \left\{ 1+\otone\right\}} \right]^{1/2}  \right.\\ &\qquad \left.- \phis
\left\{\frac{(\bb)^T\Bb\bb}{(\bf)^T\Bf\bf} \right\}^{1/2} \right| \end{split}
\\
	\begin{split}&= \left|
\phis\left\{ \frac{(\bb)^T\Bb\bb}{(\bf)^T\Bf\bf)}\right\}^{1/2}\right.\\
& \left. \left(\left[ \frac{\left\{ 1+\otone\right\}} {
\left\{ 1+\otone\right\}} \right]^{1/2}-1\right) \right.\\ &\qquad \left.+ \OTinv
\left\{ \frac{\mspeebA{X}{[zT],T}}{\mspeefA{X}{[zT]+\tau,T}}\right\}^{1/2} \right| \end{split}
\\
	&= \OTinv.
\end{align*}
For the last equality the first term is asymptotically zero since
$\frac{\left\{1+\otone\right\}}{\left\{ 1+\otone\right\}} \to 1$ as $T \to \infty$,
$\phis<\infty$ and
\begin{displaymath}
\left\{ (\bb)^T \Bb \bb \right\}^{1/2} \left\{ (\bf)^T\Bf\bf \right\}^{-1/2} < \infty,
\end{displaymath}
 or, more concisely,
$\p{z}{\tau}<\infty$.  The second term  is
$\OTinv$ as each expectation is finite.  This concludes the
proof.
\end{proof}

\subsection{Proof of Proposition \ref{prop:ptildep}}\label{app:prop:ptildep}
\begin{proof}
First recall that we defined the local partial autocorrelation as
\begin{displaymath}
 \p{z}{\tau}=\ph_{[zT],\tau,\tau;T}\left\{\frac
{\mbox{Var} \{ X_{[zT],T}-P_{[zT],\tau}(X_{{[zT]},T}) \}}
{\mbox{Var}\{X_{{[zT]+\tau},T}-P_{[zT],\tau}(X_{{[zT]+\tau},T})\}}\right\}^{1/2},
\end{displaymath}
where the coefficient $\ph_{[zT],\tau,\tau;T}$ is obtained in a manner akin to the (stationary) partial
autocorrelation coefficient by expressing $X_{[zT]+\tau,T}$ as an $\mbox{AR}(\tau)$ process and
solving the associated Yule-Walker equations. The fraction under the square root quantifies the
ratio between the backward and forward variances associated to the $\mbox{AR}(\tau)$ process.
The Yule-Walker equations here are localized at the rescaled time $z$, in the sense that they
involve observations over the interval $\left[ [zT], [zT]+\tau \right]$.

Recall that, in estimating the local partial autocorrelation, we use the
$\hat{c}(z,\tau)$
estimator of \cite{nason00:wavelet}, which was shown there to be consistent for the (true) local
autocovariance ${c}(z,\tau)$. By the classical stationary theory, it follows that the estimated
Yule-Walker coefficients of the $\mbox{AR}(\tau)$ process (solution vector to the local
Yule-Walker equations) are consistent estimators of the true coefficients, hence
$\tilde{\ph}_{[zT],\tau,\tau;T} \overset{P}{\longrightarrow} \ph_{[zT],\tau,\tau}$, and the forward
and backward variances are also estimated consistently.

Using the continuous mapping theorem \citep{billingsley1999convergence} and assuming that the
variance is non-zero, it follows that the
square-root of the ratio of estimated backward and forward variances
\begin{displaymath}
\left\{\frac{(\bbtilde)^T\Bbtilde\bbtilde}{(\bftilde)^T\Bftilde\bftilde}\right\}^{1/2}
\end{displaymath}
is a consistent estimator
of the true ratio of variances
\begin{displaymath}
\left\{\frac{(\bb)^T\Bb\bb}{(\bf)^T\Bf\bf} \right\}^{1/2}.
\end{displaymath}
This together with the
consistency of $\tilde{\ph}_{[zT],\tau,\tau;T}$, yields
\begin{equation}
\begin{split}
\tilde{\ph}_{[zT],\tau,\tau;T}\left\{ \frac{(\bbtilde)^T\Bbtilde\bbtilde}{
(\bftilde)^T\Bftilde\bftilde } \right\}^{1/2} \overset{P}{\longrightarrow}\\ \phis
\left\{\frac{(\bb)^T\Bb\bb}{(\bf)^T\Bf\bf}\right\}^{1/2}.
\end{split}
\end{equation}
\end{proof}

\section{Proofs from Section~\ref{sec:3}.}\label{app:sec3proofs}

\subsection{Cross-scale autocorrelation Haar wavelets}

\subsubsection{Proof of Proposition \ref{prop:psilj}}\label{suppmat:proof:psilj}

\begin{proof}
For completeness, the definition of
(continuous-time) Haar wavelets is
\begin{equation}\nonumber
\psi_H (x) = \begin{cases}
	-1 & \mbox{if } 0 \leq x < 1/2,\\
	1 & \mbox{if } 1/2 \leq x < 1,\\
	0 & \mbox{otherwise.}
	\end{cases}
\end{equation}
 \cite{nason00:wavelet} show that $\Psi_j (\tau) = \Psi_H (2^{-j} | \tau
|)$ where
$\Psi_j(\tau)$ is the regular discrete autocorrelation wavelet and $\Psi_H(u)$ is the continuous
Haar autocorrelation
wavelet given by
\begin{equation}\nonumber
\Psi_H (u) = \int_{-\infty}^\infty \psi_H (x) \psi_H (x-u) \, dx =\
\begin{cases}
1 - 3 |u| & \mbox{for } |u| \in [0, \frac{1}{2}],\\
|u| - 1 & \mbox{for } |u| \in (\frac{1}{2}, 1].
\end{cases}
\end{equation}
Hence, we can derive the following integral equation for $\Psi_j (\tau)$ for $\tau \geq 0$:
\begin{equation}\nonumber
\Psi_j (\tau) = \Psi_H (2^{-j} \tau)
= \int_{-\infty}^\infty \psi_H(x) \psi_H(x - 2^{-j} \tau) \, dx
\end{equation}
from the definition of $\Psi_H(u)$. Then make the substitution $x = 2^{-j}y$ to obtain:
\begin{align*}
\Psi_j (\tau) &= \int_{-\infty}^\infty  \psi_H ( 2^{-j} y)  \psi_H (2^{-j} y - 2^{-j} \tau) \,
2^{-j} \, dy\\
&= \int_{-\infty}^\infty  2^{-j/2} \psi_H( 2^{-j} y) \, 2^{-j/2} \psi_H  \{ 2^{-j} (y - \tau) \} \,
dy\\
&=\int_{-\infty}^\infty  \psi_{j, 0}(y) \psi_{j, 0} ( y- \tau) \, dy,
\end{align*}
where $\psi_{j, 0} (y) = 2^{-j/2} \psi_H (2^{-j}y)$.

Hence, by a similar argument it is the case that
\begin{equation}\nonumber
\Psi_{j, \ell} (\tau) = \int_{-\infty}^\infty  \psi_{j, 0}(y) \psi_{\ell, 0} ( y- \tau) \, dy.
\end{equation}

For Haar wavelets, since we know the precise form of $\psi_H$ we should be able to obtain an
analytical formula
for $\Psi_{j, \ell}$. To do this we consider $\ell < j$ and see that
\begin{equation}\nonumber
\Psi_{j, \ell} (\tau) = \int_{-\infty}^\infty  2^{-j/2} \psi_H (2^{-j} y) 2^{-\ell/2} \psi_H \{
2^{-\ell} (y - \tau) \} \, dy.
\end{equation}
Now let $x = 2^{-\ell} y$ and we obtain
\begin{align*}
\Psi_{j, \ell} (\tau) &= \int_{-\infty}^\infty  2^{-j/2} \psi_H \{ 2^{-(j-\ell)} x \} 2^{\ell/2}
\psi_H(x - 2^{-\ell} \tau) \, dx\\
&= \int_{-\infty}^\infty  \psi_{j - \ell} (x) \psi(x - 2^{-\ell} \tau) \, dx.
\end{align*}
Hence, it makes sense to introduce the following {\em core} function:
\begin{equation}\nonumber
\Omega_i (u) = \int_{-\infty}^\infty  \psi_i (x) \psi (x- u) \, dx,
\end{equation}
for integers $i=0, 1, 2, \ldots$.
Clearly,
\begin{align}
\Psi_{j, \ell} (\tau) = \Omega_{j - \ell} ( 2^{-\ell} \tau),\label{eq:psiinomega}
\end{align}
for $\ell < j$. Also $\Omega_0 (u) = \Psi_H (u)$.

Using Lemma~\ref{lem:omegai} and \eqref{eq:psiinomega}, we can now specify an exact formula for
$\Psi_{j, \ell}(\tau)$. For $\ell < j$ the result is shown in~\eqref{eq:HaarXCorr}. Corollary~\ref{corr:psiljREMAIN}
shows the formula for $\ell > j$.

Proposition~\ref{prop:psilj} shows that the support of the cross-correlation wavelet is
$\{ k \in \ints: -2^{\ell} <= k < 2^j \}$ for $\ell < j$.
\end{proof}

\subsection{Subsidiary result used in the proof of Lemma~\ref{lem:boundi}}
\label{suppmat:intres}

\subsubsection{Proof of Lemma \ref{lem:intres}}\label{proof:lem:intres}

\begin{proof}
The result is obtained by combining known results on the Fej\'{e}r and Dirichlet
kernels as follows. The Fej\'{e}r kernel can be defined by:
\begin{equation}\nonumber
F_n (\omega) = \frac{ \sin^2 ( n \omega/ 2) }{2\pi n \sin^2 (\omega/2) } = \frac{ 1 - \cos(n
\omega)}{2\pi n \{ 1 - \cos(\omega) \}},
\end{equation}
for $\omega \in [-\pi, \pi]$, see \cite{walter00:wavelets} Section~4.2, for example. The Fej\'{e}r kernel
can
also be written in the following alternative form
\begin{equation}\label{eq:Fej2Dir}
F_n (\omega)  = \frac{1}{n} \sum_{k=0}^{n-1} D_k (\omega),
\end{equation}
where $D_k (\omega) = \pi^{-1} \left( \frac{1}{2} + \sum_{p=1}^k \cos p \omega \right)$ is the
Dirichlet kernel,
see Section 1.2.1 of \cite{walter00:wavelets}.

Let the integral on the left-hand side of~\eqref{eq:I1} be $I1$.
Then:
\begin{align*}
I1 &= 2\pi  \int_{-\pi}^\pi \{ 1 - \cos (2a \omega) \} 2b  F_{2b} (\omega) \, d\omega\\
&= 4\pi b \left\{ \int_{-\pi}^\pi F_{2b} (\omega) \, d\omega - \int_{-\pi}^\pi \cos(2 a \omega)
F_{2b} (\omega) \, d\omega \right\}\\
&= 4\pi b \left\{ 1 - \int_{-\pi}^\pi \cos(2a \omega)  (2b)^{-1} \sum_{k=0}^{2b-1} D_k (\omega) \,
d\omega \right\},
\end{align*}
by substituting~\eqref{eq:Fej2Dir} and
since $\int_{-\pi}^\pi F_n(\omega) \, d\omega = 1$.
Then
\begin{align}\nonumber
I1 &= 4\pi b \left\{ 1 - (2b)^{-1} \sum_{k=0}^{2b-1} \int_{-\pi}^\pi \cos(2a \omega) D_k (\omega)
\, d\omega
	\right\} \nonumber\\
&= 4\pi b \{ 1- (2b)^{-1} \sum_{k=0}^{2b-1} I2_k \}, \label{eq:cD2}
\end{align}
where $I2_k = \int_{-\pi}^\pi D_k (\omega) \cos( 2a \omega) \, d\omega$.
Now,
\begin{align}
I2_k	&= \pi^{-1} \int_{-\pi}^\pi \left( \frac{1}{2} + \sum_{p=1}^k \cos p \omega \right)
		\cos(2 a \omega) \, d\omega \nonumber\\
		&=(2  \pi)^{-1} \int_{-\pi}^\pi \cos(2a \omega) \, d\omega +
			\pi^{-1} \sum_{p=1}^k \int_{-\pi}^\pi \cos (p\omega) \cos (2a \omega)
			\, d\omega\label{eq:ccint}\\
		&= (2\pi)^{-1} \left[ \frac{ \sin(2a \omega) }{2a} \right]_{-\pi}^\pi
			+ \pi^{-1} \sum_{p=1}^k
			\left[ \frac{2 a \cos (p \omega) \sin (2 a\omega) - p \cos( 2a \omega)
				\sin(p \omega) }{4a^2  - p^2} \right]_{-\pi}^\pi \nonumber\\
		&= 0,
\end{align}
for $p^2 \neq 4a^2$ for $p \neq 2a$ (and recall $a>0$).
For $p=2a$ the integral in~\eqref{eq:ccint} is
\begin{equation}
\int_{-\pi}^\pi \cos^2 ( 2 a \omega) \, d\omega = \pi + \sin( 4 a \pi)/4a = \pi,
\end{equation}
since $2a \in \nats$. Hence,
\begin{equation}
\label{eq:I2k}
I2_k = \begin{cases}
0 & \mbox{for } k < 2a\\
1 & \mbox{for } k \geq 2a.
\end{cases}
\end{equation}
Hence, substituting~\eqref{eq:I2k} into~\eqref{eq:cD2} gives, for $b > a$
\begin{equation}
I1= 4\pi b \{ 1 - (2b)^{-1} (2b - 2a) \}\\
= 4\pi b - 2\pi  (2b - 2a)\\
= 4 \pi a.
\end{equation}
Since the integral~\eqref{eq:I1} is symmetric in $a$ and $b$ we also have
$I1 = 4 \pi b$ for $b \leq a$. Hence, the result in equation (\ref{eq:I1}) follows.
\end{proof}

\subsection{Proof of Lemma~\ref{lem:boundi}}
\label{app:boundi}

\begin{proof}
 It is obvious that inequality~\eqref{eq:maineq} holds when $i_{N, z} (j, \ell, k) = 0$. This occurs
when the lower limit in the sum~\eqref{eq:iNz}, plus the extra $k - 2[zT] + N/2-1$ exceeds the
support of $\psi_{\ell, \cdot}$. In other words, $i_{N, z} (j, \ell, k) = 0$ when:
\begin{align}
 & [zT] - N + 1 + k - 2[zT] + N/2 - 1 > N_\ell - 1 \nonumber\\
&\implies  k - [zT] -N/2  > N_\ell - 1 \nonumber\\
&\implies k > [zT] + N/2 + N_\ell - 1 = b_2.
\end{align}
It can also be shown that $i_{N,z} (j, \ell, k) = 0$ when $k < [zT] -N/2 + 1$ but this inequality is
not of interest in this proof
.

For the inequalities in~\eqref{eq:b1} we decompose $\Psi$ into three terms:
\begin{equation}
\label{eq:psiineq2}
\Psi_{j, \ell} (k - 2[zT] + N/2 - 1) = L + i_{N, z} (j, \ell, k) + U,
\end{equation}
where
\begin{equation}
\label{eq:sumL}
L = \sum_{s = -\infty}^{[zT] - N} \psi_{j, s} \psi_{\ell, s+k - 2[zT] + N/2-1}
\end{equation}
and
\begin{equation}
\label{eq:sumU}
U = \sum_{s=[zT]+1}^\infty \psi_{j, s} \psi_{\ell, s+k - 2[zT] + N/2-1}.
\end{equation}
Clearly, the inequality~\eqref{eq:maineq} is satisfied when $i_{N,z} (j, \ell, k) = \Psi_{j, \ell}
(k - 2[zT] + N/2 - 1)$ which
occurs when $L = U = 0$. We now investigate the conditions when $L=U=0$.

(A) When is $U=0$? When the lower limit of the sum defining $U$ in~\eqref{eq:sumU} exceeds the
support of
$\psi_{j, \cdot}$, i.e.\
\begin{equation}
\label{eq:A1}
[zT] + 1 > N_j - 1 \implies [zT] > N_j - 2,
\end{equation}
or when the lower limit exceeds the support of $\psi_{\ell, \cdot}$, i.e.\
\begin{equation}
\label{eq:A2}
[zT] + 1 + k  - 2[zT] + N/2 -1 > N_\ell - 1 \implies k > [zT] - N/2 + N_\ell - 1.
\end{equation}

(B) When is $L=0$? When the upper limit of the sum defining $L$ in~\eqref{eq:sumL} is less than the
lower support
bound of $\psi_{j, \cdot}$, which is zero, i.e.\
\begin{equation}
\label{eq:B1}
[zT] - N  < 0 \implies [zT] < N,
\end{equation}
or when the upper limit is less than the support of $\psi_{\ell, \cdot}$, i.e.\
\begin{equation}
\label{eq:B2}
[zT] - N + k - 2[zT] + N/2 - 1  < 0 \implies   k < [zT] + N/2 + 1 = b_1.
\end{equation}

Hence, $U=L=0$ when inequalities~\eqref{eq:A1} and~\eqref{eq:B2} are satisfied. Note: we are not
particularly interested in inequalities~\eqref{eq:A2} and~\eqref{eq:B1}. For the former, the
inequality~\eqref{eq:A2} would have to be allied with~\eqref{eq:B1} (as \eqref{eq:B2} would be
contradictory to~\eqref{eq:A2}) and, asymptotically \eqref{eq:B1} will not hold (as we expect the
rate of increase of $T$ to be much bigger than $N$).

So far we have demonstrated the Lemma up to inequalities~\eqref{eq:b1} and~\eqref{eq:b2} and now
we look to establish the second part of the Lemma.

To establish~\eqref{eq:part3} it can be shown that, for Daubechies' wavelets with two or more
vanishing moments,
\begin{align}
|i_{N, z} (j, \ell, k) | &\leq 2^{-(j+\ell)/2} K^2 \sum_{s=[zT]-N+1}^{[zT]} s^{-1} (s - [zT] +
N)^{-1}\label{eq:line1} \\
&= 2^{-(j+\ell)/2} \left( \mathcal{H}_{N} - \mathcal{H}_{[zT]} + \mathcal{H}_{[zT] - N}
\right)/([zT] - N)\nonumber,
\end{align} 
where $\mathcal{H}_n$ is the $n$th Harmonic number, and $K$ is a constant (maximum absolute value
of the
wavelet).
Now using the following approximation for $\mathcal{H}_n$
\begin{equation}\nonumber
\mathcal{H}_n = \log n + \gamma + \calO(n^{-1}),
\end{equation}
where $\gamma$ is the Euler-Mascheroni constant, we can obtain the result in~\eqref{eq:part3}.

Now we consider Haar wavelets. First, let us recall what the discrete Haar wavelet is. We have
\begin{equation}\nonumber
\psi_{j, k} =
\begin{cases}
2^{-j/2} & \mbox{for } 0 \leq k < N_j /2,\\
-2^{-j/2} & \mbox{for } N_j/2 \leq k < N_j,\\
0 & \mbox{otherwise.}
\end{cases}
\end{equation}
For Haar wavelets $N_j = 2^j$ for $j \in \nats$. Next we will require the discrete Fourier transform
of the Haar wavelet given by:
\begin{equation}\nonumber
\hat{\psi}_j (\omega) = \sum_{s= -\infty}^{\infty} \psi_{j, s} e^{-i\omega s},
\end{equation}
for $\omega \in (-\pi, \pi)$.
The inverse of this transform is:
\begin{equation}
\label{eq:fourinv}
\psi_{j, s} = ( 2\pi)^{-1} \int_{-\pi}^\pi \hat{\psi}_j (\omega) e^{i \omega s} d\omega
\end{equation}
for $s \in \ints$.

Now let us work out the precise form of the Fourier transform of the discrete Haar wavelet:
\begin{align*}
\hat{\psi}_j (\omega) &= 2^{-j/2} \left( \sum_{s=0}^{N_j /2 - 1} e^{-i\omega s} - \sum_{s =
N_j/2}^{N_j - 1} e^{-i\omega s} \right)\\
&= 2^{-j/2} \left\{ \sum_{s=0}^{N_j /2 - 1} e^{-i\omega s} - \sum_{s=0}^{N_j/2 - 1} e^{-i \omega(s
+ N_j /2)} \right\}\\
&= 2^{-j/2} \sum_{s=0}^{N_j/2 - 1} e^{-i\omega s} \left( 1 - e^{-i\omega N_j/2} \right)\\
&= 2^{-j/2} \left( 1 - e^{-i\omega N_j/2} \right) \sum_{s=0}^{N_j/2 - 1} e^{-i\omega s} \\
&= 2^{-j/2}  \left( 1 - e^{-i\omega N_j/2} \right)  \frac{ 1 - \exp ( -i\omega N_j /2 ) }{1 - \exp(
-i \omega) }\\
&= 2^{-j/2} \frac{ \left( 1 - e^{-i\omega N_j/2} \right)^2}{1 - \exp( -i \omega) }
\end{align*}

We now directly examine formula~\eqref{eq:iNzalt} with a rectangular kernel, as discussed in the main body of the paper. To simplify notation, we let $B = [zT] - N$ and $r=k - 2[zT] +N/2 - 1$.
In~\eqref{eq:iNzalt} replace the discrete wavelets $\psi_{j,s}$ and $\psi_{\ell, s+r}$ by their
Fourier inverse representations
given by~\eqref{eq:fourinv} to obtain:
\begin{align}
i_{N, z} (j, \ell, k) &= \sum_{s=B+1}^{B+N} (2\pi)^{-2} \int_{-\pi}^\pi \int_{-\pi}^\pi
\hat{\psi}_j (\omega) e^{i\omega s}
	\hat{\psi}_\ell (\nu) e^{i\nu (s+r)} d\omega d\nu\label{eq:returntoi}\\
	&= (2\pi)^{-1} \int_{-\pi}^\pi \hat{\psi}_\ell (\nu) e^{i \nu r} \left\{
		(2\pi)^{-1}\times \right.\\
		& \left. \hspace{3cm} \int_{-\pi}^\pi \hat{\psi}_j (\omega) \sum_{s=B+1}^{B+N} e^{i(\omega +\nu) s}
d\omega \right\} d\nu
			\\
		&= (2\pi)^{-1} \int_{-\pi}^\pi \hat{\psi}_\ell (\nu) e^{i \nu r} \left\{
		(2\pi)^{-1} \times \right. \\
		& \left.\hspace{3cm} \int_{-\pi}^\pi \hat{\psi}_j (\omega) e^{i(\omega +\nu)(B+1)}
			\frac{ 1 - e^{i(\omega+\nu)N}}{1 - e^{i(\omega+\nu)}} d\omega \right\} d\nu\nonumber\\
		&=(2\pi)^{-1} \int_{-\pi}^\pi \hat{\psi}_\ell (\nu) e^{i \nu r} G_{j, B, N} (\nu) d\nu,
\end{align}
where
\begin{align*}
G_{j, B, N} (\nu) &= (2\pi)^{-1} \int_{-\pi}^\pi \hat{\psi}_j (\omega) e^{i(\omega +\nu)(B+1)}
			\frac{ 1 - e^{i(\omega+\nu)N}}{1 - e^{i(\omega+\nu)}} d\omega\\
			&= \sum_{s= B+1}^{B+N} e^{i \nu s} \psi_{j, s}\\
			&= \hat{\psi}_j (\nu) - \sum_{s=-\infty}^B \psi_{j, s} e^{i\nu s} -
\sum_{s=B+N+1}^\infty \psi_{j,s} e^{i\nu s}.
\end{align*}
We now examine what happens to $G_{j, B, N}(\nu)$ under four different cases depending on how the
support of
the wavelet $\psi_{j, s}$ overlaps the interval $[B+1, B+N]$ or not. Note: the support of the
wavelet
is the interval $[0, N_j - 1]$. Note: we are mostly interested in the situation when $T, N$ are
large and hence not so
interested in a potential fifth case when $[B+1, B+N] \subseteq [0, N_j-1]$.

Case-I: Suppose $[0, N_j - 1] \subseteq [B+1, B+N]$. That is, the support of the wavelet lies
entirely within
the interval $[B+1, B+N]$. Then
\begin{equation}
G_{j, B, N} (\nu) = \sum_{s= B+1}^{B+N} e^{i \nu s} \psi_{j, s} = \sum_{s=0}^{N_j - 1} \psi_{j, s}
e^{i \nu s}
	= \hat{\psi}_j (\nu). \label{eq:GjBN1}
\end{equation}

Case-II: Suppose $0 < B+1$ but $B+1 < N_j - 1$, that is right-hand end of the wavelet support
overlaps
$[B+1, B+N]$ but the left-hand end does not. Then:
\begin{align}
G_{j, B, N} (\nu) &= \sum_{s=B+1}^{N_j - 1} \psi_{j,s} e^{i\nu s} = \hat{\psi}_j (\nu) -
\sum_{s=0}^B \psi_{j,s} e^{i \nu s}\\
&= \hat{\psi}_j (\nu) -  G_{j, B} (\nu),\label{eq:case2}
\end{align}
where
\begin{equation}
\label{eq:GjB}
G_{j, B} (\nu) = \sum_{s=0}^B \psi_{j,s} e^{i \nu s}.
\end{equation}

Case-III: Suppose $[0, N_j - 1]$ and $[B, B+N]$ do not overlap. Then $G_{j, B, N} (\nu) = 0$.

Case-IV: Suppose that $0 < B+N$ and $B+N < N_j - 1$, that is the left-hand end of the wavelet
support
overlaps $[B+1, B+N]$ but the right-hand end does not. And $B+1 < 0$ which is not of interest
as it means that $[zT] - N + 1 < 0$ which should not happen, for large $T$, as $T$ increases faster
than $N$.

Now let us derive $G_{j, B} (\nu)$ from~\eqref{eq:GjB} for $B> 0$:
\begin{equation}
\label{eq:GjB2}
G_{j, B} (\nu) =
\begin{cases}
 2^{-j/2} \sum_{s=0}^B e^{i \nu s} & \mbox{if } 0 \leq B \leq N_j/2 - 1,\\
 2^{-j/2} \left( \sum_{s=0}^{N_j/2 - 1} e^{i\nu s} - \sum_{s= N_j/2}^B e^{i\nu s} \right) & \mbox{if
} N_j/2-1 < B \leq N_j - 1,\\
  2^{-j/2} \sum_{s=0}^{N_j/2-1} \left\{ e^{i \nu s} - e^{i\nu (s + N_j/2) } \right\} & \mbox{if }
N_j-1 < B.
\end{cases}
\end{equation}
Computing the sums in~\eqref{eq:GjB2} gives:
\begin{equation}
\label{eq:GjB3}
G_{j, B} (\nu) = \frac{2^{-j/2}}{ \{ \exp(i\nu) - 1 \} }
\begin{cases}
 \exp \{ i \nu (B+1) \} - 1, & \text{for $B \in S_1$,}\\
 2\exp( i\nu N_j / 2) -\exp \{ i\nu (B+1) \} - 1, & \text{for $B \in S_2$},\\
  \{ 1 - \exp( i\nu N_j/2) \} \{ \exp(i\nu N_j/2) - 1\}, & \text{for $B \in S_3$},
\end{cases}
\end{equation}
where $S_1 = \{ B: 0 \leq B \leq N_j/2 - 1 \},
S_2 = \{ B: N_j/2 - 1 < B \leq N_j -1\},
S_3 = \{ B: N_j -1 < B \}$.

Now returning to the main formula~\eqref{eq:returntoi}.

Case-I: suppose $[0, N_j - 1] \subseteq [B+1, B+N]$ then $G_{j, B, N} (\nu) = \hat{\psi}_j(\nu)$ as
given
by~\eqref{eq:GjBN1}. Hence, substituting into~\eqref{eq:returntoi} gives:
\begin{align}
i_{N, z} (j, \ell, k) &= (2\pi)^{-1} \int_{-\pi}^\pi \hat{\psi}_\ell (\nu) e^{i \nu r} G_{j, B,
N}(\nu) \, d\nu\\
&= (2\pi)^{-1} \int_{-\pi}^\pi 2^{-\ell/2} \frac{ ( 1 - e^{-i \nu N_\ell/2})^2}{1 - e^{-i\nu}}
e^{i\nu r} 2^{-j/2}
	\frac{ ( 1- e^{-i \nu N_j /2} )^2}{1 - e^{-i \nu}} \, d\nu\nonumber\\
	&=  \frac{2^{-(j+\ell)/2}}{2\pi} \int_{-\pi}^\pi e^{i \nu r} \frac{ ( 1 - e^{-i \nu
N_\ell/2})^2 ( 1- e^{-i \nu N_j /2} )^2}{ (1 - e^{-i \nu})^2} \, d\nu \label{eqn:CaseIpremod}
\end{align}
We now bound $i_{N, z}(j, \ell, k)$ by the integral of the absolute value of its integrand, i.e.\
\begin{align*}
| i_{N, z} (j, \ell, k) | & \leq  \frac{2^{-(j+\ell)/2}}{2\pi} \int_{-\pi}^\pi
	\frac{ 2 \{ 1- \cos (N_\ell \nu/2) \} 2 \{ 1 - \cos(N_j \nu/ 2) \} }{ 2 \{1 - \cos (\nu) \} } \,
d\nu\nonumber \\
	&=  \frac{ 2^{1 -(j+\ell)/2}}{2\pi} \int_{-\pi}^\pi
	\frac{ \{ 1- \cos (N_\ell \nu/2) \}  \{ 1 - \cos(N_j \nu/ 2) \} }{  1 - \cos (\nu)  } \, d\nu,
\end{align*}
because $| e^{i\nu r}| = 1$, $| (1 - e^{-i\nu} )|^2 = 2\{ 1 - \cos (\nu) \}$, and so on. Using
Lemma~\ref{lem:intres}
with $a = N_\ell/4, b = N_j/4$ gives
\begin{align}
| i_{N, z} (j, \ell, k) | &\leq  2^{-(j+\ell)/2} \min(N_\ell, N_j)\\
&= \begin{cases}
2^{ - (j - \ell)/2} & \mbox{ for $\ell \leq j$},\\
2^{ - (\ell-j)/2} & \mbox{ for $\ell > j$},
\end{cases}\label{eq:case1end}
\end{align}
as $N_j = 2^j$ for Haar wavelets.

Case-IIa. Consider the case when $0 \leq B \leq N_j/2 - 1$. From~\eqref{eq:case2} we have $G_{j, B,
N} (\nu) = \hat{\psi}_j (\nu) - G_{j, B} (\nu)$. Hence:

\begin{align*}
i_{N, z} (j, \ell, k) &\leq \frac{1}{2\pi}
\int_{-\pi}^{\pi} \hat{\psi}_{\ell}(\nu) e^{i\nu r}G_{j,B,N}(\nu)\, d\nu \\
&=\frac{1}{2\pi}
\int_{-\pi}^{\pi} \hat{\psi}_{\ell}(\nu) e^{i\nu r}\left[\hat{\psi}_j(\nu)-G_{j,B}(\nu)\right]\,
d\nu \\
&=\frac{1}{2\pi}
\int_{-\pi}^{\pi} \hat{\psi}_{\ell}(\nu) e^{i\nu r}\hat{\psi}_j(\nu)\, d\nu \\
&  -  \frac{1}{2\pi}
\int_{-\pi}^{\pi} \hat{\psi}_{\ell}(\nu) e^{i\nu
r}\frac{2^{j/2}}{e^{i\nu}-1}\left[e^{i\nu(B+1)}-1\right]\, d\nu \\
&= \eqref{eqn:CaseIpremod} - \frac{2^{-(j+\ell)/2}}{2\pi}
\int_{-\pi}^{\pi} \frac{\left(1-e^{-i\nu N_\ell/2}\right)^2\left(1-e^{i\nu(B+1)}\right)e^{i\nu
r}}{\left(1-e^{i\nu}\right)\left(1-e^{-i\nu}\right)}\, d\nu \\
&\leq \eqref{eqn:CaseIpremod} - \frac{2^{-(j+\ell)/2}}{2\pi}
\int_{-\pi}^{\pi} \frac{\left(1-e^{-i\nu N_\ell/2}\right)^2\left(1-e^{i\nu N_j/2}\right)e^{i\nu
r}}{\left(1-e^{i\nu}\right)\left(1-e^{-i\nu}\right)}\, d\nu
\end{align*}
We now bound $i_{N, z}(j, \ell, k)$ by the integral of the absolute value of its integrand, i.e.\
\begin{align*}
| i_{N, z} (j, \ell, k) | &\leq \eqref{eq:case1end} + \frac{2^{-(j+\ell)/2}}{2\pi}
\int_{-\pi}^{\pi} \frac{|\left(1-e^{-i\nu N_\ell/2}\right)^2||\left(1-e^{i\nu
N_j/2}\right)||e^{i\nu
r}|}{|\left(1-e^{i\nu}\right)\left(1-e^{-i\nu}\right)|}\, d\nu \\
&\leq \eqref{eq:case1end} + \frac{2^{-(j+\ell)/2}}{2\pi}
\int_{-\pi}^{\pi} \frac{2\left(1-\cos(\nu N_\ell/2)\right) 2}{2\left(1-\cos(\nu)\right)}\, d\nu \\
& = \eqref{eq:case1end} +  \frac{2^{-(j+\ell)/2}}{\pi} \int_{-\pi}^{\pi} 2\pi (N_\ell / 2)
F_{N_\ell/2} (\nu) \, d\nu\\
& = \eqref{eq:case1end} + 2^{-(j+\ell)/2} N_\ell\\
&= 2^{-(j+\ell)/2} \left(\min(N_\ell, N_j)+N_\ell \right),
\end{align*}

Case-IIb. Consider the case when $N_j/2-1 < B \leq N_j - 1$. Again using~\eqref{eq:case2} we have
$G_{j, B, N} (\nu) = \hat{\psi}_j (\nu) - G_{j, B} (\nu)$ and from the corresponding value
of~\eqref{eq:GjB3}, we obtain (based on the same logic as above in Cases I and IIa):

\begin{align*}
i_{N, z} (j, \ell, k) &\leq \frac{1}{2\pi}
\int_{-\pi}^{\pi} \hat{\psi}_{\ell}(\nu) e^{i\nu r}G_{j,B,N}(\nu)\, d\nu \\
&=\frac{1}{2\pi}
\int_{-\pi}^{\pi} \hat{\psi}_{\ell}(\nu) e^{i\nu r}\left[\hat{\psi}_j(\nu)-G_{j,B}(\nu)\right]\,
d\nu \\
&=\eqref{eqn:CaseIpremod} - \frac{1}{2\pi} \int_{-\pi}^{\pi} \hat{\psi}_{\ell}(\nu) e^{i\nu
r}\frac{2^{j/2}}{e^{i\nu}-1}\left[2e^{i\nu N_j/2}-e^{i\nu(B+1)}-1\right]\, d\nu \\
&\leq \eqref{eqn:CaseIpremod} - \frac{1}{2\pi} \int_{-\pi}^{\pi} \hat{\psi}_{\ell}(\nu) e^{i\nu
r}\frac{2^{j/2}}{e^{i\nu}-1}\left[2e^{i\nu N_j/2}-e^{i\nu N_j}-1\right]\, d\nu \\
&= \eqref{eqn:CaseIpremod}\\
& +  \frac{2^{-(j+\ell)/2}}{2\pi}
\int_{-\pi}^{\pi} \frac{\left(1-e^{-i\nu N_\ell/2}\right)^2\left(2e^{i\nu N_j/2}-1-e^{i\nu
N_j}\right)e^{i\nu
r}}{\left(1-e^{i\nu}\right)\left(1-e^{-i\nu}\right)}\, d\nu
\end{align*}
We now bound $i_{N, z}(j, \ell, k)$ by the integral of the absolute value of its integrand, i.e.\
\begin{align*}
| i_{N, z} (j, \ell, k) | &\leq \eqref{eq:case1end}\\
& + \frac{2^{-(j+\ell)/2}}{2\pi}
\int_{-\pi}^{\pi} \frac{|\left(1-e^{-i\nu N_\ell/2}\right)^2||\left(1-2e^{i\nu
N_j/2}+e^{i\nu N_j}\right)||e^{i\nu
r}|}{|\left(1-e^{i\nu}\right)\left(1-e^{-i\nu}\right)|}\, d\nu \\
&\leq \eqref{eq:case1end} + \frac{2^{-(j+\ell)/2}}{2\pi}
\int_{-\pi}^{\pi} \frac{2\left(1-\cos(\nu N_\ell/2)\right) 2}{2\left(1-\cos(\nu)\right)}\, d\nu \\
& = \eqref{eq:case1end} +  \frac{2^{1-(j+\ell)/2}}{\pi} \int_{-\pi}^{\pi} \pi (N_\ell / 2)
F_{N_\ell/2} (\nu) \, d\nu\\
& = \eqref{eq:case1end} + 2^{-(j+\ell)/2} N_\ell\\
&= 2^{-(j+\ell)/2} \left(\min(N_\ell, N_j)+N_\ell \right),
\end{align*}

Case-III. Clearly, $| i_{N, z} (j, \ell, k) | \leq 0 \leq \eqref{eq:case1end}$ and Case-IV does not
apply.
\end{proof}

\subsection{Proof of Lemma~\ref{lem:order2l}}
\label{app:order2I}\label{app:orderlpm}
This lemma has two parts, hence we next prove the first part. 
\begin{proof}
  In what follows we use the two bounds~\eqref{eq:maineq} for $k < b_1$ and $k > b_2$ and~\eqref{eq:otherkbound} for
$b_1 \leq k \leq b_2$ from Lemma~\ref{lem:boundi}. Denote ${\cal B}= [b_1, b_2]$ and express
  \begin{align*}
    \sum_{k=-\infty}^\infty \sum_{j=1}^{\infty} \av i_{N,z}(j, \ell, k) \av^2 &=
\sum_{k\in{\cal B}}\sum_{j=1}^{\infty} \av i_{N,z}(j, \ell, k) \av^2 + \sum_{k\in\not{\cal
B}}\sum_{j=1}^{\infty} \av i_{N,z}(j, \ell, k) \av^2.
  \end{align*}
  For the first sum:
  \begin{align*}
    \sum_{k\in{\cal B}}\sum_{j=1}^{\infty} \av i_{N,z}(j, \ell, k) \av^2 &\leq \sum_{k\in{\cal
B}}\sum_{j=1}^{\infty} 2^{-(j+\ell)}\left(\min(N_\ell,N_j)+N_\ell\right)^2 \\
  &= \sum_{k\in{\cal B}}\left\{\sum_{j=1}^{\ell-1} 2^{-(j+\ell)}\left(N_j+N_\ell\right)^2 +
\sum_{j=\ell}^{\infty} 2^{-(j+\ell)}\left(2N_\ell\right)^2 \right\} \\
  &= 2^{-\ell}\sum_{k\in{\cal B}}\left\{\sum_{j=1}^{\ell-1} 2^{-j}\left(2^{2j} +
2^{2\ell}+2^{1+j+\ell}\right) +
4\sum_{j=\ell}^{\infty} 2^{-j}2^{2\ell} \right\} \\
  &= 2^{-\ell}\sum_{k\in{\cal B}}\left\{\sum_{j=1}^{\ell-1} 2^{j} +
2^{2\ell}\sum_{j=1}^{\ell-1}2^{-j} \right.\\
& \hspace{4cm} \left. + 2^{1+\ell}\sum_{j=1}^{\ell-1} 1 +
2^{2+2\ell}\sum_{j=\ell}^{\infty} 2^{-j} \right\} \\
  &= 2^{-\ell}\sum_{k\in{\cal B}}\left\{ 2^{\ell}-2+2^{2\ell}\left(1-2^{-\ell+1}\right) \right.\\
  & \hspace{4cm} \left. +
\left(\ell-1\right) 2^{1+\ell} + 2^{2+2\ell}2^{-\ell+1}\right\} \\
  &= \left(N_\ell-1\right)2^{-\ell}\left\{ \left(2\ell+5\right)2^{\ell} -2 +2^{2\ell}\right\} \\
  &= \left(N_{\ell}-1\right)\left( 2\ell+5-2^{-\ell+1}+2^{\ell}\right) \\
  &= \mathcal{O} \left(2^{2\ell}\right).
  \end{align*}
  For the second sum:
  \begin{align*}
    \sum_{k\in\not{\cal B}}\sum_{j=1}^{\infty} \av i_{N,z}(j, \ell, k) \av^2 &\leq
\sum_{k\in\not{\cal B}}\sum_{j=1}^{\infty} |\Psi_{j,\ell}\left(k-2[zT]+N/2-1\right)|^2 \\
&\leq \sum_j B_\ell(j,j),
  \end{align*}
where $B_{\ell} (j, p)$ is the fourth-order cross-correlation wavelet absolute value product of order $r=0$, defined as
$B^{(r)}_\ell (j, i) = \sum_{p=-\infty}^\infty |p|^r | \Psi_{j, \ell}(p) \Psi_{i, \ell} (p)|$
for $r=0, 1$ and scales $\ell, j, i \in \nats$.

Splitting the sum of $j$ and using
Proposition \ref{prop:B} leads to the following:
\begin{align*}
  \sum_j B_\ell(j,j) &= \sum_{j=0}^{\ell-1} B_\ell(j,j)+\sum_{j=\ell+1}^{\infty}
B_\ell(j,j)+B_\ell(\ell,\ell) \\
  &= \sum_{j=0}^{\ell-1} 2^{-\ell}\left(2^{2j-1}+1\right) +\sum_{j=\ell+1}^{\infty}
2^{-j}\left(2^{2\ell-1}+1\right)+\frac{1}{3}2^{-\ell}\left(2^{2\ell}+5\right) \\
  &= 2^{-\ell}\left[\frac{1}{2}\frac{1}{3} \left(2^{2\ell}-1\right)+\ell\right] +
\left(2^{2\ell-1}+1\right)2^{-\ell}+\frac{1}{3}2^{-\ell}\left(2^{2\ell}+5\right)\\
  &= \mathcal{O}\left(2^{\ell}\right).
\end{align*}
Hence $\sum_{k=-\infty}^\infty \sum_{j=1}^{\infty} \av i_{N,z}(j, \ell, k) \av^2 =
\calO(2^{2\ell})$.
\end{proof}

Now we prove the second part of Lemma \ref{lem:order2l}.
\begin{proof}
Consider
\begin{align*}
\sum_{k=-\infty}^\infty \sum_{n=-\infty}^\infty &\left\{\sum_{j=1}^{\infty}\av
i_{N,z}(j, \ell, k)  i_{N,z}(j,m,n) \av\right\}^2\\
&=\sum_k \sum_n \sum_{j=1}^\infty | i_{N,z} (j, \ell, k) i_{N,z}
(j, m, n) | \times \\
& \hspace{4cm}	\sum_{p=1}^\infty | i_{N,z} (p, \ell, k) i_{N,z} (p, m, n) | \nonumber\\
	&= \sum_{j, p = 1}^\infty \sum_k |i_{N,z} (j, \ell, k) i_{N,z} (p, \ell, k) |
		\sum_n |i_{N,z} (j, m, n) i_{N,z} (p, m, n) |\\
	\begin{split}&= \sum_{j, p = 1}^\infty  \left\{ \sum_{k \in {\cal B}}  |i_{N,z} (j, \ell, k)
i_{N,z} (p, \ell, k) | \right. \\
 & \left. \hspace{4cm} +
		\sum_{k \not\in {\cal B}} |i_{N,z} (j, \ell, k) i_{N,z} (p, \ell, k) | \right\}\nonumber\\
		&\quad \times\left\{ \sum_{n \in {\cal B}}  |i_{N,z} (j, m, n) i_{N,z} (p, m, n) |  \right. \\
	& \left. \hspace{4cm} +
\sum_{n \not\in {\cal B}}|i_{N,z} (j, m, n) i_{N,z} (p, m, n) | \right\}\nonumber\end{split}\\
	&= T_{{\cal B}{\cal B}} (j, \ell, p, m) +
		T_{{\cal B}{\not{\cal B}}} (j, \ell, p, m) +
		T_{\not{\cal B}{{\cal B}}} (j, \ell, p, m) +
		T_{\not{\cal B}{\not{\cal B}}} (j, \ell, p, m). \nonumber
\end{align*}
The term $T_{\not{\cal B}\not{\cal B}} = \sum_{k \not\in {\cal B}} \sum_{n \not\in {\cal B}}$ is the
case where $i_{N,z}$ can be
bounded by $\Psi$ and is addressed in detail in Lemma~\ref{lem:PsiBound}.
The term $T_{{\cal B}{\cal B}} = \sum_{k \in {\cal B}} \sum_{n \in {\cal B}}$ is dealt with in
Lemma~\ref{lem:bothB} using the bound~\eqref{eq:otherkbound} for  $i_{N, z}$
from Lemma~\ref{lem:boundi} and the cross term is dealt with in Lemma~\ref{lem:cross}.
Each of these lemmas (below) show that each of the product terms is of order no worse than
${\cal O} \{ 2^{2(\ell +m)} \}$.
\end{proof}

\subsection{Proof of Theorem~\ref{thm:JNphi}}\label{app:thm:JNphi}
\begin{proof}
  First recall that we are under the zero-mean locally stationary wavelet process framework as
described in
Appendix \ref{sec:defls}, with $\{ X_{t, T} \}_{t=0}^{T-1}$ a doubly-index stochastic process with
representation given by
\begin{equation}\nonumber
X_{t, T} = \sum_{j=1}^{\infty} \sum_{k=-\infty}^{\infty} w^{0}_{j, k; T} \psi_{j, k}(t) \xi_{j,k}.
\end{equation}

The integrated local periodogram was defined as
\begin{equation}\nonumber
J_N (z,\phi) = \sum_{j=1}^{\infty} \phi_j I^{\ast}_N(z, j)
\end{equation}
where $\{ \phi_j \}_{j=1}^{\infty} \in \Phi$, with $\Phi$ a set of complex-valued bounded sequences
equipped with uniform norm $|| \phi ||_{\infty} := \sup_j | \phi_j |$ and in order to avoid
notational clutter $N$ replaces the interval length notation $L(T)$ present in the main body of the
paper. The quantity $I^{\ast}_N(z, j)$ denotes the uncorrected tapered local wavelet periodogram
\begin{equation}\nonumber
I^{\ast}_N (z,j) = H_N^{-1} \left| \sum_{t=0}^{N-1} h \left( \frac{t}{N} \right) X_{ t - N/2  +1, T}
\psi_{j, [zT]}(t) \right|^2,
\end{equation}
with $h: [0,1] \rightarrow \reals_{+}$ a data taper,
$H_N := \sum_{j=0}^{N-1} h^2 (j/N) \sim N \int_{0}^1 h^2(x) \, dx$ the normalizing factor and
$h(\cdotp)$ is assumed symmetric and with a bounded second derivative.

As in \cite{dahlhaus98:on} we approximate $J_N(z,\phi)$ by the corresponding statistics of a
stationary
process with the same local corresponding statistics at $t = zT$, $z$ fixed. Let
\begin{equation}\nonumber
J^Y_N (\phi) = \sum_{j=1}^{\infty} \phi_j I^{\ast,Y}_{N} (j)
\end{equation}
where
\begin{equation}\nonumber
I^{\ast,Y}_{N} (j) := H_N^{-1} \left| \sum_{s=0}^{N-1} h \left( \frac{s}{N} \right) Y_{ [zT] -
N/2  +1+s, T} \psi_{j, [zT]}(s) \right|^2
\end{equation}
is the wavelet periodogram on the segment $[zT]-N/2+1, \ldots, [zT]+N/2$ of the {\em stationary}
process
\begin{equation}\nonumber
Y_s = \sum_{j=1}^{\infty} W_j (z)  \sum_{k=-\infty}^\infty  \psi_{j, k}(s) \xi_{j, k}.
\end{equation}

{\em Note:}
The next section uses sequences of bounded variation The total variation of a sequence $\{ \phi_j
\}_{j=1}^\infty$ is
defined by $\TV ( \{ \phi_j \}) = \sum_{j=1}^\infty |\phi_{j+1} - \phi_j|$ and the space of all
sequences of
finite total variation is denoted by $\bv$, see~\cite{DS58} for example.

From equations \eqref{consist1}  and \eqref{consist2}, we obtain $J_N(z,\phi) - \E\left(J_N(z,\phi)\right)=J^Y_N(\phi) - \E\left(J^Y_N (\phi)\right)+o_p(N^{-1/2})$ and using equation
\eqref{eq:approx1} it follows that
\beq\nonumber
J_N(z,\phi) =J^Y_N(\phi)+\calO\left( N^{-1} \right)+o_p(N^{-1/2})
\eeq
which reveals the approximation we make and should be compared to equation (4.4) in
\cite{dahlhaus98:on}, where a term $\calO\left( \frac{N}{T} \right)$ appears instead of $\calO\left(
N^{-1} \right)$.

Using the uncorrected tapered local periodogram expression in equation~\eqref{eq:uncortapper} and
the LSW definition in equation~\eqref{eq:lswdef}, by rearranging formulae we can write the
integrated wavelet periodogram:
\begin{equation}\nonumber
J_N (z,\phi) = \sum_{\ell=1}^\infty \sum_{k=-\infty}^\infty \sum_{m=1}^\infty \sum_{n=-\infty}^\infty
\hat{d}_{N, z} (\ell, k, m, n)
	\xi_{\ell, k} \xi_{m, n},
\end{equation}
where
\begin{equation}\nonumber
\hat{d}_{N, z} (\ell, k, m, n) = H_N^{-1} w^0_{\ell, k} w^0_{m, n} \sum_{j=1}^{\infty} \phi_j
i_{N,z}(j, \ell, k) i_{N,z}(j, m, n)
\end{equation}
and from~\eqref{eq:iNzalt}
\begin{displaymath}
i_{N,z} (j, \ell, k) = \sum_{t=0}^{N-1} h\left( \frac{t}{N} \right) \psi_{j, [zT]-t} \psi_{\ell, k-
[zT]-t-1+N/2}.
\end{displaymath}

Using the properties of the $\{\xi_{\ell,k}\}_{\ell,k}$ field, we obtain
\begin{align*}
\E \{ J_N(z,\phi) \}&= \sum_{\ell=1}^\infty \sum_{k=-\infty}^\infty  \sum_{m=1}^\infty
\sum_{n=-\infty}^\infty \hat{d}_{N, z} (\ell, k, m, n) \delta_{l,m} \delta_{k,n}\\
&= \sum_{\ell=1}^\infty \sum_{k=-\infty}^\infty \hat{d}_{N, z} (\ell, k,\ell, k) \\
\begin{split}&= \sum_{\ell=1}^\infty \sum_{k=-\infty}^\infty H_N^{-1}\left[\left\{ w^0_{\ell,
k}-W_\ell(k/T)\right\}+\left\{ W_\ell(k/T)-W_\ell(z)\right\}\right.\\
& \hspace{3cm} \left. +W_\ell(z) \right]^2  \sum_{j=1}^{\infty}\phi_j \av i_{N,z}(j, \ell, k) \av^2\end{split}\\
&=  \underbrace{
\sum_{\ell=1}^\infty \sum_{k=-\infty}^\infty H_N^{-1}W_\ell^2(z) \sum_{j=1}^{\infty}\phi_j \av
i_{N,z}(j, \ell, k) \av^2
}_ {\E(J_N^Y(\phi))} +\mbox { cross terms},
\end{align*}
where the assumption $\sum_{j} W_j(z)^2 2^{2j}<\infty$ (at any rescaled time $z$) ensures that $\E(J_N^Y(\phi))$ is finite.

Therefore, using the LSW property that $\sup_k \left| w^0_{\ell, k; T} - W_\ell(k/T) \right| \leq
C_\ell / T$, leads to
\beq \nonumber
\av \E \{ J_N(z,\phi)\} -\E \{ J^Y_N(\phi) \}   \av \leq \mbox{ sum of the modulus of the cross terms,}
\eeq
all upper bounded by terms of the form
\begin{align*}
A&=\frac{\Av \phi \Av}{H_N}\sum_{\ell=1}^\infty \sum_{k=-\infty}^\infty \frac{C_\ell^2}{T^2}
\sum_{j=1}^{\infty} \av i_{N,z}(j, \ell, k) \av^2,\\
B&= \frac{\Av \phi \Av}{H_N} \sum_{\ell=1}^\infty \sum_{k=-\infty}^\infty
\left[ \frac{C_\ell}{T} \left\{
W_\ell(k/T)- W_\ell(z)\right\} \sum_{j=1}^{\infty}\av i_{N,z}(j, \ell, k) \av^2\right],\\
C&= \frac{\Av \phi \Av}{H_N} \sum_{\ell=1}^\infty \sum_{k=-\infty}^\infty \left[\left\{
W_\ell(k/T)- W_\ell(z)\right\}^2 \sum_{j=1}^{\infty}\av i_{N,z}(j, \ell, k) \av^2\right],\\
D&= \frac{\Av \phi \Av}{H_N} \sum_{\ell=1}^\infty \sum_{k=-\infty}^\infty \left[ \left\{
W_\ell(k/T)- W_\ell(z)\right\} W_\ell(z) \sum_{j=1}^{\infty}\av i_{N,z}(j, \ell, k) \av^2 \right].
\end{align*}

In order to further bound these quantities, we use Lemma~\ref{lem:order2l} that proves
\beq\nonumber
\sum_{k=-\infty}^\infty \sum_{j=1}^{\infty} \av i_{N,z}(j, \ell, k) \av^2 =
\calO(2^{2\ell}),\label{eq:cumul1}
\eeq
hence we obtain
\begin{align*}
A=\frac{\Av \phi \Av}{H_N}\sum_{\ell=1}^\infty \sum_{k=-\infty}^\infty \frac{C_\ell^2}{T^2}
\sum_{j=1}^{\infty} \av i_{N,z}(j, \ell, k) \av^2
&=\frac{\Av \phi \Av}{H_N} \sum_{\ell=1}^\infty \frac{C_\ell^2}{T^2} \sum_{k=-\infty}^\infty
\sum_{j=1}^{\infty} \av i_{N,z}(j, \ell, k) \av^2\\
&\leq \frac{\Av \phi \Av}{H_N} \frac{1}{T^2}\sum_{\ell=1}^\infty C_\ell^2
2^{2\ell}=\calO(N^{-1}T^{-2}),
\end{align*}
where we have used $\sum_{\ell=1}^\infty C_\ell^2 2^{2\ell}<\infty$.

\begin{align*}
B&=\frac{\Av \phi \Av}{H_N}\sum_{\ell=1}^\infty \sum_{k=-\infty}^\infty \left[ \frac{C_\ell}{T} \left\{
W_\ell(k/T)- W_\ell(z)\right\} \sum_{j=1}^{\infty} \av i_{N,z}(j, \ell, k) \av^2 \right]\\
&\leq \frac{\Av \phi \Av}{H_N} \sum_{\ell=1}^\infty \sum_{k=-\infty}^\infty  \sum_{j=1}^{\infty} \frac{C_\ell}{T} L_\ell \left| \frac{k-[zT]}{T}
\right|  \av i_{N,z}(j, \ell, k) \av^2\\
&\leq \frac{\Av \phi \Av}{H_N} \frac{1}{T}\sum_{\ell=1}^\infty C_\ell L_\ell
2^{2\ell}=\calO(N^{-1}T^{-1}),
\end{align*}
using the  Lipschitz continuity of $\{ W_j \}_j$, $ T^{-1} \av k-[zT] \av \in (0,1)$
and the H\"{o}lder inequality
$$\sum_{\ell=1}^\infty C_\ell L_\ell 2^{2\ell}<\left(\sum_{\ell=1}^\infty C_\ell^2
2^{2\ell}\right)^{1/2} \left(\sum_{\ell=1}^\infty L_\ell^2 2^{2\ell}\right)^{1/2}<\infty$$ coupled
with the assumptions in the theorem.

We now need to bound
\begin{align*}
C&=\frac{\Av \phi \Av}{H_N}\sum_{\ell=1}^\infty \sum_{k=-\infty}^\infty \left[ \left\{ W_\ell(k/T)-
W_\ell(z)\right\}^2  \sum_{j=1}^{\infty} \av i_{N,z}(j, \ell, k) \av^2\right]\\
&\leq \frac{\Av \phi \Av}{H_N} \sum_{\ell=1}^\infty \sum_{k=-\infty}^\infty  \sum_{j=1}^{\infty}
L_\ell^2 \frac{\av(k-[zT])^2 \av}{T^2}\av i_{N,z}(j, \ell, k) \av^2. 
\end{align*}
As $\sum_\ell L_\ell^2 2^{2\ell} < \infty$ and as $\av \frac{k-[zT]}{T} \av^2 \in (0,1)$, we obtain
that $C=\calO(N^{-1})$.

The term
\begin{align*}
D&=\frac{\Av \phi \Av}{H_N}\sum_{\ell=1}^\infty \sum_{k=-\infty}^\infty
\left[
\left\{W_\ell(k/T)- W_\ell(z)\right\} W_\ell(z)  \sum_{j=1}^{\infty} \av i_{N,z}(j, \ell, k)\av^2 \right]\\
&\leq \frac{\Av \phi \Av}{H_N} \sum_{\ell=1}^\infty \sum_{k=-\infty}^\infty  \sum_{j=1}^{\infty}
L_\ell W_\ell(z) \frac{\av(k-[zT]) \av}{T}
 \av i_{N,z}(j, \ell, k) \av^2
\end{align*}
can be bounded using H\"{o}lder's inequality
$$\sum_{\ell=1}^\infty W_\ell(z) L_\ell 2^{2\ell}<\left(\sum_{\ell=1}^\infty W_\ell^2(z)
2^{2\ell}\right)^{1/2} \left(\sum_{\ell=1}^\infty L_\ell^2 2^{2\ell}\right)^{1/2}<\infty$$ based on
the assumptions in the theorem and recalling that we assumed $\{Y_t\}$ to be stationary. Hence
$D=\calO(N^{-1})$.


This completes the proof of $\E\left\{ J_N(z,\phi) \right\} = \E\left\{ J^Y_N (\phi)\right\} + \calO \left(
N^{-1} \right)$.

Now let us establish consistency and its rate. Start by considering
\begin{align*}
\begin{split}\Var \{ J_N(z,\phi) \} &=\E\left\{ J_N(z,\phi)-\E(J_N(z,\phi) \right\}^2\\
& =  \sum_{\ell=1}^\infty
\sum_{k=-\infty}^\infty \sum_{m=1}^\infty \sum_{n=-\infty}^\infty  \sum_{\ell^\prime=1}^\infty
\sum_{k^\prime=-\infty}^\infty \sum_{m^\prime=1}^\infty \sum_{n^\prime=-\infty}^\infty\\
 &\quad \hat{d}_{N, z} (\ell, k, m, n)  \hat{d}_{N, z} (\ell^\prime, k^\prime, m^\prime, n^\prime)
\Cov
(\xi_{\ell, k} \xi_{m, n}, \xi_{\ell^\prime, k^\prime} \xi_{m^\prime, n^\prime}).\end{split}
\end{align*}

Using Isserlis, we can decompose
\begin{align*}
\Cov (\xi_{\ell, k} \xi_{m, n}, \xi_{\ell^\prime, k^\prime} \xi_{m^\prime, n^\prime})&=
\E(\xi_{\ell,
k} \xi_{m, n} \xi_{\ell^\prime, k^\prime} \xi_{m^\prime, n^\prime})\\
& \hspace{2cm}  - \E(\xi_{\ell, k} \xi_{m, n})
\E(
\xi_{\ell^\prime, k^\prime} \xi_{m^\prime, n^\prime})\\
\begin{split}&=\E(\xi_{\ell, k} \xi_{m, n})\E( \xi_{\ell^\prime, k^\prime} \xi_{m^\prime,
n^\prime})\\
& \hspace{2cm}+
\E(\xi_{\ell,
k} \xi_{\ell^\prime, k^\prime})\E( \xi_{m,n} \xi_{m^\prime, n^\prime})\\
&\quad + \E(\xi_{\ell, k} \xi_{m^\prime, n^\prime})\E( \xi_{\ell^\prime, k^\prime} \xi_{m,
n})-\E(\xi_{\ell, k} \xi_{m, n}) \E( \xi_{\ell^\prime, k^\prime} \xi_{m^\prime,
n^\prime})\end{split}\\
&= \delta_{\ell,\ell^\prime}\delta_{k,k^\prime}\delta_{m,m^\prime}\delta_{n,n^\prime}+
\delta_{\ell,m^\prime}\delta_{k,n^\prime}\delta_{m,\ell}\delta_{n,k^\prime},
\end{align*}
hence
\begin{align*}
\Var \{ J_N(z,\phi)\} &= \sum_{\ell=1}^\infty \sum_{k=-\infty}^\infty \sum_{m=1}^\infty
\sum_{n=-\infty}^\infty  \hat{d}^2_{N, z} (\ell, k, m, n)\\
& \hspace{2cm} +\sum_{\ell=1}^\infty \sum_{k=-\infty}^\infty \sum_{m=1}^\infty \sum_{n=-\infty}^\infty  \hat{d}_{N,
z} (\ell, k, m, n)   \hat{d}_{N, z} ( m, n,\ell, k) \\
&=2 \sum_{\ell=1}^\infty \sum_{k=-\infty}^\infty \sum_{m=1}^\infty \sum_{n=-\infty}^\infty
\hat{d}^2_{N, z} (\ell, k, m, n),
\end{align*}
as $\hat{d}^2_{N, z} (m, n, \ell, k)=\hat{d}^2_{N, z} (\ell, k, m, n)$.
Hence, we seek to bound
\beq\label{eq:sqterms}
\sum_{\ell=1}^\infty \sum_{k=-\infty}^\infty \sum_{m=1}^\infty \sum_{n=-\infty}^\infty
\hat{d}^2_{N, z} (\ell, k, m, n).
\eeq

Let us now expand the above
\begin{align*} \label{eq:indivterm}
\frac{1}{2}\Var\{ J_N(z,\phi)\} &=
\sum_{\ell=1}^\infty \sum_{k=-\infty}^\infty \sum_{m=1}^\infty \sum_{n=-\infty}^\infty
\left[\left\{ w^0_{\ell, k}-W_\ell(k/T)\right\}+\left\{ W_\ell(k/T)-W_\ell(z)\right\}\right. \\
& \left. +W_\ell(z) \right]^2\\
&\times \left[\left\{ w^0_{m,n}-W_m(n/T)\right\}+\left\{ W_m(n/T) -W_m(z)\right\} +W_m(z) \right]^2\\
& \times \left\{ \sum_{j=1}^{\infty} \phi_j \av i_{N,z}(j, \ell, k)  i_{N,z}(j,m,n) \av \right\}^2.
\end{align*}

Using an inequality of the type $(a+b+c)^2 \leq 3(a^2+b^2+c^2)$, the above quantity is upper bounded
by a linear combination of a finite number of terms, of the following types

\begin{align*}
AA&=\frac{\Av \phi^2 \Av}{H_N^2}\sum_{\ell=1}^\infty \sum_{k=-\infty}^\infty \sum_{m=1}^\infty
\sum_{n=-\infty}^\infty \frac{C_\ell^2}{T^2} \frac{C_m^2}{T^2} \left\{ \sum_{j=1}^{\infty}  \av
i_{N,z}(j, \ell, k)  i_{N,z}(j,m,n) \av \right\}^2,\\
BB&= \frac{\Av \phi^2 \Av}{H_N^2}\sum_{\ell=1}^\infty \sum_{k=-\infty}^\infty \sum_{m=1}^\infty
\sum_{n=-\infty}^\infty \frac{C_\ell^2}{T^2} \left\{ W_m(n/T)- W_m(z)\right\}^2\\
&\hspace{4cm} \times \left\{
\sum_{j=1}^{\infty}  \av i_{N,z}(j, \ell, k)  i_{N,z}(j,m,n) \av \right\}^2,\\
CC&= \frac{\Av \phi^2 \Av}{H_N^2} \sum_{\ell=1}^\infty \sum_{k=-\infty}^\infty \sum_{m=1}^\infty
\sum_{n=-\infty}^\infty \left\{ W_\ell(k/T)- W_\ell(z)\right\} ^2 \left\{ W_m(n/T) -
W_m(z)\right\} ^2\\
&\hspace{4cm} \times \left\{ \sum_{j=1}^{\infty}  \av i_{N,z}(j, \ell, k)  i_{N,z}(j,m,n) \av \right\}^2,\\
DD&= \frac{\Av \phi^2 \Av}{H_N^2} \sum_{\ell=1}^\infty \sum_{k=-\infty}^\infty \sum_{m=1}^\infty
\sum_{n=-\infty}^\infty \left\{ W_\ell(k/T)- W_\ell(z) \right\}^2 W_m^2(z)\\
&\hspace{4cm} \times
\left\{ \sum_{j=1}^{\infty}  \av i_{N,z}(j, \ell, k)  i_{N,z}(j,m,n) \av \right\}^2,\\
EE&= \frac{\Av \phi^2 \Av}{H_N^2} \sum_{\ell=1}^\infty \sum_{k=-\infty}^\infty \sum_{m=1}^\infty
\sum_{n=-\infty}^\infty W_\ell^2(z) W_m^2(z)\\
& \hspace{4cm} \times
\left\{ \sum_{j=1}^{\infty}  \av i_{N,z}(j, \ell, k)  i_{N,z}(j,m,n) \av \right\}^2.
\end{align*}

In order to bound the above quantities, we use Lemma~\ref{lem:orderlpm} which proves that
\beq\label{eq:cumul1.gen}
\sum_{k=-\infty}^\infty \sum_{n=-\infty}^\infty \left\{\sum_{j=1}^{\infty}\av  i_{N,z}(j, \ell, k)
i_{N,z}(j,m,n) \av\right\}^2 = \calO(2^{(\ell+m)}).
\eeq
Using this result we can bound each term in turn, as follows.

The first term
\begin{align*}
AA&=\frac{\Av \phi^2 \Av}{H_N^2} \sum_{\ell=1}^\infty \sum_{k=-\infty}^\infty \sum_{m=1}^\infty
\sum_{n=-\infty}^\infty \frac{C_\ell^2}{T^2} \frac{C_m^2}{T^2} \left\{\sum_{j=1}^{\infty}\av
i_{N,z}(j, \ell, k)  i_{N,z}(j,m,n) \av\right\}^2\\
&=\frac{\Av \phi^2 \Av}{H_N^2} \sum_{\ell=1}^\infty  \sum_{m=1}^\infty \frac{C_\ell^2}{T^2}
\frac{C_m^2}{T^2}
\sum_{k=-\infty}^\infty \sum_{n=-\infty}^\infty \left\{\sum_{j=1}^{\infty}\av  i_{N,z}(j, \ell, k)
i_{N,z}(j,m,n) \av\right\}^2\\
&\leq \frac{\Av \phi^2 \Av}{H_N^2} \frac{1}{T^4}\sum_{\ell=1}^\infty {C_\ell^2} 2^{2\ell}
\sum_{m=1}^\infty {C_m^2} 2^{2m}=\calO(N^{-2}T^{-4})
\end{align*}
where we used $\sum_{\ell=1}^\infty {C_\ell^2} 2^{2\ell}<\infty$.

The term
\begin{align*}
BB&=\frac{\Av \phi^2 \Av}{H_N^2}\sum_{\ell=1}^\infty \sum_{k=-\infty}^\infty \sum_{m=1}^\infty
\sum_{n=-\infty}^\infty \frac{C_\ell^2}{T^2} \left\{ W_m(n/T)- W_m(z)\right)^2 \left\{
\sum_{j=1}^{\infty} \av i_{N,z}(j, \ell, k)  i_{N,z}(j,m,n) \av \right\}^2,\\
&=\frac{\Av \phi^2 \Av}{H_N^2}\frac{1}{T^2}\sum_{\ell=1}^\infty {C_\ell^2} \sum_{m=1}^\infty
\sum_{k=-\infty}^\infty \sum_{n=-\infty}^\infty \left\{ W_m(n/T)- W_m(z)\right\}^2
\left\{\sum_{j=1}^{\infty}\av  i_{N,z}(j, \ell, k)  i_{N,z}(j,m,n) \av\right\}^2\\
&\leq \frac{\Av \phi^2 \Av}{H_N^2}\frac{1}{T^2}\sum_{\ell=1}^\infty {C_\ell^2} \sum_{m=1}^\infty
\sum_{k=-\infty}^\infty \sum_{n=-\infty}^\infty L_m^2 \left| \frac{n}{T}-z \right|^2
\left\{\sum_{j=1}^{\infty}\av  i_{N,z}(j, \ell, k)  i_{N,z}(j,m,n) \av\right\}^2\\
&\leq \frac{\Av \phi^2 \Av}{H_N^2}\frac{1}{T^2}\sum_{\ell=1}^\infty {C_\ell^2}2^{2\ell}
\sum_{m=1}^\infty
L_m^2 2^{2m}=\calO(N^{-2}T^{-2}),
\end{align*}
where we have used the Lipschitz continuity of $W_m$, $\av \frac{n}{T}-z \av^2 \in (0,1)$,
$\sum_\ell C_\ell^2 2^{2\ell}< \infty$ and $\sum_m L_m^2 2^{2m}< \infty$.

Using the same set of arguments, we bound
\begin{align*}
CC&= \frac{\Av \phi^2 \Av}{H_N^2} \sum_{\ell=1}^\infty \sum_{k=-\infty}^\infty \sum_{m=1}^\infty
\sum_{n=-\infty}^\infty \left\{ W_\ell(k/T)- W_\ell(z)\right\}^2 \left\{ W_m(n/T)-
W_m(z)\right\} ^2\\
&\times \left\{ \sum_{j=1}^{\infty} \av i_{N,z}(j, \ell, k)  i_{N,z}(j,m,n) \av \right\}^2,\\
&\leq \frac{\Av \phi^2 \Av}{H_N^2}\sum_{\ell=1}^\infty \sum_{m=1}^\infty
\sum_{k=-\infty}^\infty \sum_{n=-\infty}^\infty L_\ell^2 \left| \frac{k}{T}-z \right|^2 L_m^2 \left|
\frac{n}{T}-z \right|^2 \left\{\sum_{j=1}^{\infty}\av  i_{N,z}(j, \ell, k)  i_{N,z}(j,m,n)
\av\right\}^2\\
&\leq \frac{\Av \phi^2 \Av}{H_N^2} \sum_{\ell=1}^\infty {L_\ell^2}2^{2\ell} \sum_{m=1}^\infty
L_m^2 2^{2m}=\calO(N^{-2}).
\end{align*}

The term
\begin{align*}
DD&= \frac{\Av \phi^2 \Av}{H_N^2} \sum_{\ell=1}^\infty \sum_{k=-\infty}^\infty \sum_{m=1}^\infty
\sum_{n=-\infty}^\infty \left\{ W_\ell(k/T)- W_\ell(z) \right\}^2 W_m^2(z)\\
& \hspace{4cm} \times
\left\{ \sum_{j=1}^{\infty} \av i_{N,z}(j, \ell, k)  i_{N,z}(j,m,n) \av \right\}^2,\\
&\leq \frac{\Av \phi^2 \Av}{H_N^2} \sum_{\ell=1}^\infty \sum_{m=1}^\infty W_m^2(z)
\sum_{k=-\infty}^\infty \sum_{n=-\infty}^\infty L_\ell^2 \left| \frac{k}{T}-z \right|^2\\
& \hspace{4cm} \times
\left\{\sum_{j=1}^{\infty}\av  i_{N,z}(j, \ell, k)  i_{N,z}(j,m,n) \av\right\}^2\\
&\leq \frac{\Av \phi^2 \Av}{H_N^2} \sum_{\ell=1}^\infty {L_\ell^2}2^{2\ell} \sum_{m=1}^\infty
W_m^2(z) 2^{2m}=\calO(N^{-2}),
\end{align*}
as $\av \frac{k}{T}-z \av^2 \in (0,1)$, $\sum_\ell L_\ell^2 2^{2\ell}< \infty$ and $\sum_m W_m^2(z)
2^{2m}<\infty$ at a set $z$ (recall the process $\{Y_t\}$ was assumed stationary).

Similarly,
\begin{align*}
EE&= \frac{\Av \phi^2 \Av}{H_N^2} \sum_{\ell=1}^\infty \sum_{k=-\infty}^\infty \sum_{m=1}^\infty
\sum_{n=-\infty}^\infty W_\ell^2(z) W_m^2(z)\\
& \hspace{4cm} \times
\left\{ \sum_{j=1}^{\infty} \phi_j \av i_{N,z}(j, \ell, k)  i_{N,z}(j,m,n) \av \right\}^2\\
&\leq \frac{\Av \phi^2 \Av}{H_N^2} \sum_{\ell=1}^\infty W_\ell^2(z) \sum_{m=1}^\infty W_m^2(z)
\sum_{k=-\infty}^\infty \sum_{n=-\infty}^\infty\\
& \hspace{4cm} \times \left\{\sum_{j=1}^{\infty}\av  i_{N,z}(j, \ell, k)
i_{N,z}(j,m,n) \av\right\}^2\\
&\leq \frac{\Av \phi^2 \Av}{H_N^2} \sum_{\ell=1}^\infty W_\ell^2(z) 2^{2\ell} \sum_{m=1}^\infty
W_m^2(z) 2^{2m}=\calO(N^{-2}).
\end{align*}

We therefore obtain that
\beq
\Var \{ N^{1/2}J_N(z,\phi) \} \leq K/N, \mbox{\ \ \ for some constant $K$},
\eeq
hence $J_N(z,\phi)-\E \{ J_N(z,\phi) \}=o_P(N^{-1/2})$.
The result for the process $\{Y_t\}$ follows similarly, which concludes the proof of Theorem
\ref{thm:JNphi}.
\end{proof}


\subsection{Proof of Proposition \ref{coro:pwindow}}\label{app:coro:pwindow}
\begin{proof}
  Theorem \ref{thm:JNphi} established the limit properties of the approximation we make for
$J_{L(T)} (z,\phi)$. From~(\ref{consist1})  and~(\ref{consist2}), we obtain
$J_{L(T)}(z,\phi) - \E\left\{ J_{L(T)} (z,\phi)\right \} = J^Y_{L(T)}(\phi) - \E\left\{ J^Y_{L(T)}
(\phi)\right\} +o_p\left\{ {L(T)}^{-1/2}\right\}$ and using equation (\ref{eq:approx1}) it follows
that
\beq\label{eq:LSWapprox}
J_{L(T)}(z,\phi) =J^Y_{L(T)}(\phi)+\calO\left\{ {L(T)}^{-1} \right\}+o_p\left\{ {L(T)}^{-1/2}
\right\}.
\eeq

Now recall we defined our windowed local partial autocorrelation
estimator \\ $\ptildeW{z}{\tau}=\hat{q}_{\left[ z- L(T)/2T, z+ L(T)/ 2T \right]}(\tau)$, hence
$\ptildeW{z}{\tau}=\left(\hat{\Gamma}_{z,\tau}^{-1}\underline{\hat{\gamma}}_{z,\tau}\right)_{\tau}$
where both $\hat{\Gamma}_{z,\tau}$ and $\underline{\hat{\gamma}}_{z,\tau}$ are a matrix,
respectively vector of local tapered covariances $\hat{c}(z,\tau)$.

The elements of the covariance matrix ($\hat{\Gamma}$) and vector ($\hat{\gamma}$) are
$\hat{c}(z,\tau)=\sum_{j} \hat{S}_{j}(z)\Psi_j(\tau)$ and thus can be written as integrated
periodograms $J_{L(T)}$, since $\hat{\mathbf{S}}_{j}(z)= A_J^{-1} {\mathbf I}^{\ast}_{L(T)} (z)$ and
${\mathbf I}^{\ast}_{L(T)}(z) = \left( I^{\ast}_{L(T)} (z, 1), \ldots, I^{\ast}_{L(T)}(z, J)
\right)^T$. Using the result in equation \eqref{eq:LSWapprox}, in the manner of
\cite{dahlhaus98:on}, it follows that $\ptildeW{z}{\tau}$ has the same asymptotic distribution as in
the stationary case.
\end{proof}

\subsection{Additional Results Required For the Proofs from Section \ref{sec:3}}

\subsubsection{Proof of Lemma \ref{lem:omegai}}\label{proof:lem:omegai}
\begin{proof}
Using the substitution $y = x -u$ we can break the integral for $\Omega_i$ into two pieces as
follows:
\begin{equation}
\label{eq:omegaisplit}
\begin{split}
\Omega_i (u) & = \int_{-\infty}^{\infty} 2^{-i/2} \psi (2^{-i} x) \psi (x-u) \, dx\\
&= 2^{-i/2} \int_{-\infty}^{\infty} \psi\{ 2^{-i} (y+u) \} \psi (y) \,  dy\\
&= 2^{-i/2} \int_0^1 \psi\{ 2^{-i} (y+u) \} \psi (y) \,  dy\\
&= 2^{-i/2} \left[  - \int_0^{\frac{1}{2}} \psi\{ 2^{-i} (y+u) \} \, dy + \int_\frac{1}{2}^1 \psi\{
2^{-i} (y+u) \} \, dy \right]\\
&= 2^{-i/2} \{ - I + II \},
\end{split}
\end{equation}
where $I, II$ are the integrals in the final line of~\eqref{eq:omegaisplit}.

Let us consider integral $I$ first, making the substitution $x = 2^{-i} (y+u)$
\begin{equation}\nonumber
I = \int_0^{\frac{1}{2}} \psi\{ 2^{-i} (y+u) \} \, dy = 2^i \int_{2^{-i} u}^{2^{-i}(u +
\frac{1}{2})} \psi(x) \, dx.
\end{equation}
So, integral $I$ is the result of integrating the product of $\psi(x)$ with the moving window
$[ 2^{-i}u, 2^{-i}(u + \frac{1}{2})]$. To derive integral $I$ we first note that $I=0$ if
$2^{-i}u \geq 1$ or $2^{-i}(u + \frac{1}{2}) \leq 0$ which translates into $I=0$ if
$u \geq 2^i$ or $u \leq -\frac{1}{2}$. We break down the remainder of the case
$-\frac{1}{2} < u < 2^i$ into five subregions.

{\em (Ia)} If $2^i - \frac{1}{2} \leq u < 2^i$ then
\begin{equation}\nonumber
I = 2^i \int_{2^{-i} u}^1 \psi(x) \, dx  = 2^i \int_{2^{-i}u}^1 \, dx = 2^i ( 1 - 2^{-i}u) = 2^i -
u.
\end{equation}

{\em (Ib)} If $2^{i-1} \leq u < 2^i - \frac{1}{2}$ then
\begin{equation}\nonumber
I = 2^i \int_{2^{-i}u}^{2^{-i} (u + \frac{1}{2})} \, dx = 2^i ( 2^{-i}u  + 2^{-i-1} - 2^{-i}u ) =
2^{-1}.
\end{equation}

{\em (Ic)} If $2^{i-1} - \frac{1}{2} \leq u \leq 2^{i-1}$ then
\begin{equation}\nonumber
I = 2^i \left\{ \int_{2^{-i}u}^{\frac{1}{2}} (-1) \, dx + \int_{\frac{1}{2}}^{2^{-i}(u +
\frac{1}{2})} \, dx \right\}
	= 2u - 2^i + \tfrac{1}{2}.
\end{equation}

{\em (Id)} If $0 \leq u < 2^{i-1} - \frac{1}{2}$ then
\begin{equation}\nonumber
I = 2^i \int_{2^{-i}u}^{2^{-i}(u + \frac{1}{2})} (-1) \, dx = -\tfrac{1}{2}.
\end{equation}

{\em (Ie)} If $-\frac{1}{2} \leq u < 0$ then
\begin{equation}\nonumber
I = 2^i \int_0^{2^{-i}(u + \frac{1}{2})} (-1) \, dx = -(u + \tfrac{1}{2}).
\end{equation}

For the second integral we note that:
\begin{equation}\nonumber
II(u) = \int_{\frac{1}{2}}^1 \psi \{ 2^{-i} (y + u) \} \, dy = 2^i \int_{2^{-i} (u +
\frac{1}{2})}^{2^{-i} (u+1)} \psi(x)  = I(u+\tfrac{1}{2}).
\end{equation}
In other words, we have already done the work to evaluate $II(u)$.
Hence,
\begin{equation}\nonumber
\begin{cases}
0 & \mbox{for } u < -1,\\
-(u+1) & \mbox{for } -1 \leq u < -\tfrac{1}{2}  \mbox{   {\em (IIa)}},\\
-\tfrac{1}{2} & \mbox{for } -\tfrac{1}{2} \leq u < 2^{i-1} - 1 \mbox{   {\em (IIb)}},\\
2u - 2^i +\tfrac{3}{2} & \mbox{for } 2^{i-1}-1 \leq u < 2^{i-1} - \tfrac{1}{2}  \mbox{   {\em
(IIc)}},\\
\tfrac{1}{2} & \mbox{for } 2^{i-1}-\tfrac{1}{2} \leq u < 2^i - 1 \mbox{   {\em (IId)}},\\
2^i - u - \tfrac{1}{2} & \mbox{for } 2^i - 1 \leq u < 2^i - \tfrac{1}{2} \mbox{  {\em (IIe)}},\\
0 & \mbox{for } 2^i - \tfrac{1}{2} \leq u.
\end{cases}
\end{equation}

Now we need to put the two results together into~\eqref{eq:omegaisplit} and work out the regions of
overlap
in the two sets of intervals. Hence, define the following regions:
\begin{equation}
\label{eq:regionsIII}
\begin{cases}
III_a & = \{ u: u < -1 \},\\
III_b & = II_a = \{ u: -1 \leq u < -\tfrac{1}{2} \},\\
III_c &= I_a \cap II_b = \{ u: -\tfrac{1}{2}  \leq u < 0 \},\\
III_d &= I_b \cap II_b = \{ u: 0 \leq u < 2^{i-1} - 1\},\\
III_e &= I_b \cap II_c = \{ u: 2^{i-1}-1 \leq u < 2^{i-1} - \tfrac{1}{2}\},\\
III_f &= I_c \cap II_d = \{ u: 2^{i-1} - \tfrac{1}{2} \leq u < 2^{i-1}\},\\
III_g &= I_d \cap II_d = \{ u: 2^{i-1} \leq u < 2^i - 1 \},\\
III_h &= I_d \cap II_e = \{ u: 2^i-1 \leq u < 2^i - \tfrac{1}{2} \},\\
III_i &= I_e = \{ u: 2^i-\tfrac{1}{2} \leq u < 2^i \},\\
III_j &= \{ u: 2^i \leq u\}.
\end{cases}
\end{equation}
Putting together the two integrals with these new domains gives the result in~\eqref{eq:omegaihaar}.
\end{proof}

\subsubsection{Proof of Lemma \ref{lem:bothB}}\label{proof:lem:bothB}
\begin{proof}
We bound $T_{{\cal B}{\cal B}}(\ell, m)$ and include the indices $\ell, m$ explicitly.
For $\ell < m$ we have
\begin{align*}
T_{{\cal B}{\cal B}}(\ell, m)&= \sum_{k \in {\cal B}} \sum_{n \in {\cal B}}  \left\{
		\sum_{j=1}^\infty  | i_{N, z} (j, \ell, k) i_{N, z} (j, m ,n)| \right\}^2\\
	\begin{split}
	  &\leq \sum_{k \in {\cal B}, n \in {\cal B}}  \left[
		\sum_{j=1}^{\ell - 1} 2^{-(\ell+j - m - j)/2} \{ \min(N_j,N_\ell)+N_\ell \} \{ \min(N_j,N_m)+N_m \} \right.\\
	&\quad + \sum_{j=\ell}^{m-1}
2^{-(j+\ell)/2} \{ \min(N_j,N_\ell)+N_\ell \} 2^{-(m+j)/2} \{ \min(N_j,N_m)+N_m \} \\
	&\quad + \left. \sum_{j=m}^\infty
2^{-(j+\ell)/2} \{ \min(N_j,N_\ell)+N_\ell \} 2^{-(j+m)/2}
		\{ \min(N_j,N_m)+N_m \} \right]^2,
	\end{split}
\end{align*}
using bound~\eqref{eq:otherkbound}. Hence
\begin{align}
\begin{split}T_{{\cal B}{\cal B}} (\ell, m) &\leq \sum_{k \in {\cal B}} \sum_{n \in {\cal B}}
	\left\{ 2^{-(\ell+m)/2} \sum_{j=1}^{\ell - 1} 2^{-j}\left(2^{2j}+2^{j+m}+2^{j+\ell}+
2^{\ell+m}\right) \right. \nonumber\\
&\quad + \left. 2^{1-(\ell+m)/2}\sum_{j=\ell}^{m-1}
2^{-j}\left(2^{\ell+j}+2^{\ell+m}\right) + 2^{2-(\ell+m)/2}\sum_{j=m}^\infty 2^{-j}2^{\ell+m}
\right\}^2\end{split}\nonumber\\
		\begin{split}&= \sum_{k \in {\cal B}} \sum_{n \in {\cal B}}
	\left\{ 2^{-(\ell+m)/2} \sum_{j=1}^{\ell - 1} 2^{j}+ 2^{-(\ell-m)/2} \sum_{j=1}^{\ell - 1}
2^{0} + 2^{-(m-\ell)/2} \sum_{j=1}^{\ell - 1} 2^{0} \right. \nonumber\\
&\quad + \left. 2^{(\ell+m)/2} \sum_{j=1}^{\ell - 1} 2^{-j}+ 2^{1-(m-\ell)/2}\sum_{j=\ell}^{m-1}
2^{0} + 2^{1+(\ell+m)/2}\sum_{j=\ell}^{m-1} 2^{-j}\right. \\
& \hspace{4cm} \left. + 2^{2+(\ell+m)/2}\sum_{j=m}^\infty 2^{-j}
\right\}^2\end{split}\\
		\begin{split}&= \sum_{k \in {\cal B}} \sum_{n \in {\cal B}}
	\left\{ 2^{-(\ell+m)/2}\left(2^\ell-2\right)+ 2^{-(\ell-m)/2}\left(\ell-1\right)\right.\\
	&\quad   + 2^{-(m-\ell)/2}\left(\ell-1\right) +
2^{(\ell+m)/2}\left(1-2^{-(\ell-1)}\right)\nonumber\\
        &\quad + \left.2^{1-(m-\ell)/2}\left(m-\ell\right) +
2^{2+(\ell+m)/2}\left(2^{-\ell}-2^{-m}\right)+
2^{2+(\ell+m)/2}2^{1-m}\right\}^2\end{split}\nonumber\\
		\begin{split}&= \sum_{k \in {\cal B}} \sum_{n \in {\cal B}}
	\left\{ \left(\ell-1+1+8+2m-2\ell-4\right)2^{-(m-\ell)/2}- 2^{1-(\ell+m)/2}\right.\nonumber\\
	&\quad + \left.\left(l-1-2+4\right)2^{-(\ell-m)/2} +
2^{(\ell+m)/2}\right\}^2\end{split}\nonumber\\
		&= \sum_{k \in {\cal B}} \sum_{n \in {\cal B}}
	\left\{ \left(4-\ell+2m\right)2^{-(m-\ell)/2}- 2^{1-(\ell+m)/2}\right. \nonumber \\
	& \left. \quad + \left(l+1\right)2^{-(\ell-m)/2} + 2^{(\ell+m)/2}\right\}^2\nonumber\\
		&= \left(N_\ell -1\right)\left(N_m -1\right)
	\left\{ \left(4-\ell+2m\right)2^{-(m-\ell)/2}- 2^{1-(\ell+m)/2}\right.\\
	& \left. \quad +
\left(l+1\right)2^{-(\ell-m)/2} + 2^{(\ell+m)/2}\right\}^2\nonumber\\
		&= \left(N_\ell -1\right)\left(N_m -1\right)
	\left\{ \frac{2^{\ell+m}-2}{2^{(\ell+m)/2}}+
\frac{\left(\ell+1\right)2^{-(\ell-m)}+\left(4-\ell+2m\right)}{2^{-(\ell-m)/2}}\right\}^2\nonumber\\
		\begin{split}&= \left(N_\ell -1\right)\left(N_m -1\right)
	\left\{ \frac{2^{2(\ell+m)}-2^{1+\ell+m}+4}{2^{\ell+m}}\right.\\
	&\quad +\frac{\left(\ell+1\right)^2 2^{-2(\ell-m)}+ \left(4-\ell+2m\right)\left(l+1\right)2^{
1-(\ell-m) } +\left(4-\ell+2m\right)^2}{2^{-(\ell-m)}}\\
	&\quad +\left.\frac{\left(2^{\ell+m} -2\right)\left(\left(\ell+1\right)2^{-(\ell-m)}+
\left(4-\ell+2m\right)\right)}{2^m} \right\}\end{split}\label{eq:TBBstage}\\
		\begin{split}&= \left(2^\ell -1\right)\left(2^m -1\right)
	\left\{ 2^{\ell+m}-2+2^{2-(\ell+m)} +\left(\ell+1\right)^2 2^{-(\ell-m)}\right.\\
	&\quad  \left. +
2\left(4-\ell+2m\right)\left(l+1\right) + \left(4-\ell+2m\right)^2 2^{\ell-m}+\left(\ell+1\right)2^m\right.\\
& \quad \left. +
\left(4-\ell+2m\right)2^\ell- 2\left(\ell+1\right)2^{-\ell} -
2\left(4-\ell+2m\right)2^{-m}\right\}\end{split}\nonumber\\
		&= \mathcal{O} \{ 2^{2(\ell+m)} \} \nonumber
\end{align}
Due to symmetry, if $\ell \geq m$ then we have~\eqref{eq:TBBstage} with the indices $m$ and $\ell$
interchanged.
\end{proof}

\subsubsection{Proof of Lemma \ref{lem:cross}}\label{proof:lem:cross}
\begin{proof} We bound $T_{\not{\cal B}{\cal B}}(\ell, m)$ and include the indices $\ell, m$
explicitly.
For $\ell < m$ we have
\begin{align}
T_{\not{\cal B}{\cal B}}(\ell, m) &=
	\sum_{k \not\in {\cal B}} \sum_{n \in {\cal B}}  \left\{
		\sum_{j=1}^\infty  | i_{N, z} (j, \ell, k) i_{N, z} (j, m ,n)| \right\}^2\nonumber\\
		&= \sum_{p, j=1}^\infty U( j, p, m) V(j, p, \ell),
		\label{eq:sumwithk}
\end{align}
where
\begin{equation}
U(j, p, m) = \sum_{n \in {\cal B}} | i_{N,z} (j, m, n) i_{N,z} (p, m, n)|,\nonumber
\end{equation}
and
\begin{equation}
V(j, p, \ell) = \sum_{k \not\in {\cal B}} | i_{N, z} (j, \ell, k)  i_{N,z}(p, \ell,
k)|.\label{eqn:V}
\end{equation}
Now, for $n \in {\cal B}$ we can apply inequality~\eqref{eq:otherkbound} to obtain
\begin{align*}
U(j, p, m) &\leq \sum_{n \in {\cal B}}  2^{-(j+m)/2}\left(\min(N_j,N_m)+N_m\right)
2^{-(p+m)/2}\left(\min(N_p,N_m)+N_m\right)\\
 &= \left(N_m-1\right)2^{-(j+m)/2}\left(2^{\min(j,m)}+2^m\right)
2^{-(p+m)/2}\left(2^{\min(p,m)}+2^m\right)\\
 &= \left(N_m-1\right)2^{-m}2^{-(j+p)/2}\left(2^{\min(j,m)}+2^m\right)
\left(2^{\min(p,m)}+2^m\right)\\
 &=
\left(N_m-1\right)2^{-m}2^{-(j+p)/2} \\
& \quad \times \left\{ 2^{\min(j,m)+\min(p,m)}+2^{m+\min(j,m)}+ 2^{m+\min(p,m)}
+2^{2m}\right\}.\\
\end{align*}
Now let us examine the sum over $k$ in~\eqref{eq:sumwithk} using the $\Psi$ bound for $i_{N,z}$
from~\eqref{eq:maineq}.
\begin{align*}
V(j, p, \ell) &\leq
	\sum_{k \not\in {\cal B}} | \Psi_{j, \ell} ( k - 2[zT] +N/2 -1)  \Psi_{p, \ell} (k - 2[zT] + N/2
- 1) |\\
	&\leq \sum_{k} | \Psi_{j, \ell} ( k - 2[zT] +N/2 -1)  \Psi_{p, \ell} (k - 2[zT] + N/2 - 1)|\\
	&= \sum_q | \Psi_{j, \ell} (q) | | \Psi_{p, \ell} (q) | = B_{\ell} (j, p),
\end{align*}
where $B_{\ell} (j, p)$ is the fourth-order cross-correlation wavelet absolute value product of order $r=0$, defined as
$B^{(r)}_\ell (j, i) = \sum_{p=-\infty}^\infty |p|^r | \Psi_{j, \ell}(p) \Psi_{i, \ell} (p)|$
for $r=0, 1$ and scales $\ell, j, i \in \nats$.

Hence
\begin{align}
\begin{split}T_{\not{\cal B}{\cal B}}(\ell,m) &\leq (N_m-1)2^{-m}\left[ \sum_{j=1}^m\sum_{p=1}^m
 2^{-(j+p)/2}\left(2^{j+p}+2^{m+j}+2^{m+p}+2^{2m}\right)B_{\ell}(j,p) \right.\nonumber\\
&\quad + \sum_{j=1}^m\sum_{p=m+1}^{\infty}
2^{-(j+p)/2}\left(2^{j+m}+2^{m+j}+2^{2m}+2^{2m}\right)B_{\ell}(j,p) \nonumber\\
&\quad + \sum_{j=m+1}^{\infty}\sum_{p=1}^m
2^{-(j+p)/2}\left(2^{m+p}+2^{2m}+2^{m+p}+2^{2m}\right)B_{\ell}(j,p) \nonumber\\
&\quad + \left.\sum_{j=m+1}^{\infty}\sum_{p=m+1}^{\infty}
2^{-(j+p)/2}\left(2^{2m}+2^{2m}+2^{2m}+2^{2m}\right)B_{\ell}(j,p)\right]\end{split}\nonumber
\end{align}
Using $B_\ell (j,p) \leq K  2^{-(j+p)/2} 2^{2\ell}$ (see Proposition \ref{prop:B}), we obtain

\begin{align*}
\begin{split}T_{\not{\cal B}{\cal B}}(\ell,m) &\leq K2^{2\ell}(N_m-1)2^{-m} \left[
\sum_{j=1}^m\sum_{p=1}^m
 2^{-(j+p)/2}\left(2^{(j+p)/2}+2^{m+(j-p)/2}\right. \right. \\
 & \quad  \left. \left. +2^{m+(p-j)/2}+2^{2m-(j+p)/2}\right)\right.\nonumber\\
&\quad + \sum_{j=1}^m\sum_{p=m+1}^{\infty}
2^{-(j+p)/2}\left(2^{m+(j-p)/2}+2^{m+(j-p)/2}+2^{1+2m-(j+p)/2}\right)\nonumber\\
&\quad + \sum_{j=m+1}^{\infty}\sum_{p=1}^m
2^{-(j+p)/2}\left(2^{m+(p-j)/2}+2^{1+2m-(j+p)/2}+2^{m+(p-j)/2}\right)\nonumber\\
&\quad + \left.\sum_{j=m+1}^{\infty}\sum_{p=m+1}^{\infty}
2^{-(j+p)/2}2^{2+2m-(j+p)/2}\right]\nonumber\end{split}\\
\begin{split}&= K2^{2\ell}(N_m-1)2^{-m} \left[ \sum_{j=1}^m\sum_{p=1}^m
 \left(2^0+2^{m-p}+2^{m-j}+2^{2m-j-p}\right)\right.\nonumber\\
&\quad + \sum_{j=1}^m\sum_{p=m+1}^{\infty} \left(2^{1+m-p}+2^{1+2m-j-p}\right)
+ \sum_{j=m+1}^{\infty}\sum_{p=1}^m \left(2^{1+m-j}+2^{1+2m-j-p}\right)\nonumber\\
&\quad + \left.\sum_{j=m+1}^{\infty}\sum_{p=m+1}^{\infty}
2^{2+2m-j-p}\right]\end{split}\nonumber\\
\begin{split}&= K2^{2\ell}\left(N_m-1\right)2^{-m}\left[
m^2+2m\left(2^m-1\right)+\left(2^{2m}-2^{m+1}+1\right)
+2m\right.\\
&\quad \left.+2\left(2^m-1\right)+2m+2\left(2^m-1\right)+2^2\right] \end{split}\\
&= K2^{2\ell}\left(N_m-1\right)2^{-m}\left( m+1+2^m \right)^2\\
&= \mathcal{O}\left(2^{2(\ell+m)}\right).
\end{align*}

By similar calculations the other cross term $T_{{\cal B}{\not\cal B}}(\ell, m) = {\cal O} (
2^{2(\ell+m)})$.
\end{proof}

\subsubsection{Proof of Lemma \ref{lem:PsiBound}}\label{proof:lem:PsiBound}
\begin{proof}
Using the definition of $V$ from \eqref{eqn:V},
\begin{align*}
T_{\not{{\cal B}}\not{{\cal B}}}(\ell, m) &= \sum_{k \not\in {\cal B}} \sum_{n \not\in {\cal B}}
\left\{\sum_{j=1}^\infty|i_{N,z}(j,\ell,k)i_{N,z}(j,m,n)|\right\}^2 \nonumber\\
&= \sum_{p, j=1}^\infty V(j,p,m) V(j,p,\ell) \nonumber\\
&\leq \sum_{p, j=1}^\infty B_m(j,p) B_\ell(j,p) \nonumber\\
&\leq K^2 \sum_{p, j=1}^\infty 2^{2(\ell+m)}2^{-(j+p)}\nonumber\\
&= K' 2^{2(\ell+m)} \mbox{ for some constant } K'\nonumber\\
&= \mathcal{O}\left \{ 2^{2(\ell+m)}\right\}.
\end{align*}
\end{proof}

\section{Proof of Proposition \ref{prop:B}}\label{proof:prob:B}
\begin{proof}
%
%
\underline{Part A:} for $i, j > \ell$.
We can work out the exact formula for $B_{\ell} (j, j)$ by direct application of the
formula~\eqref{eq:HaarXCorr}
for
$\Psi_{j, \ell} (p)$ for $j > \ell$.
\begin{align}
B^{(0)}_\ell (j, j) &=   \sum_p \Psi^2_{j, \ell} (p)\nonumber\\
&= 2^{-j+\ell} \left\{ \sum_{p = -2^\ell}^{-2^{\ell-1} -1} (2^{-\ell} p + 1)^2 +
\sum_{p=-2^{\ell-1}}^{-1} (2^{-\ell}p)^2 \right.\nonumber\\
	&+
	2^{-2\ell} \sum_{p=2^{j-1} - 2^\ell}^{2^{j-1} - 2^{\ell-1}-1}  (2p - 2^j +
2^{\ell+1})^2\nonumber\\
	&+
	2^{-2\ell} \sum_{p=2^{j-1} - 2^{\ell-1}}^{2^{j-1}-1} (2^j - 2p)^2\nonumber\\
	&+
	2^{-2\ell} \sum_{p=2^j - 2^\ell}^{2^j - 2^{\ell-1} - 1} (2^j - p - 2^\ell)^2\nonumber\\
	&+
	\left. 2^{-2\ell} \sum_{p=2^j - 2^{\ell-1}}^{2^j - 1} (p - 2^j)^2\right\}. \label{eq:Bjj}
\end{align}
Hence, after some algebra,
\begin{equation}
\label{eq:trueBljj}
B^{(0)}_\ell (j, j) =   2^{-j + \ell} ( 2^{\ell} + 2^{\ell-1} + 2^{1-\ell} + 2^{-\ell} )/3
= 2^{-j} ( 2^{2\ell-1} +1 ).
\end{equation}

Next, we examine:
\begin{equation}\nonumber
B^{(0)}_\ell(j, i) = B^{(0)}_\ell (j, j+1) = \sum_p |\Psi_{j, \ell} (p) \Psi_{j+1, \ell} (p)|,
\end{equation}
for $i=j+1$.
\begin{table}
\centering
\caption{Ranges of indices for non-zero parts of Haar cross-correlation wavelet $\Psi_{j, \ell}(p)$.
	\label{tab:Psirange}}
\begin{tabular}{c|c|c}
Range & $j$  & $i$\\\hline
I & $-2^\ell \leq p < 0$ & $-2^\ell \leq p < 0$\\
II & $2^{j-1} - 2^\ell \leq p < 2^{j-1}$ & $2^{i-1} - 2^\ell \leq p < 2^{i-1}$\\
III & $2^j - 2^\ell \leq p < 2^j$ & $2^i - 2^\ell \leq p < 2^i$.
\end{tabular}
\end{table}
Examining Table~\ref{tab:Psirange} shows that Range~I for both the $j$ and $j+1$ cross-correlation
wavelets always overlap and  Range II for $i=j+1$ overlaps with Range III for $j$. Hence:
\begin{align}
B^{(0)}_{\ell} (j,  i= j+1) &= \sum_{p \in \mbox{Range I}} |\Psi_{j, \ell} (p) \Psi_{j+1, \ell} (p)| +
	\sum_{p \in \mbox{Range 3}} |\Psi_{j, \ell} (p) \Psi_{j+1, \ell} (p)|\nonumber\\
	&= 2^{-(j-\ell)/2} 2^{-(i-\ell)/2} \left\{ \sum_{p= -2^\ell}^{-2^{\ell-1} - 1} (2^{-\ell} p +
1)^2
		+ \sum_{p = -2^{\ell-1}}^{-1} (2^{-\ell}p)^2 \right. \nonumber\\
	&+ 2^{-2\ell} \sum_{p=2^{j}- 2^\ell}^{2^{j} - 2^{\ell-1} - 1} |(2p - 2^i + 2^{\ell+1} ) (2^j -
p - 2^\ell)|\nonumber\\
	&+ \left. 2^{-2\ell} \sum_{p=2^{j} - 2^{\ell-1}}^{2^{j}-1} |(2^i - 2p)(p-2^j)| \right\}
\nonumber\\
	&= 2^\ell 2^{-j/2} 2^{-(j+1)/2} \times 2^{-2-\ell} (2^{2\ell}+2)\nonumber\\
	&=  2^{-j} 2^{-3/2}   (2^{2\ell-1} +1).
	\end{align}
Finally for $|i-j| > 1$ we examine:
\begin{align}
B^{(0)}_\ell (j, i) &= \sum_p | \Psi_{j, \ell} (p) \Psi_{i, \ell} (p)|\nonumber\\
&=  2^{-(j-\ell)/2} 2^{-(i-\ell)/2}
	\left\{ \sum_{p = -2^\ell}^{-2^{\ell-1} -1} (2^{-\ell} p + 1)^2 + \sum_{p=-2^{\ell-1}}^{-1}
(2^{-\ell}p)^2 \right\}\nonumber\\
	&= 2^{-j/2} 2^{-i/2} (2^{2\ell-1} + 1)/6.
\end{align}

\underline{Part B:} for the case $i, j < \ell$. First, for $j=i$ we have
\begin{equation}\nonumber
B_\ell^{(0)} (j, j) = \sum_p \Psi^2_{j, \ell} (p)= \sum_p \Psi^2_{\ell, j} (-p)
= \sum_p \Psi^2_{\ell, j} (p) = B^{(0)}_{L} (J,J),
\end{equation}
where $L=j$ and $J= \ell$ and we
used $\Psi_{j, \ell} (p) = \Psi_{\ell, j} (-p)$ since we will use the formula for $\Psi_{\ell, j}$
where $\ell > j$ in~\eqref{eq:HaarXCorr}. Since $L = j < \ell = J$ this puts
us into the situation of~\eqref{eq:trueBljj} which gives:
\begin{equation}\nonumber
B^{(0)}_\ell (j, j) = 2^{-J} ( 2^{2L - 1} + 1) = 2^{-\ell} (2^{2j-1} + 1),
\end{equation}
as required. For $i < j < \ell$ we use the formula for $\Psi$ for $j < \ell$ given
by~\eqref{eq:Psijlessell}. Due to the form of $\Psi_{j, \ell}(p)$ for $i < j < \ell$
we can split the sum into three parts corresponding to the the three non-zero
parts of the autocorrelation wavelet given in~\eqref{eq:Psijlessell}. The condition
$i < j < \ell$ is helpful as the interval associated with $i$ nests within that
of $j$, and the $j$ interval nests within that associated with $\ell$. First,
we will deal with the last two terms of~\eqref{eq:Psijlessell} which do not depend on
$\ell$ (``front part'').
\begin{align}
\mbox{front} &= \sum_{p=1}^{2^i} | \Psi_{j, \ell} (p) || \Psi_{i, \ell} (p) |\nonumber\\
&= \sum_{p=1}^{2^{i-1}} | - 2^j p | | -2^i p| + \sum_{p=2^{i-1}+1}^{2^i}
	| -2^{-j} p | | 1 - 2^{-i} p |
\end{align}
Note how the first term in each of these sums is $|-2^{-j} p|$ because
this is the formula for $\Psi_{j, \ell} (p)$ over the interval $[1, 2^i]$ because
this interval is always contained within $[1, 2^{j-1}]$ since $i < j$.
We also ensure that the last absolute value term in
the sum is positive (if it was negative then we'd switch signs of the contents
of the absolute value as $|x| = -x$ if $x$ is negative).
Continuing
\begin{equation}\nonumber
\mbox{front} = 2^{-i-j} \sum_{p=1}^{2^{i-1}} p^2 +
	2^{-j} \sum_{p=2^{i-1}+1}^{2^i} p (1 - 2^{-i}p)
	= 2^{-j + 2i - 3}.
\end{equation}
The middle term is constructed in a similar way except that $\ell$ enters
into the equation. However, the concept that the $j$ interval can only
ever overlap the $i$ interval on its first half still remains. Hence,
\begin{align}
\mbox{middle } &= \sum_{p = -2^{\ell - 1}+1}^{-2^{\ell-1} + 2^i}  | \Psi_{j, \ell} (p) || \Psi_{i,
\ell} (p) |\nonumber\\
&= \sum_{p=-2^{\ell-1}+1}^{-2^{\ell-1} + 2^{i-1}} 2^{-j} (2^\ell + 2p) 2^{-i} (2\ell +
2p)\nonumber\\
	&+
	\sum_{p=-2^{\ell-1} + 2^{i-1} +1}^{-2^{\ell-1} + 2^i} 2^{-j} (2^\ell + 2p)
		2^{-i} (2^{i+1} - 2^\ell - 2p)\nonumber\\
&= 2^{-j-i} \left\{ \sum_{p=-2^{\ell-1}+1}^{-2^{\ell-1} + 2^{i-1}} (2^\ell + 2p)^2\right.
\nonumber\\
	&+\left.  \sum_{p=-2^{\ell-1} + 2^{i-1} +1}^{-2^{\ell-1} + 2^i} (2^\ell + 2p)
		(2^{i+1} - 2^\ell - 2p) \right\}\nonumber\\
&= 2^{-j} 2^{2i-1}.
\end{align}
The back part uses precisely the same part as the middle
\begin{align}
\mbox{back} &= \sum_{p = -2^\ell + 1}^{-2^\ell + 2^{i-1}}
	| -2^{-j} (p + 2^\ell) | | -2^{-i} (p + 2^\ell) |\nonumber\\
	&+ \sum_{p=-2^\ell + 2^{i-1} +1}^{-2^\ell + 2^i}
		| -2^{-j} (p + 2^\ell) | | 2^{-i} (2^\ell + p - 2^i) |\nonumber\\
&= 2^{-j-i} \left\{ \sum_{p = -2^\ell + 1}^{-2^\ell + 2^{i-1}} (p + 2^\ell)^2 +
	\sum_{p=-2^\ell + 2^{i-1} +1}^{-2^\ell + 2^i}
		(p + 2^\ell)(2^i - p - 2^\ell) \right\}\nonumber\\
	&= 2^{-2} 2^{-j} 2^{2i-1}.
\end{align}
Adding the front, middle and back components together and multiplying through
by the constant that appears at the front of~\eqref{eq:Psijlessell} gives:
\begin{align}
B^{(0)}_{\ell}(j, i) &= 2^{-(\ell-j)/2} 2^{-(\ell - i)/2} ( 2^{-j} 2^{-2} 2^{2i-1}
	+ 2^{-j} 2^{2i-1} + 2^{-2} 2^{-j} 2^{2i-1})\nonumber\\
	&= 2^{-\ell} 2^{i/2} 2^{j/2}\times 2^{-j} 2^{2i-1} (2^{-2} + 1 + 2^{-2})\nonumber\\
	&= \frac{3}{2} 2^{-\ell} 2^{-j/2} 2^{5i/2 - 1},
\end{align}
as required.

\underline{Part C:} for the case $i < \ell < j$. When considering the formula for
$B^{(0)}_{\ell}(j,i)$
we use formula~\eqref{eq:HaarXCorr} for $\Psi_{j, \ell}(p)$ because $\ell < j$ and
formula~\eqref{eq:Psijlessell} for $\Psi_{i, \ell}(p)$ because $i < \ell$.

Let us first consider
positive $p$ first. In this case, the only values of $\Psi_{i, \ell}(p)$ which
are nonzero are for $0 < p \leq 2^i$ and the only values of $\Psi_{j, \ell}(p)$
are for $2^{j-1} - 2^\ell \leq p$. So, if $2^i < 2^{j-1} - 2^\ell$ then there is no overlap
between these two cross-correlation wavelets, under what conditions can this
occur (subject to $i < \ell < j$)?
Suppose
$j = \ell + 1$ then $2^{j-1} - 2^\ell = 0$ and $2^i$ is never less than or equal to zero.
Suppose $j=\ell + 2$, then $2^{j-1} - 2^\ell = 2^\ell$ which is always greater than
$2^i$ because we know that $\ell > i$. Hence, we have to consider two cases
(i) $j = \ell + 1$ where the positive parts of the cross-correlation functions
overlap and $j > \ell + 1$ where they do not.

We will now work out this `overlap' contribution where $j = \ell+1$:
\begin{equation}\nonumber
\mbox{`overlap'} = \sum_{p=1}^{2^i} | \Psi_{\ell+1, \ell} (p) ||\Psi_{i, \ell}(p)|.
\end{equation}
The limits of the sum correspond to where the term $\Psi_{i, \ell} (p)$ is
non-zero for $p \geq 0$. Now we consider which parts of
$\Psi_{\ell+1, \ell}(p)$ are relevant to the sum. With $j= \ell+1$ the fifth and
sixth ranges in~\eqref{eq:HaarXCorr} turn into
\begin{gather}
0 \leq \tau < 2^{\ell-1},\label{eq:range1}\\
2^{\ell-1} \leq \tau < 2^\ell, \label{eq:range2}
\end{gather}
respectively. Since $i < \ell$ we have $2^i < 2^\ell$ and so the two ranges
in~\eqref{eq:range1} and \eqref{eq:range2} are the only ones we need to consider
for $\Psi_{\ell+1, \ell} (p)$ for positive $p$. Now the highest that $i$ can be is
$i = \ell - 1$ so that maximum $p$ is $2^i = 2^{\ell-1}$ so, in actuality, it is only
range~\eqref{eq:range1} that is active for any $p > 0$.

Hence, without the constant $2^{-(j-\ell)/2} 2^{-(\ell-i)/2} = 2^{-j/2}2^{i/2}$ at the front
of~\eqref{eq:HaarXCorr} and~\eqref{eq:Psijlessell}, we can write
\begin{align}
\mbox{`overlap'} &= \sum_{p=1}^{2^i} | 2^{-\ell} (2p - 2^{\ell+1} + 2^{\ell+1}) |
	| \Psi_{i, \ell} (p)|\nonumber\\
	&= 2^{-\ell+1} \sum_{p=1}^{2^i} p | \Psi_{i, \ell}(p)|\nonumber\\
	&= 2^{-\ell+1} \left\{ \sum_{p=1}^{2^{i-1}} p | -2^{-i} p | +
		\sum_{p=2^{i-1}+1}^{2^i} p | 1 - 2^{-i} p | \right\}.
\end{align}
Note that $1 - 2^{-i} p$ is positive on the range of $p$ in the second sum. Hence,
\begin{align}
\mbox{`overlap'} &= 2^{-\ell + 1} \left\{ 2^{-i} \sum_{p=1}^{2^{i-1}} p^2 +
	\sum_{p=2^{i-1}+1}^{2^i} p (1 - 2^{-i} p ) \right\}\nonumber\\
	&= 2^{-\ell} 2^{2i} 2^{-2}, \label{eq:overlap}
\end{align}
for $i < \ell$ and $j = \ell+1$.

For $p < 0$, $i < \ell < j$ the cross-correlation wavelet
$\Psi_{j, \ell} (p)$ splits into two parts random each of those two parts
the cross-correlation wavelet $\Psi_{i, \ell}(p)$ splits into two further parts. Hence,
the sum on the $p < 0$ terms (the `back' bit) is given by:
\begin{align}
\mbox{`Back'} &= \sum_{p= -2^\ell +1}^{-2^\ell + 2^i} | \Psi_{j, \ell} (p) || \Psi_{i, \ell}(p)|
+ \sum_{p=-2^{\ell-1} + 1}^{-2^{\ell -1} + 2^i} | \Psi_{j, \ell} (p) || \Psi_{i,
\ell}(p)|\nonumber\\
&= \sum_{p= -2^\ell +1}^{-2^\ell + 2^i} | -(2^{-\ell} p + 1 )| |\Psi_{i, \ell}(p)|
+ \sum_{p=-2^{\ell-1} + 1}^{-2^{\ell -1} + 2^i}  | 2^{-\ell} p | | \Psi_{i, \ell}(p)|
\end{align}
We can now split each of these two sums into the two sets of two ranges
dictated by the domain of $\Psi_{i, \ell} (p)|$ as follows:
\begin{align}
\mbox{`Back'} &= \sum_{p = -2^\ell + 1}^{-2^\ell+ 2^{i-1}} | 1 + 2^{-\ell} p |
	| 2^{-i} (p + 2^\ell)|\nonumber\\
	&+ \sum_{p=-2^\ell + 2^{i-1} + 1}^{-2^\ell + 2^i} | 1 + 2^{-\ell} p |
	| 2^{-i} (2^\ell + p - 2^i) |\nonumber\\
	&+ \sum_{p=-2^{\ell-1} + 1}^{-2^{\ell-1} + 2^{i-1}} | 2^{-\ell} p |
	| 2^{-i} ( 2^\ell + 2p)|\nonumber\\
	&+ \sum_{p=-2^{\ell-1} + 2^{i-1} +1}^{-2^{\ell-1} + 2^i} | 2^{-\ell} p |
	| 2^{-i} (2^{i+1} - 2^\ell - 2p)|\label{eq:Backbit}
\end{align}
For all of the eight components in the four sums in~\eqref{eq:Backbit}
all except the fourth, fifth and seventh are always positive over their respective
sum's range of $p$ values. Hence, we replace those terms in the absolute
values by their negative (e.g.\ $|x| = -x$  for $x<0$) and obtain:
\begin{align}
\mbox{`Back'} &=  2^{-i} \sum_{p = -2^\ell + 1}^{-2^\ell+ 2^{i-1}} ( 1 + 2^{-\ell} p )
	 (p + 2^\ell)\nonumber\\
	&+ 2^{-i} \sum_{p=-2^\ell + 2^{i-1} + 1}^{-2^\ell + 2^i} ( 1 + 2^{-\ell} p )
	(2^i - p - 2^\ell) \nonumber\\
	&+2^{-\ell -i } \sum_{p=-2^{\ell-1} + 1}^{-2^{\ell-1} + 2^{i-1}} (-p)
	( 2^\ell + 2p)\nonumber\\
	&+2^{-\ell -i} \sum_{p=-2^{\ell-1} + 2^{i-1} +1}^{-2^{\ell-1} + 2^i} (-p)
	(2^{i+1} - 2^\ell - 2p)\nonumber\\
	&= 2^{i-3} (2 - 2^{i-\ell}).\label{eq:Backbit2}
\end{align}
Now multiplying by the $2^{-j/2} 2^{i/2}$ which we omitted earlier gives:
\begin{align}
B^{(0)}_{\ell} (j,i) &= 2^{-j/2} 2^{i/2} 2^{i-3} ( 2 - 2^{i-\ell})\nonumber\\
&= 2^{-j/2} 2^{3i/2-3} (2 - 2^{i-\ell}),
\end{align}
as required for the second equation in~\eqref{eq:sandwichB} for
$i < \ell$ and $\ell + 1 < j$.

For the first equation in~\eqref{eq:sandwichB} we have to add $2^{-j/2} 2^{i/2}$
times equation~\eqref{eq:overlap} to obtain:
\begin{align}
B^{(0)}_\ell (j, i) &= 2^{-j/2} 2^{3i/2-3} (2 - 2^{i-\ell}) + 2^{-j/2} 2^{i/2} 2^{-\ell}
	2^{2i} 2^{-2}\nonumber\\
	&= \frac{1}{8} 2^{-(\ell+1)/2} 2^{3i/2} (2^{i-\ell} + 2),
\end{align}
as required.

\underline{Part D:} First consider $j = \ell < i$. Then:
\begin{equation}\nonumber
B^{(0)}_\ell (j, i) = \sum_p | \Psi_{j, \ell} (p) || \Psi_{i, \ell} (p)|
= \sum_p | \Psi_\ell (p) || \Psi_{i, \ell} (p)|,
\end{equation}
where $\Psi_{\ell} (p)$ is the ordinary autocorrelation wavelet from
\cite{nason00:wavelet}.
The domain of $\Psi_{\ell} (p)$ is from $-2^\ell < p 2^\ell$. The range of $\Psi_{\ell} (p)$ and
$\Psi_{i, \ell}(p)$
agrees for $p <0 $ but the two wavelets only overlap for $p>0$ if the lower end of the nonzero range
of
$\Psi_{i, \ell}(p)$, namely $2^{i-1} - 2^\ell$ is smaller than the upper end of the autocorrelation
wavelet
$2^{\ell}$, recalling that $i > \ell$. This only occurs when $i=\ell + 1$ and for $i  \geq \ell + 2$
there is
no overlap for $p > 0$.

For this latter case let use look at the negative range of $p$ values:
\begin{align}
B^{(0)}_{\ell} (\ell, i) &= \sum_{p=-2^{\ell}+1}^{-2^{\ell - 1}} | \Psi_\ell (p) || \Psi_{i, \ell} (p)| +
	\sum_{p= - 2^{\ell - 1}+1}^0 | \Psi_\ell (p) || \Psi_{i, \ell} (p)|\nonumber\\
	&= 2^{-(i-\ell)/2} \left\{  \sum_{p=-2^{\ell}+1}^{-2^{\ell - 1}} | 2^{-\ell} |p| - 1 | |
-(2^{-\ell} p + 1 )|
		\right. \nonumber\\
	&+ \left.\sum_{p= - 2^{\ell - 1}+1}^0   | 1 - 3|p| 2^{-\ell} ||2^{-\ell} p |
\right\}\nonumber\\
	&= 2^{-(i-\ell)/2} \left\{  \sum_{p=-2^{\ell}+1}^{-2^{\ell - 1}} | -2^{-\ell} p - 1 | |
2^{-\ell} p + 1 |
		\right. \nonumber\\
		&+ \left.\sum_{p= - 2^{\ell - 1}+1}^0 |   1 + 3p 2^{-\ell} | (- 2^{-\ell} p
)\right\}\nonumber\\
	&= 2^{-(i-\ell)/2} \left\{  \sum_{p=-2^{\ell}+1}^{-2^{\ell - 1}} ( 2^{-\ell} p + 1 )^2\right.
\nonumber\\
	&+ \left. \sum_{p= - 2^{\ell - 1}+1}^0 |   1 + 3p 2^{-\ell} | (- 2^{-\ell} p
)\right\}\nonumber\\
	&= 2^{-(i-\ell)/2} \{ C2 + C1 \}.
\end{align}
The sum $C2 = \sum_{p=-2^{\ell}+1}^{-2^{\ell - 1}} ( 2^{-\ell} p + 1 )^2 = \tfrac{1}{3} 2^{-3-\ell}
( 2+ 3\cdot 2^\ell + 2^{2\ell})$.
The sum $C1$ can be shown to be approximated by $(2^{3+\ell} - 9\cdot 2^{1-\ell} - 27)/216$ in
Mathematica
with error bounded by
\begin{equation}\nonumber
2 | 1 + 2\cdot 2^{-\ell} (-2^\ell/3 - 1) || -2^{\ell} (-2^{\ell}/3 - 1)| \leq 5\cdot 2^{-\ell}.
\end{equation}
Hence, for $i \geq \ell+2$ we have
\begin{equation}
B^{(0)}_\ell (\ell, i) = \tfrac{17}{27} 2^{\ell - 3}. \label{eq:igtlpo}
\end{equation}

Now, for $i = \ell + 1$ let us examine the contribution for $p > 0$:
\begin{align}
\mbox{`Front'} &= \sum_{p=1}^{2\ell} | \Psi_\ell (p) || \Psi_{\ell+1, \ell} (p)|\nonumber\\
&= \sum_{p=1}^{2^{\ell-1}} | \Psi_\ell (p) || \Psi_{\ell+1, \ell} (p)|
	+
		\sum_{p=2^{\ell-1}+1}^{2^\ell} | \Psi_\ell (p) || \Psi_{\ell+1, \ell} (p)|\nonumber\\
&= 2^{-1/2} \left\{ \sum_{p=1}^{2^{\ell-1}} | 2 - 3|p| 2^{-\ell} | | 2^{-\ell} (2p - 2^{\ell+1} +
2^{\ell+1}) | \right. \nonumber\\
&+\left. \sum_{p=2^{\ell-1}+1}^{2^\ell} | 2^{-\ell} |p|  - 1 | | 2^{-\ell} (2^{\ell+1} - 2p)|
\right\}\nonumber\\
&= 2^{-1/2} \left\{ 2^{-\ell+1} \sum_{p=1}^{2^{\ell-1}}  | 1 - 3 p 2^{-\ell} | p +
	 2^{-\ell} \sum_{p=2^{\ell-1}+1}^{2^\ell} | 2^{-\ell} p - 1 | | 2^{\ell+1} - 2p |
\right\}\nonumber\\
&= (27 - 9\cdot 2^{1-\ell} + 2^{\ell+3} )/108  + 2^{-2-\ell} (2^{2\ell} - 3\cdot 2^{\ell} +
2)/3\nonumber\\
&= \frac{17}{27} 2^{\ell-2}.\label{eq:ieqlpo}
\end{align}
Hence, adding together~\eqref{eq:igtlpo} and~\eqref{eq:ieqlpo}  gives:
\begin{equation}
B^{(0)}_{\ell} (\ell, i) = \frac{17}{9} 2^{\ell-3},
\end{equation}
for $i = \ell+1$ as required.

For the case of $i < \ell$ we use the following bound:
\begin{equation}\nonumber
B^{(0)}_\ell (\ell, i) = \sum_{p} | \Psi_{\ell} (p) | | \Psi_{i, \ell} (p)| = \sum_{p=-2^{\ell} +1}^{2^i}
| \Psi_{\ell} (p) | | \Psi_{i, \ell} (p)|
 \leq  \sum_{p=-2^{\ell} +1}^{2^i}   | \Psi_{i, \ell} (p)|,
\end{equation}
as $|\Psi_\ell (p)|$ takes its maximum value of 1 at $p=0$.
We can work out this sum directly taking care to discover when the argument inside the absolute
value sign is
negative or positive giving:
\begin{align}
2^{-(i-\ell)/2} B^{(0)}_\ell (\ell, i) & \leq 2^{-i}  \sum_{p=-2^\ell +1}^{-2^\ell + 2^{i-1}} (p + 2^\ell)
	+ 2^{-i} \sum_{p=-2^\ell +2^{i-1} + 1}^{-2^\ell + 2^i} (2^i - 2^\ell - p) \nonumber\\
	&+ 2^{-i} \sum_{p=-2^{\ell-1}+1}^{-2^{\ell-1} + 2^{i-1}} (2^\ell + 2p) +
	2^{-i} \sum_{p=-2^{\ell-1}+2^{i-1}+1}^{-2^{\ell-1} +2^{i}} (2^{i+1} -2^\ell -2p) \nonumber\\
	&+ 2^{-i} \sum_{p=1}^{2^{i-1}} p + \sum_{p=2^{i-1}+1}^{2^i} (1 - 2^{-i} p ),
\end{align}
Hence
\begin{equation}
B^{(0)}_\ell (\ell, i) \leq  2^{3i/2} 2^{-\ell/2}.
	\end{equation}
as required.
\end{proof}

\end{document}